\documentclass{amsart}

\usepackage[hidelinks]{hyperref}
\usepackage{mathrsfs}
\usepackage{enumerate}
\usepackage{graphicx,color}
\usepackage{stmaryrd}
\usepackage{amssymb}

\newtheorem{defi}{Definition}[section]
\newtheorem{prop}[defi]{Proposition}
\newtheorem{thm}[defi]{Theorem}
\newtheorem{lem}[defi]{Lemma}
\newtheorem{cor}[defi]{Corollary}
\newtheorem{ex}[defi]{Example}

\numberwithin{equation}{section}

\newcommand{\N}{\mathbb{N}}
\newcommand{\Z}{\mathbb{Z}}
\newcommand{\R}{\mathbb{R}}

\newcommand{\balpha}{\pmb{\alpha}}

\newcommand{\cD}{\mathscr D}
\newcommand{\cF}{\mathscr F}
\newcommand{\cH}{\mathscr H}		% Hausdorff measure
\newcommand{\cK}{\mathscr K}
\newcommand{\cL}{\mathscr L}		% Lebesgue measure
\newcommand{\cP}{\mathscr P}
\newcommand{\cR}{\mathscr R}

\newcommand{\cS}{\mathscr S}

\newcommand{\bF}{\operatorname{\mathbf F}}
\newcommand{\bI}{\operatorname{\mathbf I}}
\newcommand{\bM}{\operatorname{\mathbf M}}
\newcommand{\bN}{\operatorname{\mathbf N}}

\newcommand{\bV}{\operatorname{\mathbf V}}

\newcommand{\boxd}{\text{\rm box}}
\newcommand{\Haus}{\text{\rm Haus}}

\newcommand{\BV}{\operatorname{BV}}

\newcommand{\Lip}{\operatorname{Lip}}
\newcommand{\Lipc}{\Lip_c}

\newcommand{\Hol}{\operatorname{Lip}}

\newcommand{\spt}{\operatorname{spt}}
\newcommand{\diam}{\operatorname{diam}}
\newcommand{\dist}{\operatorname{dist}}

\newcommand{\defl}{\mathrel{\mathop:}=}

\newcommand{\B}{\textbf B}		%closed ball
\newcommand{\oB}{\textbf U}		%open ball

\newcommand{\degr}[3]{\operatorname{deg}\left(#1,#2,#3\right)} % #1 = point, #2 = function, #3 = domain
\newcommand{\curr}[1]{[\![{#1}]\!]}
\newcommand{\id}{{\rm id}}
\newcommand{\sign}{\operatorname{sign}}

\newcommand{\res}{\scalebox{1.7}{$\llcorner$}}

\begin{document}

\title{Functions of bounded fractional variation and fractal currents}
\author{Roger Z\"{u}st}
\address{\scriptsize Mathematical Institute, University of Bern, Alpeneggstrasse 22, 3012 Bern, Switzerland}
\email{\scriptsize roger.zuest@math.unibe.ch}

%\date{\today}


% 49Q15, %  Geometric measure and integration theory, integral and normal currents
% 28A75,  %  Length, area, volume, other geometric measure theory
% 26A16,  % Lipschitz (H۬\"older) classes

\begin{abstract}
Extending the notion of bounded variation, a function $u \in L_c^1(\mathbb R^n)$ is of bounded fractional variation with respect to some exponent $\alpha$ if there is a finite constant $C \geq 0$ such that the estimate
\[
\biggl|\int u(x) \det D(f,g_1,\dots,g_{n-1})_x \, dx\biggr| \leq C\operatorname{Lip}^\alpha(f) \operatorname{Lip}(g_1) \cdots \operatorname{Lip}(g_{n-1})
\]
holds for all Lipschitz functions $f,g_1,\dots,g_{n-1}$ on $\mathbb R^n$. Among such functions are characteristic functions of domains with fractal boundaries and H\"older continuous functions. We characterize functions of bounded fractional variation as a certain subspace of Whitney's flat chains and as multilinear functionals in the setting of Ambrosio-Kirchheim currents. Consequently we discuss extensions to H\"older differential forms, higher integrability, an isoperimetric inequality, a Lusin type property and change of variables. As an application we obtain sharp integrability results for Brouwer degree functions with respect to H\"older maps defined on domains with fractal boundaries.
\end{abstract}

\maketitle

\tableofcontents

\section{Introduction}

\subsection{Functions of bounded fractional variation} The main objects we study here are functions $u \in L_c^1(\R^n)$ for which there is an exponent $\alpha \in [0,1]$ and a finite constant $C \geq 0$ such that
\[
\biggl|\int_{\R^n} u(x) \det D(f,g_1,\dots,g_{n-1})_x\, dx\biggr| \leq C\Lip^\alpha(f)\operatorname{Lip}(g_1)\cdots\operatorname{Lip}(g_{n-1})
\]
holds for all $f,g_1,\dots,g_{n-1} \in \Lip(\R^n)$, where
\[
\Lip^\alpha(g) = \sup_{x \neq y} \frac{|g(x) - g(y)|}{|x-y|^\alpha}
\]
is the usual H\"older seminorm with exponent $\alpha$. The smallest such $C$ is denoted by $\bV^\alpha(u)$ and the resulting subspace of $L_c^1(\R^n)$ is $\BV^\alpha_c(\R^n)$. This extends the classical notion of bounded variation, where $u \in L_c^1(\R^n)$ is in $\BV_c(\R^n)$ if the total variation
\[
\bV(u) \defl \sup \biggl\{\int_{\R^n} u(x) \operatorname{div} \varphi(x) \, dx \, : \, \varphi \in C^1(\R^n,\R^n), \, \|\varphi\|_\infty \leq 1 \biggr\}
\]
is finite. Indeed we will see that $\bV^0(u) \leq \bV(u) \leq 2n\bV^0(u)$. In the language of currents the integral of interest can be written
\[
\int_{\R^n} u(x) \det D(f,g_1,\dots,g_{n-1})_x \, dx = \partial \curr u(f \, dg_1 \wedge \cdots \wedge dg_{n-1}) \ ,
\]
where $\curr u$ is the current induced by integrating differential forms with density function $u$. This indicates a connection between the definition of $\BV^\alpha_c(\R^n)$ and its action on differential forms equipped with the $\alpha$-H\"older norm. One of the primary motivations of this work is to understand this connection, respectively, to characterize functions in $L^1_c(\R^n)$, or more generally currents of some dimension, that act continuously on differential forms equipped with some H\"older norm. This is achieved partially in the following theorem, where the space $\bigcap_{\alpha < \beta < 1}\BV^{\beta}_c(\R^n)$ is characterized in three different ways. A more general version is stated in Theorem~\ref{equivalence_thm}.

\begin{thm}
	\label{intro_thm}
	Assume that $u \in L_c^1(\R^n)$ and $d \in \ ]n-1,n[$. The following statements are equivalent:
	\begin{enumerate}
		\item $u \in \bigcap_{d < \delta < n}\BV^{{\delta} - (n - 1)}_c(\R^n)$.
		\item There is a sequence $(u_k)_{k \geq 0}$ in $\BV_c(\R^n)$ such that $\sum_{k\geq 0} u_k = u$ in $L^1$, $\bigcup_{k\geq 0}\spt{u_k}$ is bounded and for all $\delta \in \ ]d,n[$ there exists $C \geq 0$ such that
		\[
		\|u_k\|_{L^1} \leq C 2^{k(\delta-n)} \quad \mbox{and} \quad \bV(u_k) \leq C 2^{k(\delta - (n-1))} \ .
		\]
		\item The map $\curr u : (f,g_1,\dots,g_n) \mapsto \int_{\R^n}u(x)f(x)\det D(g_1, \dots, g_n)_x\,dx$ defined on $\Lip(\R^n)^{n+1}$ has a continuous extension to a multilinear functional
		\[
		\Lip^{\alpha}(\R^n) \times \Lip^{\beta_1}(\R^n) \times \dots \times \Lip^{\beta_{n}}(\R^n) \to \R \ ,
		\]
		whenever $\alpha + \beta_1 + \dots + \beta_{n} > n$ and $\beta_1 + \dots + \beta_{n} > {d}$.
	\end{enumerate}
\end{thm}

\noindent In (3) it makes no difference whether continuous refers to the genuine H\"older norms or a weaker topology as used for metric currents in the sense of Ambrosio and Kirchheim (discussed below). This is due to the strict inequalities for the exponents in the statement of the theorem.

\medskip 

\noindent We want to highlight two classes of functions that are of fractional bounded variation. If $U \subset \R^n$ is some bounded open set with box counting dimension of its boundary $\dim_{\boxd}(\partial U) = d < n$, then the characteristic function of $u$ is in $\bigcap_{d < \delta < n} \BV_c^{\delta-(n-1)}(\R^n)$, Corollary~\ref{holderfractal_cor}. Formulated in terms of currents, Theorem~\ref{intro_thm} in particular implies that $\partial \curr U$ extends to H\"older differential forms of exponent $\alpha > d - (n-1)$. This was already observed by Harrison and Norton \cite{HN} and by Olbermann \cite{O}.

\medskip

\noindent In analogy to the fact that the classical space $\BV_c(\R^n)$ contains Lipschitz functions with compact support, $\BV_c^{\alpha}(\R^n)$ contains certain H\"older functions. Indeed, if $u \in \Hol_c^{\alpha}(\R^n)$ and $\alpha + \beta > 1$, then $u \in \BV_c^\beta(\R^n)$. More precisely, for any $x \in \R^n$ and $r > 0$,
\[
\bV^\beta((u-u(x))\chi_{\B(x,r)}) \leq C(n,\alpha,\beta)r^{\alpha+\beta+n-1} \Hol^\alpha(u) \ ,
\]
see Corollary~\ref{holderfractal_cor}. This may not come as a surprise since in the one-dimensional case this is implied by a result of Young \cite{Y} concerning the existence of Riemann-Stieltjes integrals of H\"older functions: For $\alpha,\beta \in \ ]0,1]$ with $\alpha + \beta > 1$ there is a constant $C(\alpha,\beta) \geq 0$ such that if $u \in \Lip^{\alpha}(\R)$, $f \in \Lip^{\beta}(\R)$, $x \in \R$ and $r > 0$, then
\[
\biggl|\int_{x-r}^{x+r} (u-u(x)) \, df\biggr| \leq C(\alpha,\beta) r^{\alpha + \beta} \Hol^\alpha(u)\Hol^\beta(f) \ .
\]
This is sharp and such an estimate does not hold if $\alpha + \beta \leq 1$.

\subsection{Additional Properties of these functions} %functions of fractional  in \texorpdfstring{$\BV^\alpha$}{TEXT}}

The implication $(1)\Rightarrow (2)$ of Theorem~\ref{intro_thm} shows that $\BV^\alpha$-functions can be approximated by classical $\BV$-functions in a controlled way. See Theorem~\ref{dyadic_thm} for a quantitative version of this statement. This approximation is actually the reason for the particular definition of the fractional variation $\bV^\alpha$. This approximation property allows to extend some classical results for $\BV$-functions to $\BV^\alpha$-functions:
\begin{enumerate}
	\item (Compactness, Proposition~\ref{compactness_prop}) Assume that $\alpha \in [0,1[$ and that $(u_k)_{k \geq 0}$ is a sequence in $\BV_c^\alpha(\R^n)$ for which $\sup_{k\geq 0} \|u_k\|_{L^1} + \bV^\alpha(u_k) < \infty$ and $\bigcup_{k \geq 0}\spt(u_k)$ is bounded. Then there exists a subsequence that converges in $L^1$ to some $u \in \BV_c^\alpha(\R^n)$ with $\bV^\alpha(u) \leq \liminf_{k\to\infty}\bV^\alpha(u_k)$.
	\item (Higher integrability, Proposition~\ref{higherint_prop}) $\BV_c^\alpha(\R^n) \subset L^p_c(\R^n)$ for $1 \leq p < \frac{n}{n-1 + \alpha}$ and the inclusion $\{u \in \BV_c^\alpha(\R^n) : \spt(u) \subset K\} \hookrightarrow L^p_c(\R^n)$ is compact for all compact sets $K \subset \R^n$.
	\item (Isoperimetric inequality, Corollary~\ref{isoperimetric_cor}) Assume that $B$ is a bounded Borel set with $\chi_B \in \BV_c^\alpha(\R^n)$ for some $\alpha \in [0,1[$. Then for all $d \in \ ]n-1+\alpha,n]$,
\[
\cL^n(B) \leq C(n,d,\alpha,\diam(B)) \bV^\alpha(\chi_B)^\frac{n}{d} \ .
\]
	\item (Lusin type property, Corollary~\ref{holderfractal_cor}) Let $\alpha,\beta\in \ ]0,1[$. If $u \in \BV^\beta_c(\R^n)$ and $\alpha + \beta < 1$, then there exists $C \geq 0$, an exhaustion by measurable sets $D_1 \subset D_2 \subset \cdots \subset \R^n$ such that $\cL^n(\R^n\setminus D_k) \leq C k^{-1}$ and
\[
|u(x) - u(y)| \leq Ck|x-y|^{\alpha}
\]
for all $x,y \in D_k$.
\end{enumerate}

\noindent This Lusin type property can be seen as a partial converse to $\Hol^{\alpha}_c(\R^n) \subset \BV_c^\beta(\R^n)$ if $\alpha + \beta > 1$ stated earlier.

\subsection{Fractal currents}

Theorem~\ref{intro_thm} shows that $\BV^\alpha$-functions can be approximated in a controlled way by $\BV$-functions. In the language of currents, $\BV$-functions correspond to normal currents and this approximation statement implies that $\BV^\alpha$-functions induce a particular type of flat chains as defined by Whitney \cite{Whi}. Taking this as a starting point one can extract a subclass of flat chains (of general dimension and codimension) that can be approximated in an analogous way by normal or integral currents, see Definition~\ref{fractal_def}. This approach is not limited to Euclidean ambient spaces and also works in the setting of currents in metric spaces as introduced by Ambrosio and Kirchheim \cite{AK}. Given a metric space $X$, an integer $n \geq 0$ and parameters $\gamma \in [n,n+1[$, $\delta \in [n-1,n[$ we define the subclass $\bF_{\gamma,\delta}(X) \subset \bF_n(X)$ of flat chains, respectively, the subclass $\cF_{\gamma,\delta}(X) \subset \cF_n(X)$ of integral flat chains. As flat chains, currents in $\bF_{\gamma,\delta}(X)$ may not have finite mass, so it is natural to work with the theory of currents introduced by Lang \cite{L} that does not rely on a finite mass axiom. Similar to the observation stated above, namely that $\chi_U \in \BV_n^\alpha(\R^n)$ for domains $U$ with fractal boundaries, the space $\bF_{\gamma,\delta}(X)$ contains currents induced by fractal like objects. The guiding principle here should be that $T \in \bF_{\gamma,\delta}(X)$ if $\gamma > \dim(\spt(T))$ and $\delta > \dim(\spt(\partial T))$. As a justification for this, if $U \subset \R^n$ is a domain with box counting dimension $\dim_{\boxd}(\partial U) < n$, then $\curr U \in \cF_{n,\delta}(\R^n)$ for all $\delta \in \ ]\dim_{\boxd}(\partial U),n[$, Lemma~\ref{dimension_lem}. So for example if $K \subset \R^2$ is the Koch snowflake domain, then $\curr K \in \cF_{2,\delta}(\R^2)$ for all $\delta > \dim(\partial K) = \frac{\log 4}{\log 3}$. In this sense the results in this work can be seen as a starting point for studying fractal-like currents. Among other things it is stated in Proposition~\ref{slicing_prop} that the class $\bF_{\gamma,\delta}(X)$ behaves well with respect to push forwards, slicing and restriction operations.

\medskip

\noindent An important part in the theory of metric currents is the equivalence of top dimensional normal currents $\bN_n(\R^n)$ and functions in $\BV(\R^n)$. The fact that $\BV$-functions have a measurable decomposition into Lipschitz functions together with the slicing theory are key tools for the closure and boundary rectifiability theorems for integer rectifiable metric currents. The space $\bF_{n,d}(\R^n)$ corresponds in a similar way to $\BV_c^{d-(n-1)}(\R^n)$, Theorem~\ref{equivalence_thm}, and since $\BV_c^\alpha(\R^n)$ has some of the features of $\BV_c(\R^n)$ it may be possible to further develop a theory of fractal currents using the structure results for $\BV_c^\alpha(\R^n)$ that we obtain.

\medskip

\noindent The extension result Theorem~\ref{fractalflat_thm} shows that a given $T \in \bF_{\gamma,\delta}(X)$ can be continuously extended to H\"older test functions if the H\"older exponents are not too small. This builds on and extends the corresponding result for normal currents \cite[Theorem~4.3]{Z} by the author. As a special case, Theorem~\ref{fractalflat_thm} applies to H\"older differential forms and thus generalizes the extension results \cite[Theorem~A]{HN} by Harrison and Norton and \cite[Theorem~2.2]{Gus} by Guseynov for integrating on domains $U \subset \R^n$ with fractal boundaries. As discussed after Lemma~\ref{dimension_lem}, the conditions of $d$-summability of $\partial U$ in \cite[Theorem~A]{HN} and the slightly more general condition in \cite[Theorem~2.2]{Gus} imply that the corresponding current $\curr U$ is in $\cF_{n,d}(\R^n)$ and for this space our extension theorem applies.

\subsection{Change of variables}

In Section~\ref{push_section} we study the change of variables formula in the context of $\BV^\alpha$-functions and with respect to maps that may only be H\"older regular. The classical change of variables formula can be stated as follows: Given $u \in L^1_c(\R^n)$, $\varphi \in C^\infty(\R^n,\R^n)$ and a differential $n$-form $\omega \in \Omega^n(\R^n)$, then
\[
\int_{\R^n} u(x) \bigl(\varphi^\#\omega\bigr)(x)\, dx = \int_{\R^n} v(y)\, \omega(y)\, dy \ ,
\]
where
\begin{equation}
\label{changeofvarintro_eq}
v(y) = \sum_{x \in \varphi^{-1}(y)} u(x) \sign(\det D\varphi_x)
\end{equation}
holds almost everywhere. In the language of currents this translates to $\varphi_\#\curr u = \curr v$. In a more general setting we obtain sharp conditions under which $\varphi_\# T$ is well defined for $T \in \bN_n(X)$, or for $\partial T$ if $T \in \bF_{\gamma,\delta}(X)$, and $\varphi : X \to \ell_\infty(\N)$ has coordinate functions of possibly different H\"older regularity, see Proposition~\ref{push_prop1} and Proposition~\ref{push_prop2}.

\medskip

\noindent In the specific situation of $\BV^\alpha$-functions the following change of variables formula holds. It also includes sharp bounds on $L^p$-norms of the push forward.

\begin{thm}
\label{pushforward_thmintro}
Let $n \geq 1$, $d \in [n-1,n[$, $u \in \BV^{d-(n-1)}_c(\R^n)$ and $\varphi : \R^n\to\R^n$. Assume that $r > 0$, $\alpha_i \in \ ]0,1]$ for $i=1,\dots,n$ are such that:
	\begin{enumerate}
		\item $\spt(u) \subset [-r,r]^n$.
		\item $\max_{i=1,\dots,n}\Hol^{\alpha_i}(\varphi^i) < \infty$.
		\item $\tau_n \defl \alpha_1 + \dots + \alpha_n > d$.
	\end{enumerate}
Then $\varphi_\# \curr{u} = \curr {v_{u,\varphi}}$ is defined for some $v_{u,\varphi} \in L^1_c(\R^n)$ with
\[
\|v_{u,\varphi}\|_{L^p} \leq C(n,\tau_n,d,p,r) \bV^{d-(n-1)}(u) \Hol^{\alpha_1}(\varphi^1)^\frac{1}{p} \cdots \Hol^{\alpha_n}(\varphi^n)^\frac{1}{p}
\]
for all $1 \leq p < \frac{\tau_n}{d}$ (or $1 \leq p < \infty$ if $d=n-1=0$). Further, if $(\varphi_k)_{k \in \N}$ is a sequence of maps that converges uniformly to $\varphi$ such that $\sup_{i,k}\Hol^{\alpha_i}(\varphi_k^i) < \infty$, then $v_{u,\varphi_k}$ converges in $L^p$ to $v_{u,\varphi}$ for any $p$ in the same range. Moreover, $v_{u,\varphi} \in \bigcap_{d'<\delta<n}\BV_c^{\delta-(n-1)}(\R^n)$ for $d' \defl n + \frac{d-\tau_{n}}{\max_i \alpha_i}$.
\end{thm}

\noindent In case all the exponents are equal $\alpha = \alpha_1 = \cdots = \alpha_n$, then $d' = \frac{d}{\alpha}$. The theorem above is a special case of Theorem~\ref{pushforward_thm} where also an estimate on $\bV^{\delta-(n-1)}(v_{u,\varphi})$ is given. Note that $\varphi_\# \curr{u} = \curr {v_{u,\varphi}}$ cannot be understood as in \eqref{changeofvarintro_eq} for smooth functions because H\"older maps may not be differentiable anywhere. But $\varphi_\# \curr{u} = \curr {v_{u,\varphi}}$ is well defined by approximation.

\medskip

\noindent Higher integrability properties of the Brouwer degree function $y \mapsto \degr \varphi U y$, where $U$ is a domain with fractal boundary and $\varphi$ is a H\"older map, has already been studied by Olbermann in \cite{O} and by De Lellis and Inauen in \cite{DI}. In \cite{Z2} domains with finite perimeter are treated but the coordinates of $\varphi$ are allowed to have different regularity. In the smooth setting it holds that $\varphi_\# \curr{\chi_U} = \curr{\degr \varphi U \cdot}$. We prove that this identity is also true for H\"older maps $\varphi$ if $U$ has fractal boundary, Lemma~\ref{deg_lem}. So these degree functions fit into the scope of Theorem~\ref{pushforward_thmintro}, and we obtain:

\begin{thm}
\label{degreeint_corintro}
Let $U \subset \R^n$ be an bounded open set such that $\partial U$ has box counting dimension $d \in [n-1,n[$. Assume $\varphi : \R^n\to\R^n$ satisfies $\max_{i}\Hol^{\alpha_i}(\varphi^i) < \infty$ for some $\alpha_1,\dots,\alpha_n \in \ ]0,1]$ with $\tau_n \defl \alpha_1 + \dots + \alpha_n > d$. Then
\[
\|\degr \varphi U \cdot \|_{L^p} \leq C(U,n,\tau_n,p) \Hol^{\alpha_1}(\varphi^1)^\frac{1}{p}\cdots\Hol^{\alpha_n}(\varphi^n)^\frac{1}{p}
\]
for all $1 \leq p < \frac{\tau_n}{d}$ (or $1 \leq p < \infty$ if $d=n-1=0$). Further, if $(\varphi_k)_{k \in \N}$ is a sequence of maps that converges uniformly to $\varphi$ such that $\sup_{i,k}\Hol^{\alpha_i}(\varphi_k^i) < \infty$, then $\degr {\varphi_k} U \cdot$ converges in $L^p$ to $\degr \varphi U \cdot$ for $p$ in the same range.

\medskip

\noindent Moreover, $\degr \varphi U \cdot \in \bigcap_{d'<\delta<n}\BV_c^{\delta-(n-1)}(\R^n)$ for $d' \defl n + \frac{d-\tau_{n}}{\max_i \alpha_i}$. If $F \in \Lip(\R^n)^n$ and $\beta_1,\dots,\beta_n \in \ ]0,1]$ satisfy $\beta \defl \beta_1 + \cdots + \beta_n > d'$, then
\begin{equation*}
\left|\int_{\R^n} \degr \varphi U y \det DF_y \, dy \right| \leq C'(U,n,\tau_n,\tau_{n-1},\beta) h(\varphi)^{\beta + 1 - n}H_{n-1}(\varphi)H_n(F)  \ ,
\end{equation*}
where $h(\varphi) \defl \min_i \Hol^{\alpha_i}(\varphi_i)$, $H_{n-1}(\varphi) \defl \max_{j}\prod_{i \neq j} \Hol^{\alpha_i}(\varphi^i)$,  and $H_{n}(F) \defl \prod_{i=1}^n\Hol^{\beta_i}(F^i)$.
\end{thm}

\noindent Also here it holds that $d' = \frac{d}{\alpha}$ in case $\alpha = \alpha_1 = \cdots = \alpha_n$. \noindent The theorem above generalizes \cite[Proposition~2.4]{Z2}, \cite[Theorem~1.1, Theorem~1.2(i)]{O} and \cite[Theorem~2.1]{DI}. It also proves a conjecture stated in \cite{DI} about the higher integrability of the Brouwer degree function for a map with coordinate functions of variable H\"older regularity.

\subsection{Structure of the paper}

In Section~\ref{prelim_sec} we introduce the notation that is used throughout the paper and review results about metric currents. We do not follow strictly the theory by Ambrosio and Kirchheim \cite{AK} or the modification by Lang \cite{L}. Mostly for simplicity of presentation we work in a setting where all currents are assumed to have compact support. With this definition, if a current is restricted to a compact set that contains its support, then the theory of Lang applies. This is justified in Subsection~\ref{metriccur_subsec}. The benefit of this approach is also that we do not have to assume that our ambient space is locally compact and we can talk for example about push forwards into infinite dimensional Banach spaces as in Section~\ref{push_section} without technical difficulties.

\medskip

\noindent In Section~\ref{fracvar_sec} we start by introducing functions of fractional bounded variation and state some direct consequences of the definition. This section can be read without any prior knowledge about currents. The main result it contains is Theorem~\ref{dyadic_thm} that allows to approximate functions of fractional bounded variation by classical functions of bounded variation in a controlled way. Building on this approximation result and the structure of $\BV$-functions we obtain compactness and higher integrability properties for $\BV^\alpha$-functions in Subsection~\ref{consequences_subsec}.

\medskip

\noindent In Section~\ref{fractalcur_sec}, motivated by Theorem~\ref{dyadic_thm}, we introduce fractal currents and show that they contain a large class of currents induced by fractal sets in Lemma~\ref{dimension_lem}. In this general setting we prove the main extension result Theorem~\ref{fractalflat_thm}. This allows to show that fractal currents of codimension zero in an Euclidean space are induced by functions of fractional bounded variation. This is done in Subsection~\ref{fractalcurr_subsec}. With this at hand we obtain different characterizations of this type of functions in Theorem~\ref{equivalence_thm} and additional properties in Corollary~\ref{holderfractal_cor}. Subsection~\ref{smoothingsection} about smoothings of currents is used to give one such characterization purely in terms of multiliear functionals on H\"older test functions without assuming that this functional is represented by integration (has finite mass).

\medskip

\noindent In Section~\ref{push_section} we first study mass bounds of push forwards of currents into the Banach space $\ell_\infty(\N)$. This allows to study the general situation of push forwards of fractal currents with respect to H\"older regular maps. In Subsection~\ref{push_subsec} this is then further specialized to finite dimensional Euclidean targets but arbitrary domains and even further in Subsection~\ref{higherint_subsec} where also the domain is assumed to be Euclidean of the same finite dimension. In this last subsection we also discuss higher integrability of such push forwards. This specializes to Brouwer degree functions on fractal domains. In order to do that these Brouwer degree functions are related to the push forward of currents with respect to H\"older maps in Lemma~\ref{deg_lem}.

\vspace{0.5cm}
{\parindent0mm
\textbf{Acknowledgements:}
I would like to thank Valentino Magnani, Eugene Stepanov and Dario Trevisan for useful feedback and suggestions, and an anonymous referee for a number of useful comments and corrections.
}

\section{Preliminaries and notation}

\label{prelim_sec}

\noindent Given a metric space $(X,d)$ we denote by $\B(x,r)$ the closed and by $\oB(x,r)$ the open ball of radius $r > 0$ around $x \in X$. Similarly, for any nonempty subset $A \subset X$ the closed neighborhood of radius $r$ is $\B(A,r) \defl \{y \in X : \dist(A,y) \leq r\}$ and the open neighborhood of radius $r$ is $\oB(A,r) \defl \{y \in X : \dist(A,y) < r\}$. For $\R^n$ we use the notation $\id_{\R^n} = (\pi^1,\dots,\pi^n)$, where $\pi^i(x_1,\dots,x_n) \defl x_i$ is the $i$th coordinate projection. A similar notation is used for $\ell_\infty \defl \{f : \N \to \R : \sup_{i \in \N} |f(i)| < \infty\}$, where we define $\pi^i(f) \defl f_i \defl f(i)$ to be the evaluation of $f$ at $i \in \N$. With $\omega_n$ we denote the volume of the unit ball in $\R^n$.

\subsection{H\"older maps}
\label{holder}
Let $\alpha \in [0,1]$. Given a map $\varphi : (X,d_X) \to (Y,d_Y)$ we define
\[
\Hol^\alpha(\varphi) \defl \sup_{x \neq x'} \frac{d_Y(\varphi(x),\varphi(x'))}{d_X(x,x')^\alpha} \ .
\]
The set of all such maps where this is finite is denoted by $\Hol^\alpha(X,Y)$. For real valued functions we abbreviate $\Hol^\alpha(X) \defl \Hol^\alpha(X,\R)$. In case $\alpha=1$ the usual notation $\Lip(X,Y)$ and $\Lip(X)$ are used. If $\alpha = 0$, then $\Lip^0(X,\R^n)$ is just the space of bounded functions. Indeed, given $\varphi : X \to \R^n$ and some fixed $x_0 \in X$, a simple consequence of the triangle inequality is that
\begin{equation}
\label{boundedfct_eq}
\|\varphi - \varphi(x_0)\|_\infty \leq \Lip^0(\varphi) \leq 2\|\varphi\|_\infty \ .
\end{equation}

\noindent Assume that $X$ is a bounded metric space and $0 \leq \alpha \leq \beta \leq 1$. If $\varphi \in \Hol^\beta(X,Y)$, then for $x,x' \in X$
\[
d_Y(\varphi(x),\varphi(x')) \leq \Hol^\beta(\varphi)d_X(x,x')^\beta \leq \Hol^\beta(\varphi)\diam(X)^{\beta-\alpha} d_X(x,x')^{\alpha} \ .
\]
Hence
\begin{equation}
\label{inclusionholder}
\Hol^\alpha(\varphi) \leq \Hol^\beta(\varphi)\diam(X)^{\beta-\alpha}
\end{equation}
and, in particular $\Hol^\beta(X,Y) \subset \Hol^\alpha(X,Y)$.

\medskip

\noindent Assume that $X$ is a bounded metric space, $0 \leq \alpha < \beta \leq 1$ and $(f_k)_{k \in \N}$ is a sequence in $\Hol^\beta(X)$ with $\sup_k \Hol^\beta(f_k) < \infty$ and such that $f_k$ converges uniformly to $f$. Then
\begin{equation}
\label{holdertrick_eq}
\lim_{k \to \infty} \Hol^\alpha(f_k-f) = 0 \ .
\end{equation}
This is well known but for the sake of convenience we give a proof here.
\begin{proof}
We may assume that $f = 0$ and set $H \defl \sup_k \Hol^\beta(f_k)$. Fix $\epsilon > 0$ and assume that $k \in \N$ is large enough such that $\|f_k\|_\infty \leq \epsilon^\beta$. If $d_X(x,x') \leq \epsilon$, then
\begin{align*}
|f_k(x)-f_k(x')| & \leq H d_X(x,x')^\beta \leq H d_X(x,x')^{\beta-\alpha}d_X(x,x')^\alpha \\
 & \leq H\epsilon^{\beta-\alpha}d_X(x,x')^\alpha \ .
\end{align*}
If $d_X(x,x') \geq \epsilon$, then
\begin{align*}
|f_k(x)-f_k(x')| & \leq 2\epsilon^\beta \leq 2\epsilon^{\beta-\alpha}d_X(x,x')^\alpha \ .
\end{align*}
Thus $\lim_{k\to\infty}\Hol^\alpha(f_k) = 0$. This proves \eqref{holdertrick_eq}.
\end{proof}

\noindent We avoid this trick for the most part but we will use it in the proof of Theorem~\ref{pushforward_thm}. This is also the reason for the comment after Theorem~\ref{intro_thm} in the introduction. So whenever there are open boundary conditions on H\"older exponents in a statement it is often not really relevant what topology on H\"older functions we choose. Because of this, we try to check the sharpness of the boundary case for H\"older exponents whenever there is an open boundary condition in a statement.

\medskip

\noindent The following construction to approximate H\"older functions by Lipschitz functions is described for example in the appendix of \cite{G2} written by Semmes. For a proof see \cite[Theorem~B.6.16]{G2} or \cite[Lemma~2.2]{Z}. This construction is very similar to the one used in order to prove the McShane-Whitney extension theorem for Lipschitz functions.

\begin{lem}
\label{approximation_lem}
Let $f \in \Lip^\alpha(X)$ for some $\alpha \in \ ]0,1]$ and $H \geq 0$ such that $\Lip^{\alpha}(f)\leq H$. For $\epsilon \in \ ]0,1]$ define $f_\epsilon : X \to \R$ by
\begin{equation}
\label{approxdef_not}
f_{\epsilon}(x) \defl \inf_{y \in X} f(y) + H \epsilon^{\alpha - 1}d(x,y) \ .
\end{equation}
Then
\begin{enumerate}
	\item $\|f_{\epsilon} - f\|_\infty \leq H \epsilon^{\alpha}$,
	\item $\Lip(f_{\epsilon}) \leq H \epsilon^{\alpha - 1}$,
	\item $\Hol^\alpha(f_\epsilon) \leq 3H$,
	\item $f_{\epsilon}(x) = \inf_{y \in \B(x,\epsilon)} f(y) + H \epsilon^{\alpha - 1}d(x,y)$,
	\item $\spt(f_\epsilon) \subset \B(\spt(f),\epsilon)$,
	\item if $g \in \Lip^\alpha(X)$ with $\Lip^{\alpha}(g)\leq H$, then $\|f_\epsilon - g_\epsilon\|_\infty \leq \|f-g\|_\infty$.
\end{enumerate}
\end{lem}

\noindent In case $f$ is bounded, then
\[
\bar f_\epsilon \defl \min\{\max\{f_\epsilon,-\|f\|_\infty\},\|f\|_\infty\}
\]
satisfies all the properties of Lemma~\ref{approximation_lem} except (4) but additionally $\|\bar f_\epsilon\|_\infty \leq \|f\|_\infty$.

\subsection{Metric currents}

\label{metriccur_subsec}

The space of currents of dimension $n$ in $\R^m$ is the dual space of compactly supported differential $n$-forms equipped with an appropriate topology. The resulting theory is described in great detail in the book about this topic by Federer \cite{F}. The theory of metric currents was introduced by Ambrosio and Kirchheim in \cite{AK}. Lang gives a definition of metric currents in locally compact metric spaces that does not rely on a finite mass assumption \cite{L}. Mostly for simplicity's sake we restrict the discussion in this work to currents in metric spaces that have compact support.

\begin{defi}[Metric currents with compact support]
\label{metric_current_def}
Let $X$ be a metric space and $n \geq 0$ be an integer. A multilinear map $T : \Lip(X)^{n+1} \to \R$ is an $\mathbf{n}$-\textbf{dimensional metric current} if it satisfies:
\begin{enumerate}
	\item $T(f,g^1,\dots,g^n) = 0$ whenever some $g^i$ is constant in a neighbourhood of $\spt(f)$.
	\item There is a compact set $K \subset X$ such that $T(f,g^1,\dots,g^n) = 0$ whenever $\spt(f) \cap K = \emptyset$.
	\item $\lim_{k\to\infty} T(f_k,g_k^1,\dots,g_k^n) = T(f,g^1,\dots,g^n)$ whenever $\lim_{k\to\infty}\|f_k - f\|_\infty = \lim_{k\to\infty} \|g^i_k - g^i\|_\infty = 0$ for all $i$ and $\sup_{i,k}\{\Lip(g^i_k),\Lip(f_k)\} < \infty$.
\end{enumerate}
The vector space of such $T$ is denoted by $\cD_n(X)$.
\end{defi}

\noindent We often abbreviate $T(f,g)$ for $T(f,g^1,\dots,g^n)$ or $T(f)$ for $T(f^1,\dots,f^{n+1})$. Here are some definitions that we will use: The \textbf{boundary} of a current $T \in \cD_n(X)$ for $n \geq 1$ is defined by
\[
\partial T(f,g^1,\dots,g^{n-1}) \defl T(1,f,g^1,\dots,g^{n-1}) \ .
\]
It can be shown that $\partial T \in \cD_{n-1}(X)$ and $\spt(\partial T) \subset \spt(T)$. In case $\varphi : X \to Y$ is Lipschitz, then the \textbf{push forward} $\varphi_\# : \cD_n(X) \to \cD_n(Y)$ is defined by
\[
(\varphi_\# T)(f,g^1,\dots,g^{n-1}) \defl T(f\circ \varphi ,g^1 \circ \varphi,\dots,g^{n-1}\circ \varphi) \ .
\]
The \textbf{support of} $\mathbf T$, denoted by $\spt(T)$, is the set of points $x \in X$ with the property that for any $\epsilon > 0$ there are $f,g^1,\dots,g^n \in \Lip(X)$ with $\spt(f) \subset \B(x,\epsilon)$ and $T(f,g) \neq 0$. Compare with Section~$3$ of \cite{L} for the definitions above. A sequence $(T_k)_{k\geq 0}$ in $\cD_n(X)$ \textbf{converges weakly} to $T \in \cD_n(X)$ if there exists a compact set $K \subset X$ such that $\bigcup_k\spt(T_k) \subset K$ and $\lim_{k \to \infty} T_k(f,g) = T(f,g)$ for all $(f,g) \in \Lip(X)^{n+1}$. 

\medskip

\noindent Since we use a slightly different definition than Lang we want to make sure that the support of a current as defined here is actually compact and does what it is supposed to do. On a temporary basis we define $\cK_T$ to be the collection of all closed sets $A \subset X$ for which $T(f,g^1,\dots,g^n) = 0$ whenever $\spt(f) \cap A = \emptyset$. The lemma below is in a slightly different setting contained in \cite[Lemma~3.2]{L}. For the sake of convenience we include a proof here.

\begin{lem}
\label{spt_lem}
For $T \in \cD_n(X)$, the following statements hold:
\begin{enumerate}
	\item $\spt(T) = \bigcap \cK_T$. In particular $\spt(T)$ is compact. 
	\item $T(f,g) = 0$ whenever $f = 0$ on $\spt(T)$.
	\item $T(f,g^1,\dots,g^n) = 0$ whenever some $g^i$ is constant on $\spt(f)$.
\end{enumerate}
\end{lem}

\begin{proof}
(1): If $x \in X \setminus \spt(T)$ there is an $\epsilon > 0$ such that $X \setminus \oB(x,\epsilon) \in \cK_T$. Hence $x \in X \setminus \bigcap \cK_T$ and therefore $\bigcap \cK_T \subset \spt(T)$. On the other hand, if $x \in \spt(T)$, then any ball $\B(x,\epsilon)$ intersects any set in $\cK_T$. Thus $x$ is in the closure of any set in $\cK_T$ and because these sets are closed, $x \in \bigcap \cK_T$. This shows that $\spt(T) = \bigcap \cK_T$ and $\spt(T)$ is compact because $\cK_T$ contains a compact set by definition.

\medskip 

\noindent (2): Let $K$ be a compact set as guaranteed by axiom (2). We know that $\spt(T) \subset K$ by (1). Consider $f,g^1,\dots,g^n \in \Lip(X)$ and assume that $f = 0$ on $\spt(T)$. Without loss of generality we can assume that $f \geq 0$. Otherwise we decompose $f$ into Lipschitz functions $f = f_+ - f_-$ where $f_+ \defl \max\{0,f\}$ and $f_-\defl \max\{0,-f\}$. Define $f_k(x) \defl \max\{0,f(x) - Lk^{-1}\}$ for $k \in \N$, where $L = \Lip(f)$. If $f_k(x) > 0$, then for any $y \in \spt(T)$ it holds
\[
Lk^{-1} < f(x) \leq f(y) + Ld(x,y) = Ld(x,y) \ .
\]
Hence $x \notin \oB(\spt(T),k^{-1})$. The compact set $K \setminus \oB(\spt(T),k^{-1})$ can be covered by finitely many open balls $\oB(x_1,\epsilon_1),\dots,\oB(x_m,\epsilon_m)$ such that $\B(x_i,\epsilon_i) \cap \spt(T) = \emptyset$ for all $i$ and $T(h,g) = 0$ whenever $\spt(h) \subset \B(x_i,\epsilon_i)$ for some $i$. Consider a Lipschitz partition of unity $(\varphi_i)_{i=1}^m$ with $\spt(\varphi_i) \subset \B(x_i,\epsilon_i)$ and $\sum_i \varphi_i = 1$ in an open neighbourhood $U_K$ of $K \setminus \oB(\spt(T),k^{-1})$. The Lipschitz function $\varphi \defl 1- \sum_i \varphi_i$ satisfies $f_k(x)\varphi(x) = 0$ for $x \in U_K\cup \oB(\spt(T),k^{-1})$ because if $f_k(x) \neq 0$, then $x \in U_K$ but this implies $\varphi(x) = 1- \sum_i \varphi_i(x) = 0$. Thus $\spt(f_k\varphi) \cap K = \emptyset$ and $\spt(f_k\varphi_i) \subset \B(x_i,\epsilon_i)$ for all $i$ and therefore
\begin{align*}
T(f_k,g) & = T(f_k\varphi,g) + \sum_{1\leq i\leq m} T(f_k\varphi_i,g) = 0 \ .
\end{align*}
This holds for all $k$. By taking the limit, it follows from the continuity axiom that $0 = T(f,g^1,\dots,g^n)$.

\medskip

\noindent (3): This can be proved as in \cite{L}. For $k \in \N$ let $\beta_k : \R \to \R$ be given by $-\beta_k(-s)=\beta_k(s)=\max\{0,s-k^{-1}\}$ for $s \geq 0$. If $g^i$ is equal to $c \in \R$ on $\spt(f)$, then $\beta_k\circ (g^i-c) + c$ converges to $g^i$ and is constant in a neighborhood of $\spt(f)$. Thus the statement follows as above by taking the limit.
\end{proof}

\noindent The lemma above shows that any $T \in \cD_n(X)$ can be recovered from its restriction $T_K \in \cD_n(K)$ to any compact set $K \subset X$ that contains $\spt(T)$. Here, $T_K : \Lip(K)^{n+1} \to \R$ is defined by $T_K(\tilde f,\tilde g) \defl T(f,g)$, where $\tilde f$ and $\tilde g$ are arbitrary Lipschitz extensions of $f$ and $g$ respectively. It can be shown that $T_K$ actually defines a metric current, see \cite[Proposition 3.3]{L} for the related result about the restriction of local currents. On the other hand, $T$ can be recovered from $T_K$ by restricting the test functions defined on $X$ to $K$. Because $K$ is compact (and hence locally compact), the axioms for $\cD_n(K)$ described here are identical to the axioms for local currents in \cite{L} and all the results for currents obtained in this reference hold for $\cD_n(K)$. We will thus apply the results of \cite{L} to currents in our setting without mentioning the restriction to some compact set $K$. Below are some more basic definitions and properties that use the concept of mass.

\medskip

\noindent As in \cite[Definition~4.1]{L}, for any open set $V \subset X$, the \textbf{mass} in $V$ of a current $T \in \cD_n(X)$ is defined by
\[
\bM_V(T) \defl \sup \sum_{\lambda \in \Lambda} T(f_\lambda, g^1_{\lambda},\dots,g^n_{\lambda}) < \infty \ ,
\]
where the supremum ranges over all finite collections $\{(f_\lambda, g_{\lambda}^1,\dots,g_{\lambda}^n)\}_{\lambda \in \Lambda}$ of Lipschitz maps in $\Lip(X)^{n+1}$ that satisfy $\bigcup_{\lambda \in \Lambda} \spt(f_\lambda) \subset V$, $\sum_{\lambda \in \Lambda} |f_\lambda| \leq 1$ and $\sup_{i,\lambda}\Lip(g^i_{\lambda}) \leq 1$. With $\bM_n(X)$ we denote the currents $T \in \cD_n(X)$ of finite mass, i.e.\ $\bM(T) \defl \bM_X(T) < \infty$. The set function $\|T\| : 2^X \to [0,\infty]$ is defined by
\[
\|T\|(A) \defl \inf\{\bM_V(T) \ :\ V \subset X \mbox{ open},\ A \subset V \} \ .
\]
Similar to \cite[Theorem~4.3]{L} we obtain:

\begin{lem}
\label{mass_lem}
$\|T\|(X\setminus \spt(T)) = 0$, and if $\bM(T) < \infty$, then $\|T\|$ is a finite Radon measure that satisfies
\begin{align*}
\bigl|T(f,g^1,\dots,g^n)\bigr| & \leq \int_X |f(x)| \, d\|T\|(x) \cdot \prod_{i=1}^n \Lip(g^i|_{\spt(f)}) \\
 & \leq \bM(T) \|f\|_\infty\prod_{i=1}^n \Lip(g^i|_{\spt(f)})
\end{align*}
for all $(f,g^1,\dots,g^n) \in \Lip(X)^{n+1}$.
\end{lem}

\begin{proof}
$\|T\|(X\setminus \spt(T)) = \bM_{X\setminus \spt(T)}(T) = 0$ follows directly from Lemma~\ref{spt_lem} and the rest is as in the proof of \cite[Theorem~4.3]{L}.
\end{proof}

\noindent If $T \in \bM_n(X)$, then $T$ extends to a functional on $\mathcal B^\infty(X) \times \Lip(X)^n$ by \cite[Theorem~4.4]{L}, where $\mathcal B^\infty(X)$ is the space of bounded Borel functions on $X$. This in particular allows to define the restriction $T \res u$ in case $u \in \mathcal B^\infty(X)$ by $(T\res u)(f,g) = T(fu,g)$, \cite[Definition~4.5]{L}.

\medskip

\noindent A current $T \in \cD_n(X)$ with bounded normal mass $\bN(T) \defl \bM(T) + \bM(\partial T) < \infty$ is a \textbf{normal current}. The vector space of all normal currents is $\bN_n(X)$. Note also here that because we assume currents to have compact support, a normal current $T \in \bN_n(X)$ can be seen as a current in $\bN_n(K)$ for any compact set $K \supset \spt(T)$. All the results about normal currents in \cite{L} apply to this restriction in $\bN_n(K)$. Similarly we will rely on the results in \cite{L} about currents with finite mass. If $X = U$ for some open subset $U \subset \R^n$, our notion of metric normal current agrees with the classical definition in \cite[Section~4.1.7]{F} and the normal masses are comparable, \cite[Theorem~5.5]{F}. Note that normal currents as defined in \cite[Section~4.1.7]{F} are also assumed to have compact support. Standard examples of currents are given by functions $u \in L^1_c(U)$ where $U \subset \R^n$ is open. It is shown in \cite[Proposition~2.6, Equation~(4.5)]{L} that
\[
\curr u(f,g^1,\dots,g^n) \defl \int_U u(x)f(x)\det D(g^1,\dots,g^n)_x \, dx
\]
defines a current in $\bM_n(U)$ with $\bM(\curr u) = \|u\|_{L^1} = \int_U |u(x)|\,dx$. We will also abbreviate $\curr B$ for $\curr{\chi_B}$ in case $B \subset U$ is some Borel set with compact closure in $U$.

\medskip

\noindent Given an open subset $U \subset \R^n$, a function $u \in L^1(U)$ is of bounded variation, i.e.\ in $\BV(U)$, if
\[
\bV(u) \defl \sup\biggl\{\int_{U} u(x)\operatorname{div} \varphi (x) \, dx \, : \, \varphi \in C_c^1(U,\R^n),\, \|\varphi\|_\infty \leq 1 \biggr\} < \infty \ .
\]
Because our currents are assumed to have compact support we mostly consider the space $\BV_c(U) \defl L^1_c(U) \cap \BV(U)$. It is easy to see that for $u \in \BV_c(U)$ we can replace $C_c^1(U,\R^n)$ by $C^1(U,\R^n)$ in the definition of $\bV(u)$ above. Because we will use the relationship between normal currents and functions of bounded variation several times, we formulate it as a lemma. It follows directly from \cite[Theorem~7.2]{L}, see also \cite[Theorem~3.7]{AK}.

\begin{lem}
	\label{BVnormal_lem}
Assume that $U \subset \R^n$ is an open set. If $u \in \BV(U)$, then
\[
\biggl|\int u(x) \det D(f,g^1,\dots,g^{n-1})_x\, dx\biggr| \leq \bV(u)\|f\|_\infty\operatorname{Lip}(g^1)\cdots\operatorname{Lip}(g^{n-1}) \ ,
\]
for all $(f,g^1,\dots,g^{n-1}) \in \Lip_c(U) \times \Lip(U)^{n-1}$. The identification $\bN_n(U) = \BV_c(U)$ holds in the sense that any $T \in \bN_n(U)$ is equal to $\curr u$ for some $u \in \BV_c(U)$ and for any $u \in \BV_c(U)$, the current $\curr u$ is in $\bN_n(U)$. Moreover, $u \in \BV_c(U)$ has mass $\bM(\curr u) = \|u\|_{L^1}$ and boundary mass $\bM(\partial \curr{u}) = \|Du\|(U) = \bV(u)$, where $\|Du\|$ is the Borel measure associated with the variation of $u$.
\end{lem}

\noindent If $K \subset X$ is a compact subset, then the \textbf{flat norm} of $T \in \cD_n(X)$ is defined by
\begin{equation}
\label{flat_def1}
\bF_K(T) \defl \inf\bigl\{\bM(T-\partial S) + \bM(S) \, : \, S \in \bN_{n+1}(X),\, \spt(S) \subset K\bigr\} \ .
\end{equation}
This defines a norm on $\bN_{n,K}(X) \defl \{T \in \bN_n(X)\, : \, \spt(T) \subset K\}$. The closure of $\bN_{n,K}(X)$ with respect to $\bF_K$ is $\bF_{n,K}(X)$. The space of \textbf{flat chains} in $X$, denoted by $\bF_n(X)$, is the union of $\bF_{n,K}(X)$ ranging over all compact subsets $K \subset X$. Note that in case $U \subset \R^m$ is some open set, then this definition agrees with the classical definition of $\bF_n(U)$ in \cite[Section~4.1.12]{F} due to \cite[Theorem~5.5]{L}. It follows from \cite[Theorem~4.1.23]{F} that any $T \in \bF_n(U)$ can be approximated with respect to the flat norm by real polyhedral chains. Moreover if $m = n$, the space $\bF_n(U)$ can be identified with $L^1_c(U)$, see \cite[Section~4.1.18]{F}. 

\medskip

\noindent The space of $n$-dimensional integral currents in $X$ is $\bI_n(X)$, see \cite[Section~4.1.24]{F} for the classical definition and \cite[Definition~4.2]{AK} or \cite[Definition~8.6]{L} for the definition in the setting of metric currents (again we additionally assume integral currents to have compact support). Similarly to flat chains we can define
\begin{equation}
\label{flat_def2}
\cF_K(T) \defl \inf\bigl\{\bM(T-\partial S) + \bM(S) \, : \, S \in \bI_{n+1}(X),\, \spt(S) \subset K\bigr\} \ ,
\end{equation}
for $T \in \bI_{n,K}(X) \defl \{T \in \bI_n(X)\, : \, \spt(T) \subset K\}$. The resulting space $\cF_n(X)$ of \textbf{integral flat chains} in $X$ is obtained analogously to $\bF_n(X)$ above. In the classical setting, where $U \subset \R^m$ is some open set, it holds
\[
\cF_n(U) = \bigl\{R+\partial S \, : \, R \in \cR_n(U),\, S \in \cR_{n+1}(U)\bigr\} \ ,
\]
where $\cR_n(U)$ is the space of $n$-dimensional integer rectifiable currents in $U$, see \cite[Section~4.1.24]{F}. If $m = n$, then $\cF_n(U) = \cR_n(U)$ can be identified with $L_c^1(U,\Z)$.

\subsection{H\"older currents}

Let $(X,d)$ be a metric space, $n \geq 0$ be an integer and $\alpha_1, \dots,\alpha_{n+1} \in \ ]0,1]$. A multilinear map
\begin{equation}
\label{holder_current}
\bar T : \Hol^{\alpha_1}(X) \times \dots \times \Hol^{\alpha_{n+1}}(X) \to \R
\end{equation}
is a \textbf{H\"older current} if it satisfies the same axioms as a current in Definition~\ref{metric_current_def} with the occurrences of $\Lip$ for the $i$th coordinate replaced by $\Hol^{\alpha_i}$. Note that in case $\alpha = \alpha_1 = \cdots = \alpha_{n+1} > 0$, then $\bar T$ is a H\"older current as in \eqref{holder_current} if and only if $\bar T \in \cD_n(X,d^\alpha)$. On compact subsets $K \subset X$ the inclusion $\Lip(K) \subset \Hol^\alpha(K)$ is continuous by \eqref{inclusionholder} and dense by Lemma~\ref{approximation_lem}. Therefore any H\"older current as in \eqref{holder_current} is the unique continuous extension of a current $T \in \cD_n(X)$. Note that even if $X$ is not compact, the support of our functionals are, and it is therefore always possible to restrict to a compact metric space due to Lemma~\ref{spt_lem}. It is stated in \cite[Theorem~4.7]{Z} that a nonzero H\"older current as in \eqref{holder_current} can only exist if $\alpha_1 + \dots + \alpha_{n+1} > n$. This is sharp due to \cite[Theorem~4.3]{Z}: Any $T \in \bN_n(X)$ has a unique continuous extension to a current in $\cD_n(X,d^\alpha)$ if $\alpha > \frac{n}{n+1}$, or more generally, to a unique H\"older current $\bar T$ as in \eqref{holder_current} if $\alpha_1 + \dots + \alpha_{n+1} > n$. We will use the following bounds on this extension provided by \cite[Equation~(4.7)]{Z}. Let $T \in \bN_n(X)$, $\epsilon \in \ ]0,1]$, $\beta \defl \alpha_2 + \cdots + \alpha_{n+1}$, $\gamma \defl \alpha_1 + \beta$, $f = (f^1,\dots,f^{n+1}) \in \Hol^{\alpha_1}(X) \times \dots \times \Hol^{\alpha_{n+1}}(X)$ and $f_\epsilon = (f^1_\epsilon,\dots,f_\epsilon^{n+1}) \in \Lip(X)^{n+1}$. Assume that $\gamma > n$ and for all $i = 1,\dots,n+1$,
\begin{enumerate}
	\item $\Lip(f^i_\epsilon) \leq H_i \epsilon^{\alpha_i-1}$,
	\item $\|f^i-f^i_\epsilon\|_\infty \leq H_i\epsilon^{\alpha_i}$,
	\item $\|f^1_\epsilon\|_\infty \leq H_i'$,
\end{enumerate}
where $H_i \geq \Hol^{\alpha_i}(f^i)$ and $H_1' \geq \|f^1\|_\infty$. These assumptions are justified by Lemma~\ref{approximation_lem}. Then
\begin{equation}
\label{extension_bound}
\bigl|\bar T(f)-T(f_\epsilon)\bigr| \leq C\bigl[\bM(T) H_1 \epsilon^{\gamma-n} + \bM(\partial T)H_1' \epsilon^{\beta-(n-1)}\bigr]\prod_{i=2}^{n+1}H_i \ ,
\end{equation}
for some $C = C(n,\gamma) \geq 0$. Note that $\beta - (n-1) \geq \gamma - n > 0$ by assumption.

\section{Functions of bounded fractional variation}
\label{fracvar_sec}

\noindent In this section we define functions of bounded fractional variation and prove some properties. This section can be read without any knowledge about currents.

\subsection{Simple consequences}

For an open set $U \subset \R^n$, a function $u \in L^1(U)$ and an exponent $\alpha \in [0,1]$ we define
\begin{equation}
\label{bvalpha_def}
\bV^\alpha(u) \defl \sup \biggl|\int_{U} u(x) \det D(f,g^1,\dots,g^{n-1})_x \,dx \biggr| \ ,
\end{equation}
where the supremum is taken over all $(f,g^1,\dots,g^{n-1}) \in \Lip_c(U) \times \Lip(U)^{n-1}$ with $\Lip^{\alpha}(f) \leq 1$ and $\Lip(g^i)\leq 1$ for $i=1,\dots,n-1$. The class of functions $u \in L^1(U)$ with $\bV^\alpha(u) < \infty$ is denoted by $\BV^\alpha(U)$ and $\BV^\alpha_c(U) \defl \BV^\alpha(U) \cap L^1_c(U)$ are those with (essentially) compact support. The next lemma links this definition with the classical definition of functions of bounded variation.

\begin{lem}
	\label{classicaleq_lem}
	$\BV(U) = \BV^0(U)$. Indeed, if $u \in L^1(U)$, then
	\[
	\bV^0(u) \leq \bV(u) \leq 2n\bV^0(u) \ .
	\]
\end{lem}

\begin{proof}
	Let $u \in \BV(U)$ and fix some $x_0 \in U$. It follows from Lemma~\ref{BVnormal_lem} that for all $(f,g^1,\dots,g^{n-1}) \in \Lip_c(U) \times \Lip(U)^{n-1}$ it holds
	\begin{align}
	\nonumber
	\biggl|\int_{U} u(x) &\det D\bigl(f,g^1,\dots,g^{n-1}\bigr)_x \, dx \biggr| \\
	\nonumber
	& = \biggl|\int_{U} u(x) \det D\bigl(f-f(x_0),g^1,\dots,g^{n-1}\bigr)_x \, dx \biggr| \\
	\label{massestimate_eq}
	& \leq \bV(u)\|f - f(x_0)\|_\infty \Lip(g^1) \cdots \Lip(g^{n-1}) \ .
	\end{align}
	From \eqref{boundedfct_eq} it follows that $\|f-f(x_0)\|_\infty \leq \Lip^0(f)$ and thus $\bV^0(u) \leq \bV(u)$ and $u \in \BV^0(U)$. For the other inclusion, let $u \in \BV^0(U)$. If $\varphi \in C_c^1(U,\R^n)$ with $\|\varphi\|_\infty \leq 1$, then $\Lip^0(\varphi^i) \leq 2\|\varphi^i\|_\infty \leq 2$ for all $i=1,\dots,n$ and from \eqref{boundedfct_eq} it follows
	\begin{align*}
	\biggl|\int_{U} u(x) \operatorname{div} \varphi(x) \, dx\biggr| & \leq \sum_{i = 1}^n \biggl|\int_{U} u(x) \det D\bigl(\varphi^i, \pi^1,\dots,\pi^{i-1},\pi^{i+1},\dots,\pi^n\bigr)_x\, dx\biggr| \\
	& \leq 2n\bV^0(u) \ .
	\end{align*}
	Hence $\bV(u) \leq 2n\bV^0(u)$. This shows that $\BV(U) = \BV^0(U)$ with the estimates on the variations as stated.
\end{proof}

\noindent The following lower semicontinuity result is immediate.

\begin{lem}
	\label{lowersemicont_lem}
	Let $U \subset \R^n$ be an open subset and $\alpha\in [0,1]$. If $(u_k)_{k \in \N}$ is a sequence in $\BV^\alpha(U)$ that converges to $u \in L^1(U)$ weakly (in $L^1$) on compact subsets of $U$, then $\|u\|_{L^1} \leq \liminf_{k\to\infty}\|u_k\|_{L^1}$ and $\bV^\alpha(u) \leq \liminf_{k\to\infty}\bV^\alpha(u_k)$.
\end{lem}

\begin{proof}
	Note that $\|u\|_{L^1} = \sup \{|\int_U u v| \,:\, v \in L_c^\infty(U),\, \|v\|_\infty\leq 1\}$ and that $x \mapsto \det D\varphi_x$ is in $L_c^\infty(U)$ if $\varphi \in \Lip_c(U) \times \Lip(U)^{n-1}$. So both $\|u\|_{L^1}$ and $\bV^\alpha(u)$ are defined as the supremum over some set of test functions. The lower semicontinuity is therefore immediate.
\end{proof}

\noindent Note that in case $u \in \BV_c^\alpha(\R^n)$ we can drop the compactness assumption on $f$ in the definition of $\bV^\alpha(u)$. This can be seen by modifying $f \in \Lip^\alpha(f) \cap \Lip(f)$ outside the support of $u$. Actually, if $u \in L^1_c(\R^n)$, then
\[
\bV^\alpha(u) = \sup \biggl|\int_{U} u(x) \det D(f,g^1,\dots,g^{n-1})_x \, dx \biggr| \ ,
\]
where the supremum is taken over all functions $(f,g^1,\dots,g^{n-1}) \in \Lip(\R^n)^n$ with $\Lip^{\alpha}(f|_{\spt(u)}) \leq 1$ and $\Lip(g^i|_{\spt(u)})\leq 1$ for $i=1,\dots,n-1$. It is noted in \eqref{inclusionholder} that $\Hol^\alpha(f) \leq \Hol^\beta(f)\diam(\spt(u))^{\beta-\alpha}$ for $0 \leq \alpha \leq \beta \leq 1$. This implies continuous inclusions
\begin{equation}
\label{bvinclusion}
\BV_c(\R^n) = \BV_c^0(\R^n) \subset \BV_c^\alpha(\R^n) \subset \BV_c^\beta(\R^n) \quad \mbox{for} \quad 0 \leq \alpha \leq \beta \leq 1 \ ,
\end{equation}
with bounds on the variations depending on $\diam(\spt(u))$ and the corresponding exponents. This dependence on $\diam(\spt(u))$ and the restriction to compactly supported functions in the inclusions above can be avoided by choosing a different definition of $\bV^\alpha$. We could replace the seminorm $\Lip^\alpha(f)$ in the definition of $\bV^\alpha$ \eqref{bvalpha_def} by the genuine norm
\[
|f|_{\alpha} \defl  \max\bigl\{\|f\|_\infty, 2^{\alpha-1} \Hol^\alpha(f) \bigr\} \ .
\]
Then $0 \leq \alpha \leq \beta \leq 1$ implies $\bV^\alpha(u) \geq \bV^\beta(u)$ for all $u \in L^1(U)$ and the inclusions in \eqref{bvinclusion} hold without assuming that $u$ has compact support. Qualitatively all the results we mention below hold true if we make this change, particularly the main result Theorem~\ref{equivalence_thm}, but for the applications in the last section we get better quantitative bounds, respectively, we obtain them more directly with the definition we have chosen in \eqref{bvalpha_def}. Our definition seems also natural because of the observation that if $f$ is constant equal to $c \neq 0$, then $\int_{\R^n} u(x) \det D(f,g^1,\dots,g^{n-1})_x \, dx = 0$ and $\Lip^\alpha(f|_{\spt(u)}) = 0$ but $|f|_{\alpha}$ is nonzero.

\medskip

\noindent A simple application of the change of variables formula shows that if $\eta_r : \R^n \to \R^n$ is the rescaling $\eta_r(x) \defl rx$ by $r > 0$ and $u \in \BV^\alpha_c(\R^n)$, then $\bV^\alpha(u \circ \eta_r) \leq r^{-(n-1)-\alpha}\bV^\alpha(u)$. Scaling back with $\eta_{r^{-1}}$ implies equality
\begin{equation}
\label{rescaling_eq}
\bV^\alpha(u \circ \eta_r) = r^{-(n-1)-\alpha}\bV^\alpha(u) \ .
\end{equation}
Here is a proof of $\bV^\alpha(u \circ \eta_r) \leq r^{-(n-1)-\alpha}\bV^\alpha(u)$: Let $F \in \Lip(\R^n,\R^n)$ with $\Lip^\alpha(F^1) \leq 1$ and $\Lip(F^i) \leq 1$ for $i = 2,\dots,n$. Due to the change of variables formula
	\begin{align*}
	\biggl|\int_{\R^n}(u\circ\eta_r)(x)&\det D F_x\, dx\biggr| \\
	& = \biggl|\int_{\R^n}(u\circ\eta_r)(x)\det D(F \circ \eta_r^{-1})_{\eta_r(x)}\det D(\eta_r)_x \, dx\biggr| \\
	& = \biggl|\int_{\R^n}u(y)\det D(F \circ \eta_r^{-1})_y \, dy \biggr| \\
	& \leq \bV^\alpha(u)\Hol^\alpha(F^1 \circ \eta_{r^{-1}})\Lip(F^2 \circ \eta_{r^{-1}})\cdots\Lip(F^n \circ \eta_{r^{-1}}) \\
	& = r^{-\alpha - (n-1)}\bV^\alpha(u) \ .
	\end{align*}
In the last line we used that $\Hol^\alpha(f \circ \eta_s) = \Hol^\alpha(f)s^\alpha$ for $s > 0$, $0 \leq \alpha \leq 1$ and $f : \R^n \to \R$.

\subsection{Approximation theorem}

\label{approxthm_subsec}

Below we prove an approximation result for functions in $\BV^\alpha_c(\R^n)$ by functions in $\BV_c(\R^n)$. These approximations are obtained by averaging on dyadic cubes. In order to obtain the bounds on the total variation of these approximations the construction of \cite[Lemma~1]{RY} is used. In \cite{RY} Rivi\`ere and Ye used this is as the elementary starting point to solve the prescribed Jacobian problem for densities of different regularity. The theorem below can be seen as the main technical result of this work and it is also the reason why $\bV^\alpha(u)$ for $u \in L^1(U)$ is defined as it is and not as
\[
\sup \biggl\{\int_U u(x) \operatorname{div} \varphi(x)\, dx \, : \, \varphi \in C_c^1(U,\R^n), \, \Hol^\alpha(\varphi) \leq 1 \biggr\} \ ,
\]
which may seem more appropriate in analogy with the classical definition. It is actually not clear to the author to what extent these definitions are equivalent. The specific use of our definition of $\bV^\alpha(u)$ in the proof below is in estimate \eqref{mainestimate_eq}.

\begin{thm}[Approximation Theorem]
	\label{dyadic_thm}
	For all $n \in \N$ there exists a constant $C = C(n)\geq 0$ with the following property: For any $\alpha \in [0,1[$ and $u \in \BV^\alpha_c(\R^n)$ with $\spt(u) \subset [-r,r]^n$ for some $r > 0$ there is a sequence $(u_k)_{k \geq 0}$ in $\BV_c(\R^n)$ such that:
	\begin{enumerate}
		\item The partial sums of $\sum u_k$ converge to $u$ in $L^1$.
		\item $\spt(u_k) \subset [-r,r]^n$ for all $k \geq 0$.
		\item For $k \geq 0$,
		\begin{equation*}
		\|u_k\|_{L^1} \leq Cr^{1-\alpha}\bV^\alpha(u) 2^{k(\alpha-1)} \quad \mbox{and} \quad \bV(u_k) \leq Cr^{-\alpha}\bV^\alpha(u) 2^{k\alpha} \ .
		\end{equation*}
		\item $u_k = \sum_{R \in \cP_k} a_R \chi_R$, where $a_R \in \R$, $\cP_0 = \{[-r,r]^n\}$ and $\cP_k = \{r2^{1-k}(p + [0,1]^n) : p \in \Z^n\}$ for $k \geq 1$.
	\end{enumerate}
\end{thm}

\begin{proof}
	We first prove the theorem in case that $\spt(u) \subset Q \defl [-1,1]^n$. By definition
	\begin{equation}
	\label{prelim_est}
	\biggl|\int u(x)\det D(f,g^1,\dots,g^{n-1})_x\, dx\biggr| \leq \bV^\alpha(u)\Lip^\alpha(f)\Lip(g^1)\cdots\Lip(g^{n-1})
	\end{equation}
	for all $(f,g^1,\dots,g^{n-1}) \in \Lip(\R^n)^n$.
	
	\medskip
	
	\noindent For $k \geq 0$ let $\cP_k$ be as in the statement and define $v_k \in L_c^1(\R^n)$ by
	\[
	v_k \defl \sum_{R \in \cP_k} \chi_R \frac{1}{\cL^n(R)} \int_R u(x) \, dx \ .
	\]
	The sequence $v_k$ converges in $L^1$ to $u$. This follows from the facts that if $u$ is continuous, then $v_k$ converges uniformly to $u$, the construction of $v_k$ is linear in $u$ and the $L^1$-norm decreases when passing from $u$ to $v_k$. Set $\eta \defl 2^n-\frac{1}{2} \in \ ]1,2^n[$. It follows from \cite[Lemma~1]{RY} and the discussion thereafter that for any given cube $R \subset Q$ in $\cP_1$ there is a bi-Lipschitz map $\varphi_R : Q \to Q$ with $\varphi_R(x) = x$ for $x \in \partial Q$, $\det D\varphi_R = \eta$ almost everywhere on $R$ and $\det D \varphi_R = \eta' \defl \frac{2^n-\eta}{2^n-1}$ almost everywhere on $Q \setminus R$. Note that $\eta + (2^n - 1)\eta' = 2^n$ and hence $\int_{Q} \det D\varphi_R = 2^n = \cL^n(Q)$. Given $R \in \cP_k$ for $k \geq 1$ let $\hat R \in \cP_{k-1}$ be the unique cube that contains $R$. Applying a similarity transformation, there is a bi-Lipschitz map $\varphi_R : \hat R \to \hat R$ as on $Q$ above. It is crucial to note that all these maps have a common Lipschitz constant $L = L(n) \geq 1$.
	
	\medskip
	
	\noindent For $R \in \cP_k$ let $R_1,\dots,R_{2^n} \in \cP_{k+1}$ be an enumeration of the subcubes of $R$ and choose some $R' \in \{R_1,\dots,R_{2^n}\}$ with
	\[
	\int_{R'} v_{k+1}(x) - v_k(x) \, dx = \sup_{1 \leq i \leq 2^n} \int_{R_i} v_{k+1}(x) - v_k(x) \, dx \ .
	\]
	It holds that 
	\begin{equation}
	\label{Rprimeest}
	2^{n+1} \int_{R'} v_{k+1}(x) - v_k(x) \, dx \geq \int_{R} |v_{k+1}(x) - v_k(x)| \, dx \ .
	\end{equation}
	To see this note first that $\int_R v_{k+1} - v_k = 0$ and let $J \subset \{1,\dots, 2^n\}$ be the (nonempty) subset with $\int_{R_j} v_{k+1}(x) - v_k(x) \geq 0$ for $j \in J$. Then
	\[
	\frac{1}{2}\int_R |v_{k+1} - v_k| = \sum_{j \in J} \int_{R_j} v_{k+1} - v_k \leq (\# J) \int_{R'} v_{k+1} - v_k \ ,
	\]
	and this implies \eqref{Rprimeest}.
	
	\medskip 
	
	\noindent For $k \geq 0$ define $\varphi_k : Q \to Q$ to be equal $\varphi_{R'}$ on any $R \in \cP_k$. This makes sense because $\varphi_{R'}$ is the identity on $\partial R$. For any $k \geq 0$ the following properties hold:
	\begin{enumerate}[$\quad$(a)]
		\item $\int v_k \det D\varphi_k = \int v_k$.
		\item $\int v_{k+1} \det D\varphi_k = \int u \det D\varphi_k$.
		\item $\int (v_{k+1} - v_k)\det D\varphi_k \geq 2^{-n-1} (\eta - \eta') \int |v_{k+1} - v_k|$.
	\end{enumerate}
	Statement (a) follows from the observation that for any $R \in \cP_k$ it holds that $\int_R \det D\varphi_k = \cL^n(R)$ and that $v_k$ is constant on $R$. (b) is a consequence of the fact that $\det D\varphi_k$ is essentially constant on any $R \in \cP_{k+1}$ and $\int_R u = \int_R v_{k+1}$ by construction. Because of \eqref{Rprimeest} and $\int_{R} v_{k+1} - v_k = 0$ for $R \in \cP_k$ we get
	\begin{align*}
	\int_{Q} (v_{k+1} - v_k)\det D\varphi_k & = \sum_{R \in \cP_k} \int_R (v_{k+1} - v_k)\det D\varphi_k \\
	& = \sum_{R \in \cP_k} \eta \int_{R'} (v_{k+1} - v_k) + \eta' \int_{R \setminus R'} (v_{k+1} - v_k) \\
	& = \sum_{R \in \cP_k} (\eta - \eta') \int_{R'} (v_{k+1} - v_k) \\
	& \geq 2^{-n-1} (\eta - \eta') \sum_{R \in \cP_k}  \int_{R} |v_{k+1} - v_k| \\
	& = 2^{-n-1} (\eta - \eta') \int_{Q} |v_{k+1} - v_k| \ .
	\end{align*}
	This shows (c). Together with (a) and (b) we obtain the following crucial integral estimate
	\begin{align}
	\nonumber
	\int_{Q} u(\det D\varphi_k - \det D\id_Q) & = \int_{Q} u(\det D\varphi_k - 1) = \int_{Q} (u\det D\varphi_k - v_k) \\
	\nonumber & = \int_{Q} (v_{k+1} - v_k)\det D\varphi_k \\
	\label{determinantestimate}
	& \geq 2^{-n-1} (\eta - \eta') \int_{Q} |v_{k+1} - v_k| \ .
	\end{align}
	
	\noindent Since $\varphi_k(R) = R$ for $R \in \cP_k$, it follows that
	\[
	\|\varphi_k - \id_{Q}\|_\infty \leq \diam(R) = 2\sqrt{n}2^{-k} \ .
	\]
	Assume first that two points $x,y \in Q$ satisfy $|x-y| > 2^{-k}$. Then
	\begin{align*}
	|\varphi_k(x) + x - \varphi_k(y) - y| & \leq |\varphi_k(x) - x| + |\varphi_k(y) - y| \\
	& \leq 4\sqrt{n}2^{-k} = 4\sqrt{n}2^{-k(1-\alpha)}2^{-k\alpha} \\
	& \leq 4\sqrt{n}2^{-k(1-\alpha)}|x-y|^\alpha \ .
	\end{align*}
	If $0 < |x-y| \leq 2^{-k}$, then due to $\sup_k \Lip(\varphi_k) \leq L$,
	\begin{align*}
	|\varphi_k(x) + x - \varphi_k(y) - y| & \leq (L + 1)|x-y| = (L + 1)|x-y|^{1-\alpha}|x-y|^\alpha \\
	& \leq (L + 1)2^{-k(1-\alpha)}|x-y|^\alpha \ .
	\end{align*}
	Hence $\Lip^\alpha(\varphi_k - \id_{Q}) \leq C_1 2^{-k(1-\alpha)}$ for some constant $C_1 = C_1(n) \geq 0$. Together with \eqref{prelim_est} and \eqref{determinantestimate} this H\"older seminorm estimate implies that
	\begin{align}
	\nonumber
	2^{-n-1} (\eta - \eta') & \int_{Q} |v_{k+1} - v_k| \\
	\nonumber
	& \leq \biggl|\int_{Q} u(\det D\varphi_k - \det D\,\id_{Q})\biggr| \\
	\nonumber
	& = \biggl|\sum_{i=1}^n \int_{Q} u \, \det D(\varphi_k^1,\dots,\varphi_k^{i-1},\varphi_k^{i}-\pi^{i},\pi^{i+1},\dots,\pi^n)\biggr| \\
	\nonumber
	& \leq \sum_{i=1}^n \biggl|\int_{Q} u \, \det D(\varphi_k^{i}-\pi^{i},\varphi_k^1,\dots,\varphi_k^{i-1},\pi^{i+1},\dots,\pi^n)\biggr| \\
	\label{mainestimate_eq}
	& \leq n \bV^\alpha(u) C_1 2^{-k(1-\alpha)}L^{n-1} \ .
	\end{align}
	Therefore $\|v_{k+1} - v_k\|_{L^1} \leq C_2\bV^\alpha(u) 2^{-k(1-\alpha)}$ for some constant $C_2= C_2(n) \geq 0$. The total variation $\bV(v_{k+1} - v_k)$ is now straight forward to estimate. Given $k \geq 0$, the function $v_{k+1} - v_k$ is constant, say equal to $a_R$, on any $R \in \cP_{k+1}$. Because $\|\chi_R\|_{L^1} = 2^{-kn}$ and $\bV(\chi_R) = 2n2^{-k(n-1)}$, we get
	\begin{align*}
	\bV(a_R\chi_R) = |a_R|2n2^{-k(n-1)} = 2n2^{k}\|a_R\chi_R\|_{L^1} \ .
	\end{align*}
	Hence
	\begin{align*}
	\bV(v_{k+1} - v_k) & \leq \sum_{R \in \cP_{k+1}} \bV(a_R\chi_R) = \sum_{R \in \cP_{k+1}} 2n2^{k}\|a_R\chi_R\|_{L^1} \\
	& = 2n2^{k}\|v_{k+1} - v_k\|_{L^1} \\
	& \leq 2nC_2 \bV^\alpha(u) 2^{k\alpha} \ .
	\end{align*}
	Set $u_0 \defl v_0$ and $u_k \defl v_{k}-v_{k-1}$ for $k \geq 1$. If $\spt(u) \subset Q$, then 
	\[
	\|u_0\|_{L^1} = \biggl|\int_Q u\biggr| \leq \bV^\alpha(u)\Hol^\alpha(\pi^1|_Q) \leq 2 \bV^\alpha(u) \ .
	\]
	Similarly, $\bV(u_0) = n|\int_Q u| \leq 2n \bV^\alpha(u)$. This establishes the result in case the support of $u$ is contained in $[-1,1]^n$.
	
	\medskip
	
	\noindent Given $u$ with $\spt(u) \subset [-r,r]^n$, the rescaled function $u \circ \eta_r$, where $\eta_r(x) = rx$, has support in $[-1,1]^n$ with $\bV^\alpha(u\circ \eta_r) = r^{1-n-\alpha}\bV^\alpha(u)$ by \eqref{rescaling_eq}. Using the decomposition $u_k$ for $u \circ \eta_r$ as above, and scaling back we get
	\begin{align*}
	\|u_k \circ \eta_{r^{-1}}\|_{L^1} & \leq r^n \|u_k\|_{L^1} \leq r^nC \bV^\alpha(u \circ \eta_r)2^{k(\alpha-1)}
	= C r^{1-\alpha}\bV^\alpha(u)2^{k(\alpha-1)} \ .
	\end{align*}
	Similarly,
	\begin{align*}
	\bV(u_k \circ \eta_{r^{-1}}) & \leq r^{n-1} \bV(u_k) \leq r^{n-1} C \bV^\alpha(u \circ \eta_r)2^{k\alpha}
	= C r^{-\alpha}\bV^\alpha(u)2^{k\alpha} \ .
	\end{align*}
	This concludes the proof.
\end{proof}

\noindent In Proposition~\ref{reverse_prop} we state a partial converse to this theorem. This means that given a sequence $(u_k)_{k \geq 0}$ in $\BV_c(\R^n)$ that satisfy (1), (2) and (3) of the theorem above (with $\bV(u)$ replaced by some constant $V \geq 0$), then the sum $u = \sum u_k$ is in $\BV_c^\beta(\R^n)$ for all $\beta > \alpha$. But $u$ may not be in $\BV^\alpha_c(\R^n)$ as we will see in Example~\ref{counterbv_ex}.

\subsection{Compactness and higher integrability}

\label{consequences_subsec}

As a consequence of Theorem~\ref{dyadic_thm} we can generalize the $L^1$-compactness theorem of $\BV$-functions to $\BV^\alpha$-functions.

\begin{prop}[Compactness in $\BV_c^\alpha(\R^n)$]
	\label{compactness_prop}
	Let $\alpha \in [0,1[$ and $(u_k)_{k \in \N}$ be a sequence in $\BV_c^\alpha(\R^n)$ with $\sup_{k\in \N} \|u_k\|_{L^1} + \bV^\alpha(u_k) < \infty$ and $\bigcup_{k \in \N}\spt(u_k) \subset K$ for some compact $K \subset \R^n$. Then there exists a subsequence of $(u_k)_{k \in \N}$ that converges in $L^1$ to some $u \in \BV_c^\alpha(\R^n)$ with $\bV^\alpha(u) \leq \liminf_{k\to\infty}\bV^\alpha(u_k)$.
\end{prop}

\begin{proof}
	Up to taking a subsequence we may assume that $\lim_{k\to\infty}\bV^\alpha(u_k)$ exists. Let $r > 0$ be such that $K \subset [-r,r]^n$ and set $V \defl \sup_{k\geq 0} \bV^\alpha(u_k) < \infty$. From Theorem~\ref{dyadic_thm} we obtain functions $u_{k,l} \in \BV(\R^n)$ for $k \in \N$ and integers $l \geq 0$ with $\spt(u_{k,l})\subset [-r,r]^n$, $\sum_{l\geq 0} u_{k,l} = u_k$ in $L^1$,
	\begin{equation}
	\label{uniform_bounds}
	\|u_{k,l}\|_{L^1} \leq CV2^{l(\alpha-1)} \quad \mbox{and} \quad \bV(u_{k,l}) \leq CV2^{l\alpha} \ ,
	\end{equation}
	for some constant $C = C(n,\alpha,r) \geq 0$. Using $L^1$-compactness of $\BV(\R^n)$, see e.g.\ \cite[Theorem~3.23]{AFP}, and a diagonal argument we obtain a subsequence $(k(i))_{i\in\N}$ such that $(u_{k(i),l})_{i \in \N}$ converges in $L^1$ to some $v_l \in L^1(\R^n)$ for each $l \geq 0$. Due to the bound \eqref{uniform_bounds} and lower semicontinuity in $\BV(\R^n)$, this limit satisfies
	\begin{equation}
	\label{uniform_bounds2}
	\|v_{l}\|_{L^1} \leq CV2^{l(\alpha-1)} \quad \mbox{and} \quad \bV(v_{l}) \leq CV2^{l\alpha} \ ,
	\end{equation}
	for all $l \geq 0$. From the first bound it is clear that $\sum_l v_l$ converges in $L^1$ to some $u \in L^1(\R^n)$ with $\spt(u) \subset [-r,r]^n$. Fix $l_0 \geq 0$. It follows from \eqref{uniform_bounds} and \eqref{uniform_bounds2} that
	\begin{align*}
	\limsup_{i\to\infty}\|u_{k(i)} - u\|_{L^1} & \leq \limsup_{i\to\infty}\sum_{l \geq 0}\|u_{k(i),l}-v_l\|_{L^1} \\
	& \leq \limsup_{i\to\infty} \sum_{l > l_0} \bigl(\|u_{k(i),l}\|_{L^1} + \|v_l\|_{L^1}\bigr) + \sum_{0 \leq l \leq l_0}\|u_{k(i),l}-v_l\|_{L^1} \\
	& \leq \sum_{l > l_0}2CV2^{l(\alpha-1)} + \sum_{0 \leq l \leq l_0}\limsup_{i\to\infty}\|u_{k(i),l}-v_l\|_{L^1} \\
	& = \frac{2CV}{1-2^{\alpha-1}}2^{(l_0+1)(\alpha-1)} \ .
	\end{align*}
	Because this is true for all $l_0 \geq 0$ we see that $\lim_{i\to\infty} u_{k(i)} = u$ in $L^1$. With the lower semicontinuity property, Lemma~\ref{lowersemicont_lem}, we conclude that $u \in \BV_c^\alpha(\R^n)$ with the bound on the variation as stated.
\end{proof}

\noindent The classical embedding result $\BV(\R^n) \hookrightarrow L^\frac{n}{n-1}(\R^n)$ together with the approximation theorem for $\BV_c^\alpha(\R^n)$ implies higher integrability also for this space.

\begin{prop}[Higher integrability]
	\label{higherint_prop}
	Assume that $(u_k)_{k \geq 0}$ is a sequence in $\BV(\R^n)$ that satisfies
	\[
	\|u_k\|_{L^1} \leq V \sigma^{k(\alpha-1)} \ , \quad \bV(u_k) \leq V \sigma^{k\alpha} \ ,
	\]
	for some $\alpha \in \ ]0,1[$, $\sigma > 1$ and $V \geq 0$. Then $u =  \sum_{k \geq 0} u_k$ is in $L^p(\R^n)$ if $1\leq p < \frac{n}{n-1 + \alpha}$. Indeed, $\|u\|_{L^p} \leq C(n,\sigma,p) V$ if $p < \frac{n}{n-1 + \theta}$. In particular, $\BV_c^\alpha(\R^n) \subset L^p_c(\R^n)$ for $1 \leq p < \frac{n}{n-1 + \alpha}$ and the inclusion is compact if restricted to functions with support in some fixed compact set.
\end{prop}

\begin{proof}
	There is a constant $C_n > 0$ such that for each $k \geq 0$ the estimate $\|u_k\|_{L^q} \leq C_n \bV(u_k)$ holds for $q = \frac{n}{n-1}$ if $n > 1$ and $q = \infty$ if $n=1$, see e.g.\ \cite[Theorem~3.47]{AFP}. For any $\theta \in \ ]\alpha,1]$ let $p_\theta \geq 1$ be such that the equation $\frac{1}{p_\theta} = \frac{\theta}{1} + \frac{1-\theta}{q}$ holds. By H\"older interpolation,
	\begin{align*}
	\|u_k\|_{L^{p_\theta}} & \leq \|u_k\|_{L^1}^\theta\|u_k\|_{L^q}^{1-\theta} \leq VC_n^{1-\theta} \sigma^{\theta k(\alpha-1)} \sigma^{(1-\theta)k\alpha} \\
	& = VC_n^{1-\theta} \sigma^{k(\alpha-\theta)} \ .
	\end{align*}
	Hence $\sum_{k \geq 0} \|u_k\|_{L^{p_\theta}}$ is finite. In case $n > 1$ we obtain the boundary value $p_\alpha = \frac{n}{n - 1 + \alpha}$ and similarly $p_\alpha = \frac{1}{\alpha}$ in case $n = 1$.
	
	\medskip
	
	\noindent The last statement of the proposition follows directly from Theorem~\ref{dyadic_thm} and Proposition~\ref{compactness_prop}. Indeed assume that $u \in \BV_c^\alpha(\R^n)$ with $\spt(u)\subset [-r,r]^n$ for some $r > 0$. Fix some $p \in [1, \frac{n}{n-1 + \theta}[$. Pick some $q \in \ ]p, \frac{n}{n-1 + \alpha}[$ and a corresponding $\theta \in \ ]0,1[$ such that $\frac{1}{p} = \frac{\theta}{1} + \frac{1-\theta}{q}$. With the decomposition as in Theorem~\ref{dyadic_thm} it follows as above that
	\[
	\|u\|_{L^{p}} \leq \|u\|_{L^1}^\theta\|u\|_{L^q}^{1-\theta} \leq C'\|u\|_{L^1}^\theta\bV^\alpha(u)^{1-\theta} \ ,
	\]
	for some constant $C'=C'(n,\alpha,p,q,r) \geq 0$. Hence Proposition~\ref{compactness_prop} implies compactness in $L^p$.
\end{proof}

\noindent With this proposition we obtain an isoperimetric type inequality for bounded Borel sets $B \subset \R^n$ with $\chi_B \in \BV_c^\alpha(\R^n)$. This statement may not be optimal since it does not reproduce the isoperimetric inequality for sets of bounded perimeter.

\begin{cor}[Isoperimetric inequality]
	\label{isoperimetric_cor}
	Assume that $B \subset [-r,r]^n$ is a Borel set with $\chi_B \in \BV_c^\alpha(\R^n)$ for some $\alpha \in [0,1[$. Then for all $d \in \ ]n-1+\alpha,n]$,
	\[
	\cL^n(B) \leq C(n,d,\alpha,r) \bV^\alpha(\chi_B)^\frac{n}{d} \ .
	\]
\end{cor}

\begin{proof}
	Theorem~\ref{dyadic_thm} guarantees a decomposition $\chi_B = \sum_{k \geq 0} u_k$ in $L^1$ with 
	\[
	\|u_k\|_{L^1} \leq C'\bV^\alpha(\chi_B)2^{k(\alpha-1)} \mbox{ and } \bV(u_k)_{L^1} \leq C'\bV^\alpha(\chi_B)2^{k\alpha} \ ,
	\]
	for some $C' = C'(n,\alpha,r) \geq 0$. Proposition~\ref{higherint_prop} implies that for $d \in \ ]n-1+\alpha,n]$ it holds that
	\[
	\cL^n(B)^\frac{1}{p} = \|\chi_B\|_{L^p} \leq C(n,d,\alpha,r) \bV^\alpha(\chi_B) \ ,
	\]
	where $p \defl \frac{n}{d}$. This implies the statement.
\end{proof}

\noindent This does not recover the classical isoperimetric inequality for $\alpha = 0$, in which case the inequality also holds for $d = n-1$. But this is not surprising since we already remarked after Theorem~\ref{dyadic_thm} that some information about the critical exponent is lost in the approximation theorem. It thus may be possible that Corollary~\ref{isoperimetric_cor} and also part of Proposition~\ref{higherint_prop} are also valid for the exponent $d = n-1+\alpha$. The compactness of the inclusion in Proposition~\ref{higherint_prop} is sharp though, at least in the classical case $\alpha = 0$.

\section{Fractal currents}

\label{fractalcur_sec}

\noindent To see the connection between functions of fractional bounded variation with metric currents, note that the integral in \eqref{bvalpha_def}, the definition of $\bV^\alpha(u)$, can be expressed as
\begin{align*}
\int_{U} u(x) \det D(f,g^1,\dots,g^{n-1})_x \, dx & = \partial \curr u(f,g^1,\dots,g^{n-1}) \\
& = \curr u(1,f,g^1,\dots,g^{n-1}) \ .
\end{align*}
Due to the correspondence between $\BV$-functions and normal currents, Lemma~\ref{BVnormal_lem}, the approximation result for $\BV^\alpha$- functions, Theorem~\ref{dyadic_thm}, can be formulated as follows: If $u \in \BV^\alpha_c(\R^n)$ for $\alpha \in [0,1[$, there is a sequence $(R_k)_{k \geq 0}$ of normal currents in $\bN_n(\R^n)$ such that $\curr u = \sum_{k \geq 0} R_k$ as a weak limit (in mass actually) with mass bounds
\begin{equation}
\label{decomp_eq}
\bM(R_k) \leq V \rho^{k(\alpha-1)} \quad \mbox{and} \quad \bM(\partial R_k) \leq V \rho^{k\alpha}
\end{equation}
for constants $V \geq 0$ and $\rho > 1$. These mass bounds indicate that $\curr u$ has a particular controlled type of flat approximation by normal currents, compare with \eqref{flat_def1}. The existence of such a decomposition into normal currents for a particular current $T \in \cD_n(\R^n)$ does not need any properties of the ambient space and thus can be formulated in more generality. This is the basic idea behind the definition of fractal currents below.

\subsection{Fractals as fractal currents}

Our definition of a fractal current in an arbitrary metric space $X$ is the following:

\begin{defi}[Fractal currents]
\label{fractal_def}
Let $n \geq 0$ be an integer, $\gamma \in [n,n+1[$ and $\delta \in [n-1,n[$. A current $T \in \cD_n(X)$ is a \textbf{fractal current} in $\bF_{\gamma,\delta}(X)$ if there exists a compact set $K \subset X$, sequences $(R_k)_{k \geq 0}$ in $\bN_{n}(X)$, $(S_k)_{k \geq 0}$ in $\bN_{n+1}(X)$, and parameters $\sigma,\rho > 1$ such that:
\begin{enumerate}
	\item $\bigcup_{k \geq 0} \spt(R_k)\cup\spt(S_k)\subset K$.
	\item The partial sums of $\sum_{k \geq 0} R_k + \partial S_k$ converge weakly to $T$.
	\item 
	\[
\sum_{k \geq 0} \bM(S_k)\sigma^{k((n+1) - \gamma)} < \infty \ , \quad \sum_{k \geq 0} \bM(\partial S_k)\sigma^{k(n - \gamma)} < \infty \ ,
\]
\[
\sum_{k \geq 0} \bM(R_k)\rho^{k(n - \delta)} < \infty \ , \quad \sum_{k \geq 0} \bM(\partial R_k)\rho^{k((n-1) - \delta)} < \infty \ .
\]
\end{enumerate}
\end{defi}

\noindent The guiding principle here is that $T \in \bF_{\gamma,\delta}(X)$ for $\gamma > \dim(\spt(T))$ and $\delta > \dim(\spt(\partial T))$ which we will justify in Lemma~\ref{dimension_lem}. It is straightforward to adapt this definition to chains with coefficients in a normed Abelian group $G$ as defined in \cite{PH}. In this context the approximating sequences of normal currents $R_k$ and $S_k$ are replaced by rectifiable $G$-chains in $\cR_n(X;G)$ and $\cR_{n+1}(X;G)$ respectively. The resulting collection of fractal $G$-chains $\cF_{\gamma,\delta}(X;G)$ (or just $\cF_{\gamma,\delta}(X)$ if $G = \Z$) is then a subclass of flat $G$-chains. It is immediate from the definitions \eqref{flat_def1} and \eqref{flat_def2} and the discussion there that $\bF_{\gamma,\delta}(U)$ and $\cF_{\gamma,\delta}(U)$ are classical flat chains and integral flat chains respectively in case $U \subset \R^m$ is open.

\medskip 

\noindent Note if $u \in \BV^\alpha_c(\R^n)$ for $\alpha \in [0,1[$ has a decomposition $\curr u = \sum_{k \geq 0} R_k$ as in \eqref{decomp_eq}, and if $\beta \in \ ]\alpha,1[$, then
\[
\sum_{k \geq 0} \bM(R_k) \rho^{k(1-\beta)} \leq V \sum_{k \geq 0} \rho^{k(\alpha-\beta)} \leq C(n,\alpha,\beta,\rho)V < \infty \ ,
\]
and similarly $\sum_{k \geq 0} \bM(\partial R_k) \rho^{-k\beta} \leq C(n,\delta,\alpha,\rho)V$. Hence
\begin{equation}
\label{bvinclusion_eq}
\BV^\alpha_c(\R^n) \subset \bF_{n,n-1+\beta}(\R^n) \ .
\end{equation}

\medskip

\noindent One may ask why we adopt a summability condition in the definition of a fractal current instead of bounds similar to \eqref{decomp_eq} that result from the decomposition in Theorem~\ref{dyadic_thm}. As we will see in Example~\ref{counterbv_ex} there is some information lost in Theorem~\ref{dyadic_thm} and more importantly Theorem~\ref{fractalflat_thm} on extensions to H\"older test functions is most general with the summability condition used. A drawback of Definition~\ref{fractal_def} is that $\bF_{\gamma,\delta}(X)$ may not be a vector space unless the parameters $\rho$ and $\sigma$ are fixed.

\medskip

\noindent It is quite clear that $\bF_{n,n-1}(X) = \bN_n(X)$ and $\bF_{\gamma,\delta}(X) \subset \bF_{\gamma',\delta'}(X)$ if $\gamma \leq \gamma'$ and $\delta \leq \delta'$. Further, if $T \in \bF_{\gamma,\delta}(X)$ is $n$-dimensional for $n \geq 1$, then $\partial T \in \bF_{\delta,n-2}(X)$. We use the convention that $\partial T = 0$ if $T \in \cD_0(X)$ is zero dimensional. Whenever $T$ is a zero-dimensional fractal current we assume that $R_0 \in \bM_0(X) = \bN_0(X)$ and $R_k = 0$ for $k \geq 1$ and thus $T \in \bF_{\gamma,-1}(X)$ for some $\gamma \in [0,1[$. With the remark above we see that the boundary operator behaves well in the context of fractal currents. This is also true for other operations on metric currents such as restriction \cite[Definition~2.3]{L}, push forward \cite[Definition~3.6]{L} and slicing \cite[Definition~6.3]{L}.

\medskip

\noindent Slicing is a priori only defined for normal currents. Assume $0 \leq m \leq n$ are integers, $T \in \bF_{\gamma,\delta}(X)$ and $g \in \Lip(X)^m$. In case $(R_k)$ and $(S_k)$ are sequences of normal currents for $T$ as in the Definition~\ref{fractal_def} we define $\langle T, g, y \rangle = \lim_{k\to\infty} \langle R_k + \partial S_k, g, y \rangle$ for $y \in \R^m$ in case this makes sense as a weak limit. The restriction to a Borel set $T \res B$ is in general only defined if $T$ has finite mass \cite[Theorem~4.4]{L}. As for slices above we can define for $x \in X$ and $r > 0$ the restricting $T \res \B(x,r) = \lim_{k \to \infty} (R_k + \partial S_k) \res \B(x,r)$ in case this is well defined as a weak limit. We will not show that the two definitions above are almost everywhere independent on the approximating sequences $(R_k)$ and $(S_k)$. With these definitions we have the following proposition.

\begin{prop}
	\label{slicing_prop}
	Let $0 \leq m \leq n$ be integers, $\gamma \in [n,n+1[$, $\delta \in [n-1,n[$, $T \in \bF_{\gamma,\delta}(X)$. Then:
	\begin{enumerate}
		\item If $\varphi \in \Lip(X,Y)$, then $\varphi_\# T \in \bF_{\gamma,\delta}(Y)$.
		\item If $(f,g) \in \Lip(X)^{m+1}$, then $T \res (f,g) \in \bF_{\gamma-m,\delta-m}(X)$.
		\item If $g \in \Lip(X)^m$, then $\langle T, g, y \rangle \in \bF_{\gamma-m,\delta-m}(X)$ for almost all $y \in \R^m$ and
		\begin{equation}
		\label{sliceformula_eq}
		\int_{\R^m} \langle T, g, y \rangle(f) \, dy = (T \res (1,g))(f)
		\end{equation}
		for all $f \in \Lip(X)^{n+1-m}$.
		\item If $x \in X$, then $T \res \B(x,r) \in \bF_{\gamma,\delta}(X)$ for almost all $r > 0$.
	\end{enumerate}
\end{prop}

\begin{proof}
	Assume that $(R_k)_{k\geq 0},(S_k)_{k\geq 0},K,\sigma,\rho$ are as in the definition of a fractal current such that $T = \sum_{k \geq 0} R_k + \partial S_k$. Statement (1) is clear by simple mass estimates for the push forward. (2) is a consequence of Equations~(4.10) and (5.1) in  \cite{L}. So for all $k \geq 0$ it holds 
	\begin{align*}
		\bM(R_k\res(f,g)) & \leq \|f|_K\|_\infty\Lip(g)^m\bM(R_k) \ ,
	\end{align*}
	and
	\begin{align*}
		\bM(\partial (R_k\res(f,g))) & \leq \Lip(g)^{m}\bigl(\Lip(f) \bM(R_k) + \|f|_K\|_\infty \bM(\partial R_k)\bigr) \ ,
	\end{align*}
	with similar estimates for $\bM(S_k\res(1,f,g))$ and $\bM(\partial (S_k\res(1,f,g)))$. It follows from \cite[Theorem~6.4]{L} that for almost all $y \in \R^m$ the slice $\langle R_k, g, y \rangle$ is an element of $\bN_{n-m}(X)$ for all $k \geq 0$. Moreover,
	\begin{equation*}
		\sum_{k\geq 0} \int_{\R^m} \bM(\langle R_k, g, y \rangle) \rho^{k(n - \delta)} \, dy \leq \sum_{k\geq 0}\Lip(g)^m \bM(R_k) \rho^{k(n - \delta)} < \infty \ .
	\end{equation*}
	With the monotone convergence theorem this implies
	\begin{align*}
		\int_{\R^m} \sum_{k \geq 0} \bM(\langle R_k, g, y \rangle) \rho^{k(n - \delta)} \, dy & = \sum_{k \geq 0} \int_{\R^m} \bM(\langle R_k, g, y \rangle) \rho^{k(n - \delta)} < \infty \ .
	\end{align*}
	Since $y \mapsto \sum_{k \geq 0} \bM(\langle R_k, g, y \rangle) \rho^{k(n - \delta)}$ has a finite integral, the function itself hast to be finite almost everywhere, i.e.\ 
	\[
	\sum_{k\geq 0} \bM(\langle R_k, g, y \rangle) \rho^{k(n-m - (\delta-m))} < \infty
	\]
	for almost all $y$. Because $\langle \partial R_k, g, y \rangle = (-1)^m \partial \langle R_k, g, y \rangle$ by \cite[Equation~(6,9)]{L}, we similarly conclude that
	\[
	\sum_{k\geq 0} \bM(\partial\langle R_k, g, y \rangle) \rho^{k(n-1-m - (\delta-m))} < \infty
	\]
	for almost all $y \in \R^m$. The same reasoning applies to the sequence $(S_k)$. This shows that for almost all $y \in \R^m$,
	\[
	\langle T, g, y \rangle = \sum_{k \geq 0} \langle R_k, g, y \rangle + (-1)^m\partial\langle S_k, g, y \rangle \in \bF_{\gamma-m,\delta-m}(X) \ .
	\]
	The additional integral identity \eqref{sliceformula_eq} follows from Lebesgue’s dominated convergence theorem and the corresponding identities for $R_k$ and $S_k$, \cite[Theorem~6.4(2)]{L}.
	
	\medskip
	
	\noindent (4): For any normal current $W$ it holds that $\partial (W\res \B(x,r)) = (\partial W) \res \B(x,r) + \langle W, d_x, r\rangle$ for almost all $r > 0$ by Definition~6.1 and Theorem~6.2 of \cite{L}. So, for almost all $r > 0$,
	\begin{align*}
	\sum_{k \geq 0} \bM(R_k \res \B(x,r))\rho^{k(n-\delta)} & \leq \sum_{k \geq 0} \bM(R_k)\rho^{k(n-\delta)} < \infty \ , \\
	\sum_{k \geq 0} \bM(\partial (R_k \res \B(x,r)))\rho^{k(n-1-\delta)} & \leq \sum_{k \geq 0} \bigl(\bM(\partial R_k)  + \bM(\langle R_k, d_x, r\rangle) \bigr)\rho^{k(n-1-\delta)} \\
	 & < \infty \ .
	\end{align*}
	Similar estimates hold for $S_k \res \B(x,r)$ too, and this shows (4).
\end{proof}

\noindent Push forwards of certain fractal currents with respect to H\"older maps are treated in Section~\ref{push_section}.

\medskip

\noindent Before proving a more general statement, it is shown that the Koch snowflake domain induces a fractal current. Similar fractals can be treated alike.

\begin{ex}
	\label{vonKoch1}
	The Koch snowflake domain is a compact subset $K \subset \R^2$ with boundary $\partial K$ of Hausdorff dimension $d \defl \frac{\log(4)}{\log(3)}$. $K$ can be written as the closure of the union $\bigcup_{k \geq 0} K_k$ where $K_0$ is an equilateral triangle of area $a_0$ and $K_k$ consists of $3 \cdot 4^{k-1}$ disjoint equilateral triangles with area $a_0 3^{-2k}$ for $k \geq 1$. Thus for $k \geq 1$
	\[
	\cL^2(K_k) = 3 \cdot 4^{k-1} a_0 3^{-2k} = \frac{3a_0}{4} \biggl(\frac{4}{3^2}\biggr)^{k} = \frac{3a_0}{4} 3^{k(d - 2)} \ ,
	\]
	and similarly, if $v_0$ is the perimeter of $K_0$, then the perimeter of $K_k$ is
	\[
	\bV(\chi_{K_k}) = 3\cdot 4^{k-1} v_0 3^{-k} = \frac{v_0}{4} \biggl(\frac{4}{3}\biggr)^{k} = \frac{v_0}{4} 3^{k(d - 1)} \ .
	\]
	If $\delta > d$, then for some $C > 0$,
	\[
	\sum_{k \geq 0} \|\chi_{K_k}\|_{L^1} 3^{k(2-\delta)} \leq C\sum_{k \geq 0} 3^{k(d - \delta)} < \infty \ ,
	\]
	and
	\[
	\sum_{k \geq 0} \bV(\chi_{K_k}) 3^{k(1-\delta)} \leq C\sum_{k \geq 0} 3^{k(d - \delta)} < \infty \ .
	\]
	Hence $\curr K \in \cF_{2,\delta}(\R^2)$ and $\partial \curr K \in \cF_{\delta,0}(\R^2)$ for any $\delta \in \ ]d,2[$.
\end{ex}

\noindent This example generalizes to domains with boundaries of a given box counting dimension. If $A \subset \R^n$ is a bounded set and $\epsilon > 0$, let $N_A(\epsilon)$ be the minimal number of balls of radius $\epsilon$ needed to cover $A$. The \textbf{box counting dimension} of $A$ is defined by
\[
\dim_{\boxd}(A) \defl \lim_{\epsilon \downarrow 0} \frac{\log(N_A(\epsilon))}{\log(1/\epsilon)} \ ,
\]
in case this limit exists. Assume that $U \subset \R^n$ is bounded and open with a Whitney decomposition $\mathcal W$, see e.g.\ \cite[Chapter~VI, Theorem~1]{S} for its definition and existence. For any integer $k$ let $\mathcal W_k$ be the cubes in $\mathcal W$ of side length $2^{-k}$. Decomposing each cube in $\mathcal W_k$ for $k < 0$, we may assume that $\mathcal W = \bigcup_{k \geq 0} \mathcal W_k$. It is noted for example in the proof of \cite[Lemma~2]{HN} that $\#\mathcal W_k \leq C(n)N_{\partial U}(2^{-k})$ for $k \geq 1$. If $\delta > \dim_{\boxd}(\partial U)$, then $\log_2(N_{\partial U}(2^{-k})) \leq \delta\log_2(2^k) = \delta k$ for all $k$ big enough. It follows that there is a constant $C'(n,\delta,U) \geq 0$ such that
\begin{equation}
\label{whitney_bound}
\#\mathcal W_k \leq C'(n,\delta,U)2^{k\delta}
\end{equation}
for all $k \geq 0$.

\begin{lem}
	\label{dimension_lem}
	Assume that $U \subset \R^n$ is bounded and open with $\dim_{\boxd}(\partial U) < n$. For all $\delta \in \ ]\dim_{\boxd}(\partial U),n[$ there is a constant $C(n,\delta,U)\geq 0$ and compact sets $R_k \subset U$ with 
	\begin{enumerate}
		\item $\chi_U = \sum_{k \geq 0} \chi_{R_k}$ almost everywhere and in $L^1$,
		\item $\cL^n(R_k) \leq C(n,\delta,U)2^{k(\delta-n)}$ and $\bV(\chi_{R_k}) \leq C(n,\delta,U)2^{k(\delta-(n-1))}$.
	\end{enumerate}
	In particular, $\curr U \in \cF_{n,\delta}(\R^n)$ and $\partial \curr U \in \cF_{\delta,n-2}(\R^n)$ for all $\delta \in \ ]\dim_{\boxd}(\partial U),n[$.
\end{lem}

\begin{proof}
	Set $R_k \defl \bigcup \mathcal W_k$. (1) is clear since $\mathcal W$ is composed of countably many essentially disjoint closed cubes with $\bigcup \mathcal W = U$. Since $\cL^n(R_k) = (\#\mathcal W_k)2^{-kn}$ and $\bV(\chi_{R_k}) \leq 2n(\#\mathcal W_k)2^{-k(n-1)}$ we obtain with \eqref{whitney_bound} that $\cL^n(R_k) \leq C'2^{k\delta}2^{-kn}$ and also $\bV(\chi_{R_k}) \leq 2nC'2^{k\delta}2^{-k(n-1)}$ for all $k \geq 0$.
\end{proof}

\noindent More generally we obtain that $\curr U \in \cF_{n,d}(\R^n)$ in case $\partial U$ is $d$-summable as defined in \cite{HN}. This is contained in the proof of \cite[Lemma~2]{HN}, where it is observed that $\partial U$ is $d$-summable if and only if $\sum_{k \geq 0} N_{\partial U}(2^{-k}) 2^{-kd} < \infty$. With $\#\mathcal W_k \leq C(n)N_{\partial U}(2^{-k})$ for $k \geq 1$, $\bM(R_k) = (\#\mathcal W_k)2^{-kn}$ and $\bM(\partial R_k) \leq 2n(\#\mathcal W_k)2^{-k(n-1)}$ the statement follows. Similar conclusions can be drawn using the generalization of $d$-summability introduced in \cite[Theorem~2.2]{Gus}. Indeed, the condition on $U$ in \cite[Theorem~2.2]{Gus} immediately implies that $\curr U \in \cF_{n,d}(\R^n)$. Therefore Theorem~\ref{fractalflat_thm} below generalizes the extension results for H\"older differential forms \cite[Theorem~A]{HN} and \cite[Theorem~2.2]{Gus}.

\subsection{Extension theorem}

First we state an extension result for fractal currents. It builds on \cite[Theorem~4.3]{Z} and on the bound \eqref{extension_bound} obtained in the proof thereof.

\begin{thm}[Extension~theorem]
\label{fractalflat_thm}
Let $(X,d)$ be a metric space and $n \geq 0$ be an integer. If $T \in \bF_{\gamma,\delta}(X)$ for some $\gamma \in \ ]n,n+1[$ and $\delta \in \ ]n-1,n[$, then $T$ has a unique continuous extension to a H\"older current
\[
\bar T : \Hol^{\alpha_1}(X) \times \dots \times \Hol^{\alpha_{n+1}}(X) \to \R
\]
whenever $\alpha_1 + \dots + \alpha_{n+1} \geq \gamma$ and $\alpha_2 + \dots + \alpha_{n+1} \geq \delta$ (in case $n \geq 1$). Moreover, if $(R_k)$ and $(S_k)$ are approximating sequences for $T$ as in Definition~\ref{fractal_def} with parameters $\rho$ and $\sigma$, then for all $f = (f^1,\dots,f^{n+1}) \in \Hol^{\alpha_1}(X) \times \dots \times \Hol^{\alpha_{n+1}}(X)$,
\begin{align*}
\biggl|\sum_{k \geq 0} R_k(f)\biggr| & \leq C\|f^1\|_{\alpha_1}H_{n}\sum_{k \geq 0} \bM(R_k)\rho^{k(n - \delta)} + \bM(\partial R_k) \rho^{k((n-1) - \delta)} \ , \\
\biggl|\sum_{k \geq 0} \partial S_k(f)\biggr| & \leq CH_{n+1}\sum_{k \geq 0} \bM(S_k)\sigma^{k((n+1) - \gamma)} + \bM(\partial S_k) \sigma^{k(n - \gamma)} \ ,
\end{align*}
where $C = C(n,\gamma,\delta) \geq 0$, $H_n \defl \prod_{i=2}^{n+1} \Hol^{\alpha_i}(f^i)$ ($H_n=1$ in case $n=0$), $H_{n+1} \defl H_n \Hol^{\alpha_1}(f^1)$ and $\|f^1\|_{\alpha_1} \defl \|f^1\|_\infty + \Hol^{\alpha_1}(f^1)$.
\end{thm}

\begin{proof}
Without loss of generality we assume that $\gamma \defl \alpha_1 + \dots + \alpha_{n+1}$ and $\delta \defl \alpha_2 + \dots + \alpha_{n+1}$ due to the fact that $\bF_{\gamma,\delta}(X) \subset \bF_{\gamma',\delta'}(X)$ in case $\gamma \leq \gamma'$ and $\delta \leq \delta'$. Note that we use the convention that $R_k = 0$ for $k \geq 1$, $\delta = -1$ and $T \in \bF_{\gamma,-1}(X)$ in case $n=0$. From Lemma~\ref{approximation_lem} and the remark thereafter it follows that for any $\epsilon \in \ ]0,1]$ and $i \in \{1,\dots,n+1\}$ there are approximations $f^i_\epsilon \in \Lip(X)$ such that $\Lip(f^i_\epsilon) \leq \Hol^{\alpha_i}(f^i)\epsilon^{\alpha_i-1}$, $\|f^i_\epsilon - f^i\|_\infty \leq \Hol^{\alpha_i}(f^i)\epsilon^{\alpha_i}$ and $\|f^i_\epsilon\|_\infty \leq \|f^i\|_\infty$. With Lemma~\ref{mass_lem},
\[
|S_k(1,f_\epsilon)| \leq \bM(S_k)H_{n+1}\epsilon^{\gamma - (n+1)} \ .
\]
Hence,
\begin{align*}
\sum_{k \geq 0}|S_k(1,f_{\sigma^{-k}})| & \leq H_{n+1}\sum_{k \geq 0} \bM(S_k)\sigma^{k((n+1) - \gamma)} < \infty \ .
\end{align*}
Similarly,
\begin{align*}
\sum_{k \geq 0} |R_k(f_{\rho^{-k}})| & \leq \sum_{k \geq 0} \bM(R_k)\|f^1_{\rho^{-k}}\|_\infty \prod_{i=2}^{n+1} \Lip(f^i_{\rho^{-k}}) \\
	& \leq \|f^1\|_\infty H_n \sum_{k \geq 0} \bM(R_k)\rho^{k(n-\delta)} < \infty \ .
\end{align*}
As recalled in \eqref{extension_bound}, for any $k \geq 0$ it holds that
\begin{align*}
|R_k(f) & - R_k(f_{\rho^{-k}})| \\
	& \leq C(n,\gamma)H_n\bigl[\bM(R_k) \Lip^{\alpha_1}(f^1)\rho^{k(n - \gamma)} + \bM(\partial R_k) \|f^1\|_\infty \rho^{k((n-1) - \delta)}\bigr] \ ,
\end{align*}
and similarly
\begin{align*}
|\partial S_k(f) - \partial S_k(f_{\sigma^{-k}})| \leq C(n,\gamma)\bM(\partial S_k) H_{n+1} \sigma^{k(n - \gamma)} \ .
\end{align*}
These two differences are summable in $k$ and hence with the estimates above we see that $T$ has an extension with bounds as in the statement. As in the proof of \cite[Theorem~4.3]{Z} it can be shown that this extension satisfies all the axioms of a H\"older current as defined above. This uses the additional Lipschitz approximation properties of Lemma~\ref{approximation_lem}.
\end{proof}

\subsection{Fractal currents have bounded fractional variation}

\label{fractalcurr_subsec}

As stated at the end of Subsection~\ref{approxthm_subsec} there is a partial converse to Theorem~\ref{dyadic_thm}. This builds directly on the extension theorem for fractal currents above in the special case where $X = \R^n$.

\begin{prop}
	\label{reverse_prop}
	Let $\beta \in \ ]0,1[$ and assume that $(u_k)_{k \geq 0}$ is a sequence in $\BV_c(\R^n)$ such that $\bigcup_{k\geq 0}\spt(u_k)$ is bounded and there are constants $V \geq 0$ and $\rho > 1$ such that for all $k \geq 0$,
	\begin{align*}
	\sum_{k \geq 0} \|u_k\|_{L^1}\rho^{k(1-\beta)} \leq V \quad \mbox{and} \quad \sum_{k \geq 0} \bV(u_k)\rho^{-k\beta} \leq V \ .
	\end{align*}
	Then $u \defl \sum_{k \geq 0} u_k$ is in $\BV_c^\beta(\R^n)$ and satisfies $\bV^{\beta}(u) \leq C(n,\beta,\rho) V$. In particular, $\bF_{n,n - 1 + \beta}(\R^n) \subset \BV_c^\beta(\R^n)$.
\end{prop}

\begin{proof}
	Set $R_k \defl \curr{u_k}$ which is in $\bN_n(\R^n)$ by Lemma~\ref{BVnormal_lem} with $\bM(R_k)= \|u_k\|_{L^1}$ and $\bM(\partial R_k) = \bV(u_k)$. Since $u = \sum_{k \geq 0} u_k$ in $L^1$ it holds that $\curr u = \sum_{k \geq 0} \curr{u_k}$ in mass. If $F \in \Lip(\R^n)^n$ with $\Hol^{\beta}(F^1) \leq 1$ and $\Lip(F^i) \leq 1$ for $i=2,\dots,n$, it follows from Theorem~\ref{fractalflat_thm} (where $\alpha_2 = \beta$, $\alpha_i=1$ for $i\neq 2$, $\delta = n-1+\beta$ and $\gamma = \delta + 1 = n + \beta$) that
	\begin{align*}
	|\partial \curr u(F)| & \leq \biggl|\sum_{k \geq 0} \partial R_k(F)\biggr| = \biggl|\sum_{k \geq 0} R_k(1,F)\biggr| \\
	& \leq C'(n,\beta) \sum_{k \geq 0} \bM(R_k) \rho^{k(n-\delta)} + \bM(\partial R_k) \rho^{k((n-1)-\delta)} \\
	& \leq C'(n,\beta) \sum_{k \geq 0} \bM(R_k) \rho^{k(1-\beta)} + \bM(\partial R_k) \rho^{-k\beta} \\
	& \leq 2 C'(n,\beta) V \ .
	\end{align*}
	With the definition of $\bV^{\beta}(u)$ the first part of the proposition follows immediately.
	
	\medskip
	
	\noindent For the second part let $T \in \bF_{n,n - 1 + \beta}(\R^n)$. As in Definition~\ref{fractal_def} there is a sequence $(R_k)_{k \geq 0}$ in $\bN_k(\R^n)$ such that $T = \sum_{k\geq 0} R_k$. Note that there is no sequence $(S_k)_{k \geq 0}$ because $\cD_{n+1}(\R^n) = 0$. By Lemma~\ref{BVnormal_lem} we can write $R_k = \curr{u_k}$ for $u_k \in \bN_n(\R^n)$. Because $\sum R_k$ converges in mass to $T$, the sum $\sum u_k$ converges in $L^1$ to some $u \in L^1_c(\R^n)$. Thus $T = \curr u$ and the statement follows from the first part.
\end{proof}

\noindent Together with \eqref{bvinclusion_eq} we immediately obtain that
\begin{equation}
\label{inclusions_eq}
\BV_c^\alpha(\R^n) \subset \bF_{n,n - 1 + \beta}(\R^n) \subset \BV_c^\beta(\R^n)\ ,
\end{equation}
for all $0 \leq \alpha < \beta < 1$. A more in-depth analysis of these inclusions is given later in Theorem~\ref{equivalence_thm}. Note that for $\alpha = 0$, $\BV_c^0(\R^n) = \BV_c(\R^n) = \bN_n(\R^n) = \bF_{n,n - 1}(\R^n)$ by Lemma~\ref{BVnormal_lem}, Lemma~\ref{classicaleq_lem} and the remark after \eqref{bvinclusion_eq}. Apart from this, it is not clear to the author whether $\bF_{n,n - 1 + \beta}(\R^n)$ and $\BV_c^\beta(\R^n)$ are the same or different classes.

\medskip

\noindent Example~\ref{counterbv_ex} of the Koch snowflake domain demonstrates that there is some information lost in Theorem~\ref{dyadic_thm}. Indeed, if $u \in L^1_c(\R^n)$ has an approximation by $\BV$-functions with respect to some exponent $\alpha \in [0,1[$ as stated at the beginning of this section in \eqref{decomp_eq}, then $u$ does not necessarily need to be in $\BV^\alpha_c(\R^n)$. Assuming a summability condition, this holds though as shown in the proposition above.

\begin{ex}
	\label{counterbv_ex}
	Let $\varphi : \partial K \to S^1$ be the inverse of a parametrization of the closed Koch snowflake curve such that
	\[
	L^{-1}|x-y| \leq \angle(\varphi(x),\varphi(y))^\alpha \leq L|x-y|
	\]
	for all $x,y \in \partial K$, where $\alpha = \frac{1}{d} = \frac{\log(3)}{\log(4)}$ is equal to the reciprocal of the Hausdorff dimension of $\partial K$ and $L \geq 1$ is some constant. For $k \geq 1$ define $f_k,g_k : S^1 \to \R$ by
	\[
	f_k(p) \defl \sum_{j=1}^k \frac{1}{2^{j(1-\alpha)}}\cos(2^j\angle(p)) \ , \quad g_k(p) \defl \sum_{j=1}^k \frac{1}{2^{j\alpha}}\sin(2^j\angle(p)) \ .
	\]
	There is a constant $H > 0$ such that $\sup_k\{\Hol^{1-\alpha}(f_k),\Hol^\alpha(g_k)\} \leq H$, see e.g.\ \cite[Theorem~B.6.3]{G2}, and
	\begin{align*}
	\oint_{S^1} f_k\, dg_k & = \sum_{j=1}^k \int_0^{2\pi}\frac{1}{2^{j(1-\alpha+\alpha)}}\cos(2^jt)2^j\cos(2^jt)\, dt \\
	& = \sum_{j=1}^k \int_0^{2\pi} \cos(2^jt)^2 \, dt = \pi k \ .
	\end{align*}
	The functions $f_k \circ \varphi, g_k \circ \varphi : \partial K \to \R$ are Lipschitz because $f_k, g_k$ and $\varphi$ are Lipschitz. Further, for all $x,y \in \partial K$ and $k \geq 1$,
	\begin{align*}
	|f_k(\varphi(x))-f_k(\varphi(y))| & \leq H\angle(\varphi(x),\varphi(y))^{1-\alpha} \leq HL^{\frac{1-\alpha}{\alpha}}|x-y|^{\frac{1-\alpha}{\alpha}} \ ,
	\end{align*}
	and similarly $|g_k(\varphi(x))-g_k(\varphi(y))|\leq HL|x-y|$. Note that $\frac{1-\alpha}{\alpha}=d-1$ and thus $\Hol^{d-1}(f_k\circ \varphi) \leq HL^{d-1}$ and $\Lip(g_k\circ \varphi) \leq HL$. Using a Whitney type extension, see e.g.\ \cite[Chapter~VI Theorem 3]{S}, there are Lipschitz extensions $F_k, G_k : \R^2 \to \R$ of $f_k \circ \varphi$ and $g_k \circ \varphi$ respectively with $\sup_k\{\Hol^{d-1}(F_k), \Lip(G_k)\} < \infty$. Now
	\begin{equation*}
	\int_{K} \det D(F_k,G_k)_x \, dx = \oint_{S^1} (F_k \circ \varphi^{-1}) \, d(G_k \circ \varphi^{-1}) = \oint_{S^1} f_k \, dg_k = \pi k \ .
	\end{equation*}
	The first equation holds because $\varphi^{-1}$ is H\"older continuous with respect to an exponent bigger than $1/2$ and in terms of (metric) currents it follows from $\partial \curr K = (\varphi^{-1})_\#\curr{S^1}$ which is a consequence of \cite[p.~17]{Z}. Thus $\chi_K$ is not in $\BV_c^{d - 1}(\R^2)$ and therefore $\curr K$ is also not in $\bF_{2,d}(\R^n)$ by Proposition~\ref{reverse_prop} above. But $\chi_K$ has a decomposition as in \eqref{decomp_eq} resulting from Theorem~\ref{dyadic_thm} with $\rho = 3$ as stated in Example~\ref{vonKoch1}.
\end{ex}

\subsection{H\"older functions as fractal currents}

Together with Proposition~\ref{reverse_prop} the following lemma shows that H\"older functions are functions of fractional bounded variation.

\begin{lem}
	\label{holderinclusion_lem}
	Assume that $u \in \Hol^{\alpha}(\R^n)$ for some $\alpha \in \ ]0,1[$. Then $\curr {u\chi_{\B(z,r)}} \in \bF_{n,\delta}(\R^n)$ for $z \in \R^n$ and $r > 0$ whenever $\delta + \alpha > n$ (note that $\delta \geq n-1$ by assumption). Indeed, if $u(x_0) = 0$ for some $x_0 \in \B(z,r)$, then there is a constant $C = C(n,r) \geq 0$ and a sequence $(u_k)_{k \geq 0}$ of Lipschitz functions on $\R^n$ such that $u = \sum_{k \geq 0} u_k$ uniformly,
	\begin{align*}
	\|u_k\chi_{\B(z,r)}\|_{L^1} \leq C \Hol^\alpha(u) 2^{-k\alpha} \quad \mbox{and} \quad \bV(u_k\chi_{\B(z,r)}) \leq C \Hol^\alpha(u) 2^{k(1-\alpha)} \ .
	\end{align*}
\end{lem}

\begin{proof}
	We assume $z = 0$ and that $u(x_0)=0$ for some $x_0 \in B \defl \B(0,r)$. Otherwise we would decompose $u$ into a function with this property and a constant function. Because $u(x_0)= 0$ it holds that $\|u|_{B}\|_\infty \leq (2r)^\alpha\Hol^\alpha(u)$. With the McShane-Whitney extension theorem we can extend $u|_{B}$ to a function on $\R^n$ without changing $\Hol^\alpha(u|_{B})$ and $\|u|_{B}\|_\infty$ and such that the support of the extension is contained in $B' \defl \B(0,3r)$. Indeed define $\tilde u : B \cup \R^n \setminus B' \to \R$ to be equal to $u$ on $B$ and zero elsewhere. Note that $\Hol^\alpha(\tilde u) = \Hol^\alpha(u|_{B})$ and $\|\tilde u|_{B}\|_\infty = \|u|_{B}\|_\infty$. Then we can define the extension $\bar u : \R^n \to \R$ of $\tilde u$ by
	\[
	\bar u(x) \defl \min\left\{\|u|_{B}\|_\infty, \max\left\{-\|u|_{B}\|_\infty, \inf_{y \in B \cup \R^n \setminus B'} \tilde u(y) + \Hol^\alpha(\tilde u)|x-y|^\alpha\right\}\right\} \ .
	\]
	We thus can assume that the original function $u$ is already zero on $\R^n\setminus B'$ and satisfies $\|u\|_\infty \leq (2r)^\alpha\Hol^\alpha(u)$. For $\epsilon \in \ ]0,1]$ let $f_\epsilon : \R^n \to \R$ be the Lipschitz approximation of $u$ as defined in Lemma~\ref{approximation_lem} with the additional property that $\|f_\epsilon\| \leq \|u\|_\infty$. It holds $\spt(f_\epsilon) \subset \B(0,3r + \epsilon)$, $\|f_\epsilon\|_\infty \leq \|u\|_\infty \leq (2r)^\alpha\Hol^\alpha(u)$,
	\[
	\|u-f_\epsilon\|_\infty \leq \Hol^\alpha(u)\epsilon^\alpha \quad \mbox{and} \quad \Lip(f_\epsilon) \leq \Hol^\alpha(u)\epsilon^{\alpha-1} \ .
	\]
	Define $u_0 \defl f_1$ and $u_k \defl f_{2^{-k}}-f_{2^{-k-1}}$ for $k \geq 1$. Then $\|u_k\|_\infty \leq 3\Hol^\alpha(u)2^{-k\alpha}$ and $\Lip(u_k) \leq 2 \Hol^\alpha(u)2^{k(1-\alpha)}$ for all $k \geq 1$. Similarly, $\Lip(u_0) = \Lip(f_1) \leq \Hol^\alpha(u)$ and $\|u_0\|_\infty \leq (2r)^\alpha\Hol^\alpha(u)$. If $T \defl \curr{B} \in \bN_n(\R^n)$ and $f : \R^n \to \R$ is Lipschitz, then by \cite[Equation~(5.1)]{L}
	\begin{align*}
	\bV(f\chi_{B}) & = \bM(\partial\curr{f\chi_{B}}) = \bM(\partial(T \res f)) \leq \|f\|_\infty\bM(\partial T) + \Lip(f)\bM(T) \\
	& \leq C_n(\|f\|_\infty r^{n-1} + \Lip(f)r^n) \ .
	\end{align*}
	Thus there is some $C = C(n,r) \geq 0$ such that for all $k \geq 0$
	\[
	\bV(u_k\chi_{B}) \leq C\Lip^\alpha(u) 2^{k(1-\alpha)}
	\]
	and also
	\[
	\|u_k\chi_{B}\|_{L^1} \leq C \Lip^\alpha(u) 2^{-k\alpha} \ .
	\]
	Hence if $\delta > n - \alpha$, then
	\begin{align*}
	\sum_{k \geq 0} \bM(\curr{u_k\chi_{B}}) 2^{k(n - \delta)} & \leq C \sum_{k \geq 0} 2^{k(n - \alpha - \delta)} < \infty \ ,
	\end{align*}
	and
	\begin{align*}
	\sum_{k \geq 0} \bM(\partial\curr{u_k\chi_{B}}) 2^{k((n-1) - \delta)} & \leq C \sum_{k \geq 0} 2^{k(n - \alpha - \delta)} < \infty \ .
	\end{align*}
	This proves that $\curr {u\chi_{B}} \in \bF_{n,\delta}(\R^n)$ with the decomposition as stated.
\end{proof}

\noindent Below is a Lusin type result that can be seen as a partial converse to the lemma above. Similar to the fact that any function in $\BV(\R^n)$ has a measurable decomposition into Lipschitz functions, see e.g.\ \cite[Theorem 5.34]{AFP}, any function of fractional bounded variation has a measurable decomposition into H\"older functions.

\begin{prop}
	\label{almostholder_prop}
	Assume that $u \in L^1_c(\R^n)$ satisfies $\curr u \in \bF_{n,\delta}(\R^n)$ for some $\delta \in [n-1,n[$. Then there is a constant $C=C(n,\delta,u) \geq 0$ and an exhaustion by measurable sets $D_1 \subset D_2 \subset \cdots \subset \R^n$ such that $\cL^n(\R^n\setminus D_k) \leq C k^{-1}$ and
	\[
	|u(x) - u(y)| \leq Ck|x-y|^{n-\delta}
	\]
	for all $x,y \in D_k$.
\end{prop}

\begin{proof}
	Because $\curr u \in \bF_{n,\delta}(\R^n)$, there is a sequence $(R_k)_{k\geq 0}$ in $\bN_n(\R^n)$ such that $\curr u = \sum_{k \geq 0} R_k$ as in Definition~\ref{fractal_def}. Any $R_k \in \bN_n(\R^n)$ can be identified with some $u_k \in \BV_c(\R^n)$ by Lemma~\ref{BVnormal_lem}. Thus $u = \sum_{k \geq 0} u_k$ in $L^1$, and there are constants $V \geq 0$ and $\rho > 1$ such that
	\[
	\sum_{k \geq 0} \|u_k\|_{L^1} \rho^{k(n - \delta)} \leq V \quad \mbox{and} \quad \sum_{k \geq 0} \bV(u_k) \rho^{k(n-1 - \delta)} \leq V \ .
	\]
	Since $\sum_{k \geq 0}\|u_k\|_{L^1} < \infty$, a standard argument in measure theory using the monotone convergence theorem shows that the partial sums of $\sum u_k$ also converge pointwise almost everywhere to $u$.
	
	\medskip 
	
	\noindent For a finite Borel measure $\mu$ on $\R^n$ the maximal function is defined by 
	\[
	M_{\mu}(x) \defl \sup_{r > 0} \frac{\mu(\B(x,r))}{\omega_n r^n} \ .
	\]
	There is a constant $c_n \geq 1$ such that 
	\begin{equation}
	\label{ineqprep}
	\cL^n\bigl(\{x \in \R^n \ :\ M_\mu(x) > s\}\bigr) \leq c_n s^{-1}\mu(\R^n) \ ,
	\end{equation}
	for all $s > 0$, see e.g.\ \cite[Theorem~2.19]{M}, and
	\begin{equation}
	\label{lipschitzestimate}
	|v(x) - v(y)| \leq c_n\bigl(M_{\|Dv\|}(x) + M_{\|Dv\|}(y)\bigr)|x-y|
	\end{equation}
	holds for all Lebesgue points $x,y \in \R^n$ of a function $v \in \BV(\R^n)$, see e.g.\ \cite[Lemma~7.1]{L}. 
	
	\medskip 
	
	\noindent Define $A_{k,s} \defl \{x \in \R^n \, :\, |u_k(x)| > s\rho^{-k(n-\delta)}\}$ and $A_{s} \defl \bigcup_{k \geq 0} A_{k,s}$. By the assumption on the $L^1$-norms,
	\begin{align*}
	\cL^n(A_{s}) & \leq \sum_{k\geq 0} \cL^n(A_{k,s}) \leq \sum_{k\geq 0} \int_{A_{k,s}} |u_k(x)| s^{-1} \rho^{k(n-\delta)} \, dx \\
	 & \leq s^{-1}\sum_{k\geq 0} \|u_k\|_{L^1} \rho^{k(n-\delta)} \leq s^{-1} V \ .
	\end{align*}
	Similarly, set $B_{k,s} \defl \{x \in \R^n \, :\, M_{\|Du_k\|}(x) > s\rho^{k(1-n+\delta)}\}$ and $B_{s} \defl \bigcup_{k \geq 0} B_{k,s}$. By the assumption on the variation and \eqref{ineqprep},
	\begin{align*}
	c_nV & \geq c_n \sum_{k \geq 0} \|Du_k\|(\R^n) \rho^{k(n-1 - \delta)} \\
	& \geq \sum_{k \geq 0} s \cL^n\bigl(\{x \in \R^n \ :\ M_{\|Du_k\|}(x) > s\rho^{k(1-n+\delta)}\}\bigr) \\
	& \geq s \cL^n(B_{s}) \ .
	\end{align*}
	Let $D_{s}$ be the set of all $x \in \R^n \setminus (A_{s} \cup B_{s})$ that are Lebesgue points of $\R^n \setminus (A_{s} \cup B_{s})$, Lebesgue points of $u_k$ for all $k\geq 0$ and satisfy $\lim_{k\to\infty}u_k(x)=u(x)$. With the estimates above for $\cL^n(A_{s})$ and $\cL^n(B_{s})$ it follows $\cL^n(\R^n \setminus D_{s}) \leq 2c_n s^{-1}$. Given points $x,y \in D_{s}$ with $0 < |x-y| \leq 1$ let $l \geq 0$ be such that $\rho^{-(l+1)} < |x-y| \leq \rho^{-l}$. From \eqref{lipschitzestimate} it follows for $\alpha = n-\delta$,
	\begin{align*}
	|u(x) - u(y)| & \leq \sum_{k \geq 0} |u_k(x) - u_k(y)| \\
	& \leq \sum_{k > l} |u_k(x)| + |u_k(y)| + \sum_{0 \leq k \leq l} |u_k(x) - u_k(y)|\\
	& \leq 2s\sum_{k > l} \rho^{-k\alpha} + c_n\sum_{0 \leq k \leq l} \bigl(M_{\|Du_k\|}(x) + M_{\|Du_k\|}(y)\bigr)|x-y| \\
	& \leq 2s\sum_{k > l} \rho^{-k\alpha} + c_n2s|x-y|\sum_{0 \leq k \leq l} \rho^{k(1-\alpha)} \\
	& = \frac{2s}{1-\rho^{-\alpha}}\rho^{-(l+1)\alpha} + c_n2s|x-y|\frac{\rho^{(l+1)(1-\alpha)} - 1}{\rho^{1-\alpha}-1} \\
	& \leq C'(n,\alpha,\rho) s\bigl(|x-y|^{\alpha} + |x-y||x-y|^{\alpha-1}\bigr) \\
	& \leq 2C'(n,\alpha,\rho)s|x-y|^{\alpha} \ .
	\end{align*}
	If $x,y \in D_{s}$ are such that $|x-y| \geq 1$, then similarly
	\begin{align*}
	|u(x) - u(y)| & \leq \sum_{k \geq 0} |u_k(x)| + |u_k(y)| \leq 2s\sum_{k \geq 0} \rho^{-k\alpha} \\
	 & = \frac{2s}{1 - \rho^{-\alpha}} \leq \frac{2s}{1 - \rho^{-\alpha}}|x-y|^\alpha \ .
	\end{align*}
	This proves the statement. 
\end{proof}

\noindent The proposition above contains as a special case that if $\delta = n-1$, then $u$ as a function of bounded variation has a measurable partition into Lipschitz pieces. In this sense Proposition~\ref{almostholder_prop} is best possible with respect to the condition on the exponents.

\subsection{Smoothing and finite mass}

\label{smoothingsection}

The goal of this subsection is to show that any metric current $T \in \cD_n(\R^n)$ whose boundary has an extension to a current in $\cD_{n-1}(\R^n,|\cdot|^\alpha)$ for some $\alpha < 1$ has finite mass, Proposition~\ref{finitemass_prop}. By a result of De Philippis and Rindler \cite[Theorem~1.14]{DR} it then follows that $T = \curr u$ for some $u \in L^1_c(\R^n)$. Proposition~\ref{finitemass_prop} can be seen as positive evidence that any metric current in the sense of Lang living in some $\R^n$ is a locally flat chain. This problem is still open even for $\cD_n(\R^n)$ in case $n > 1$ because these currents are not assumed to have locally finite mass, and therefore \cite[Theorem~1.14]{DR} does not apply directly.

\medskip

\noindent First we need a smoothing result for currents which is technical but straight forward. It is for the most part contained in the proof of \cite[Theorem~4.7]{Z}. For the reader's convenience we repeat the argument here.

\begin{lem}
	\label{approxseq}
	Let $V \defl \Hol^{\alpha_1}(\R^n) \times \cdots \times \Hol^{\alpha_{n+1}}(\R^n)$ for exponents satisfying $\alpha_1 + \cdots + \alpha_{n+1} > n$. Given a H\"older current $T : V \to \R$, define for $\epsilon > 0$
	\[
	T_\epsilon \defl \frac{1}{\omega_n\epsilon^n}\int_{\B(0,\epsilon)} \tau_{x\#} T \,  dx \ ,
	\]
	where $\tau_x(y) = x+y$. Then $T_\epsilon : V \to \R$ is also a H\"older current and
	\begin{enumerate}
		\item $\lim_{\epsilon \downarrow 0} T_\epsilon(f^1,\dots,f^{n+1}) = T(f^1,\dots,f^{n+1})$ for all $(f^1,\dots,f^{n+1}) \in V$,
		\item $\spt(T_\epsilon) \subset \{x \in \R^n : \dist(x,\spt(T)) \leq \epsilon \}$,
		\item $T_\epsilon \in \bN_n(\R^n)$.
	\end{enumerate}
\end{lem}

\begin{proof}
	For simplicity we assume that $\alpha = \alpha_1 = \cdots = \alpha_{n+1}$, i.e.\ $T \in \cD_n(\R^n,|\cdot|^\alpha)$, the general case is proved alike.
	
	\medskip 
	
	\noindent The map $x \mapsto {\tau_x}_{\#}T(f) = T(f\circ \tau_x)$ is continuous for a fixed $f=(f^1,\dots,f^{n+1}) \in \Hol^\alpha(\R^n)^{n+1}$ because of the continuity property of $T$. Hence $T_\epsilon$ is a multilinear functional on $\Hol^\alpha(\R^n)^{n+1}$ that satisfies the locality axiom by definition. $T_\epsilon$ has compact support because $T$ has, and $T_\epsilon$ is continuous as a consequence of Lebesgue's dominated convergence theorem. Indeed for fixed $L \geq 0$ the Arzel\`a-Ascoli theorem and the continuity axiom for metric currents imply that the supremum
	\begin{equation}
	\label{supdef}
	\sup \bigl\{ |T(f)| \, : \, f \in \Hol^\alpha(\R^n)^{n+1}, \, \|f^1\|_\infty \leq L, \, \Hol^\alpha(f^i) \leq L \bigr\}
	\end{equation}
	is attained and therefore finite. Note that because $\spt(T)$ is compact it can be assumed that each $f^2,\dots,f^{n+1}$ in \eqref{supdef} satisfies $f^i(x_0) = 0$ for some fixed $x_0 \in \spt(T)$ and the support of all the functions $f^1,\dots,f^{n+1}$ is contained in some fixed compact set depending on $L$ and $\spt(T)$. So $T_\epsilon$ is indeed a current in $\cD_n(\R^n,|\cdot|^\alpha)$. In order to see that $T_\epsilon$ converges weakly to $T$ note that $\tau_{x\#} T$ converges weakly to $T$ for $x \to 0$. Using Lebesgue's dominated convergence theorem again on the basis that \eqref{supdef} is finite shows that $T_\epsilon$ converges weakly to $T$. This proves (1). Statement (2) is obvious.
	
	\medskip 
	
	\noindent To check the mass bounds of $T_\epsilon$ and $\partial T_\epsilon$ seen as currents in $\cD_n(\R^n)$, note first that
	\begin{align}
	\label{mequal}
	\bM(T_\epsilon) & = \sup \bigl\{T_\epsilon(f,\id) \, :\, f\in \Lipc(\R^n), \, \|f\|_\infty \leq 1 \bigr\} \ , \\
	\label{mestim}
	\bM(\partial T_\epsilon) & \leq \sup \bigl\{n \cdot \partial T_\epsilon(f,\hat \pi_i) \, :\, f \in \Lipc(\R^n), \, \|f\|_\infty \leq 1,\, i=1,\dots,n \bigr\} \ ,
	\end{align}
	where $\hat \pi_i(x_1,\dots,x_n) \defl (x_1,\dots,x_{i-1},x_{i+1},\dots,x_n)$. This follows from the chain rule for currents, see e.g.\ \cite[Theorem 2.5]{L}, and the fact that $C^\infty(\R^n)$ is dense in $\Lip(\R^n)$ (equipped with the weak topology used in the continuity axiom for metric currents). If we set
	\[
	f_\epsilon(y) \defl \frac{1}{\omega_n\epsilon^n} \int_{\B(0,\epsilon)} f(y+x) \, dx
	\]
	for $f \in \Lipc(\R^n)$ it follows that
	\begin{align}
	\label{tequal1}
	T(f_\epsilon,\id_{\R^n}) & = \frac{1}{\omega_n\epsilon^n}\int_{\B(0,\epsilon)} T(f\circ \tau_{x},\id_{\R^n}) \, dx = T_\epsilon(f,\id_{\R^n}) \ , \\
	\label{tequal2}
	\partial T(f_\epsilon,\hat \pi_i) & = \frac{1}{\omega_n\epsilon^n}\int_{\B(0,\epsilon)} \partial T(f\circ \tau_{x},\hat \pi_i) \, dx = \partial T_\epsilon(f,\hat \pi_i) \ ,
	\end{align}
	for $i=1,\dots,n$. This can be seen by approximating the integral by Riemann sums. Next we estimate $\Lip(f_\epsilon)$ in case $\|f\|_\infty \leq 1$. If $0 < |y-z| < 2\epsilon$, then $B_{y,z,\epsilon} \defl \B(\frac{y+z}{2},\epsilon - \frac{|y-z|}{2}) \subset \B(y,\epsilon) \cap \B(z,\epsilon)$. Hence
	\begin{align*}
	|f_\epsilon(y)-f_\epsilon(z)| & = \frac{1}{\omega_n\epsilon^n} \left| \int_{\B(y,\epsilon)} f(x) \, dx - \int_{\B(z,\epsilon)} f(x) \, dx \right| \\
	& \leq \frac{1}{\omega_n\epsilon^n} \int_{\B(y,\epsilon) \Delta \B(z,\epsilon)} |f(x)| \, dx \\
	& \leq \frac{1}{\omega_n\epsilon^n} \bigl(\cL^n(\B(y,\epsilon)) + \cL^n(\B(y,\epsilon)) - 2\cL^n(\B(\tfrac{y+z}{2},\epsilon - \tfrac{|y-z|}{2})\bigr) \\
	& = \frac{2}{\omega_n\epsilon^n}\bigl(\omega_n \epsilon^n - \omega_n\bigl(\epsilon - \tfrac{|y-z|}{2}\bigr)^n\bigr) \\
	& = 2\bigl(1 - \bigl(1 - \tfrac{|y-z|}{2\epsilon}\bigr)^n\bigr) \\
	& \leq \frac{n}{\epsilon}|y-z| \ .
	\end{align*}
	If $|y-z| \geq 2\epsilon$, then $|f_\epsilon(y) - f_\epsilon(z)| \leq 2 \leq \frac{1}{\epsilon}|y-z|$. Hence $\Lip(f_\epsilon) \leq \frac{n}{\epsilon}$. From \eqref{mequal},\eqref{mestim},\eqref{tequal1} and \eqref{tequal2} it follows that $\bM(T_\epsilon) < \infty$ and $\bM(\partial T_\epsilon) < \infty$ with the same reasoning as used above in order to conclude that the supremum in \eqref{supdef} is achieved.
\end{proof}

\noindent The next proposition gives a sufficient condition on a general current in $\cD_n(\R^n)$ to have finite mass. The crucial part in the proof is an application of \cite[Theorem~1]{RY} where it is shown that a continuous density can be realized in a weak sense as the Jacobian determinant of a H\"older map.

\begin{prop}
	\label{finitemass_prop}
	Assume that the boundary of $T \in \cD_n(\R^n)$ has a continuous extension to a H\"older current $\Hol^\alpha(\R^n)^{n} \to \R$ for some $\alpha \in \ ]0,1[$, i.e.\ $\partial T$ extends to an element in $\cD_{n-1}(\R^n,|\cdot|^\alpha)$. Then $\bM(T) < \infty$.
\end{prop}

\noindent (Note that in case $\partial T \neq 0$, then $\alpha$ necessarily has to be in the range $]\frac{n-1}{n},1[$ due to \cite[Theorem~4.3]{Z}. The assumption in this proposition in particular holds if $T$ extends to a current in $\cD_n(\R^n,|\cdot|^\alpha)$ for some $\alpha \in \ ]\frac{n}{n+1},1[$.)

\begin{proof}
	Without loss of generality we can assume that $\spt(T) \subset \operatorname{int}(Q)$ where $Q \defl [0,1]^n$ and that $\alpha > \frac{n-1}{n}$. The next statement is a direct consequence of \cite[Theorem~1]{RY} and the construction of the approximating sequence therein. There are constants $C_\alpha > 0$ and $0 < c_\alpha < \frac{1}{2}$ with the following property: If $f : Q \to [1 - c_\alpha,1 + c_\alpha]$ is a continuous function with $\int_{Q} f(x) \, dx = 1$, then there is a sequence of bi-Lipschitz maps $\varphi_k : Q \to Q$ and a homeomorphism $\varphi : Q \to Q$ with:
	\begin{enumerate}
		\item $\varphi|_{\partial Q} = \varphi_k|_{\partial Q} = \id_{\partial Q}$.
		\item $\sup_k \Hol^\alpha(\varphi_k) \leq C_\alpha$.
		\item $\lim_{k\to\infty} \|\varphi_k - \varphi\|_\infty = 0$.
		\item $(\det D\varphi_k)_{k\in\N}$ converges to $f$ in $L^\infty(Q)$.
		\item for all open sets $E \subset Q$, 
		\[
		\int_{E} f(x) \, dx = \cL^n(\varphi(E)) \ .
		\]
	\end{enumerate}
	Assume that $v \in \BV_c(\operatorname{int}(Q))$. The induced current $\curr{v}$ is in $\bN_{n}(\R^n)$ by Lemma~\ref{BVnormal_lem}. Because of \cite[Theorem~4.3]{Z} and since $\alpha > \frac{n-1}{n}$, the boundary $\partial \curr{v}$ is a normal current and has an extension to an element of $\cD_{n-1}(\R^n,|\cdot|^\alpha)$. From (2),(3) and (4) it follows
	\begin{align}
	\nonumber
	\int_{Q} v(x) f(x) \, dx & = \lim_{k \to \infty} \int_{Q} v(x) \det D\varphi_k(x) \, dx = \lim_{k \to \infty} \partial \curr{v}(\varphi_k) \\
	\label{holderchoice2}
	& = \partial \curr v(\varphi) \ .
	\end{align}
	
	\noindent For $\epsilon > 0$ consider the smoothings $T_\epsilon \in \bN_n(\R^n)$ of $T$ as defined in Lemma~\ref{approxseq}. We can assume that $\epsilon$ is small enough such that $\spt(T_\epsilon) \subset \operatorname{int}(Q)$. As a normal current, $T_\epsilon = \curr {u_\epsilon}$ for some $u_\epsilon \in \BV_c(\operatorname{int}(Q))$ again by Lemma~\ref{BVnormal_lem}. Assume further that $f$ is additionally in $\Lip(Q, [1,1 + c_\alpha])$. Then $c_f \defl \int_Q f \in [1,1 + c_\alpha]$ and hence
	\[
	1-c_\alpha \leq \frac{1}{1 + c_\alpha } \leq \frac{1}{1 + c_\alpha }f \leq \frac{1}{c_f}f \leq f \leq 1+c_\alpha \ .
	\]
	Let $\varphi : Q \to Q$ be the homeomorphism associated to the density $\frac{1}{c_f}f$. From Lemma~\ref{approxseq} and \eqref{holderchoice2} it follows that
	\begin{align}
	\nonumber
	\frac{1}{c_f}T(f,\id_Q) & = \lim_{\epsilon \downarrow 0} \frac{1}{c_f}T_\epsilon(f,\id_Q) = \lim_{\epsilon \downarrow 0} \int_Q u_\epsilon(x) \frac{1}{c_f} f(x) \, dx \\
	\nonumber
	& = \lim_{\epsilon \downarrow 0} \curr{u_\epsilon}(1,\varphi) = \lim_{\epsilon \downarrow 0} T_\epsilon(1,\varphi) \\
	\label{shiftover}
	& = T(1,\varphi) \ .
	\end{align}
	Given a function $g \in \Lip(Q,[-1,1])$ define $f_g \defl 1 + \frac{c_\alpha}{2} + \frac{c_\alpha}{2}g$ which is an element of $\Lip(Q,[1,1+c_\alpha])$. Equation \eqref{shiftover} together with $c_{f} \leq 1+c_\alpha$ implies
	\begin{align*}
	& \sup\bigl\{|T(g,\id_Q)| \, :\, g \in \Lip(Q,[-1,1]) \bigr\} \\
	& \quad \leq \sup \bigl\{ \bigl|T(1 + \tfrac{2}{c_\alpha},\id_Q)\bigr| + \tfrac{2}{c_\alpha}|T(f_g,\id_Q)| \, :\, g \in \Lip(Q,[-1,1]) \bigr\} \\
	& \quad \leq (1 + \tfrac{2}{c_\alpha})|T(1,\id_Q)| + \tfrac{2}{c_\alpha} \sup \bigl\{ |T(f,\id_Q)| \, :\, f \in \Lip(Q,[1,1+c_\alpha]) \bigr\} \\
	& \quad \leq (1 + \tfrac{2}{c_\alpha})|\partial T(\id_Q)| + \tfrac{2(1+c_\alpha)}{c_\alpha}\sup \bigl\{|\partial T(\varphi)| \, :\, \|\varphi\|_\infty \leq \sqrt{n},\, \Hol^\alpha(\varphi) \leq C_\alpha \bigr\} \ .
	\end{align*}
	The supremum in the last line is achieved because of the continuity of $\partial T$ as a current in $\cD_{n-1}(\R^n,|\cdot|^\alpha)$ and the Arzel\`a-Ascoli theorem. Hence $\bM(T) < \infty$.
\end{proof}

\subsection{Equivalent characterizations}

Combined with earlier results we show here different characterizations for functions of bounded fractional variation. The theorem below in particular implies Theorem~\ref{intro_thm} stated in the introduction. Note that $\bF_{n,\delta}(\R^n) = \bF_{\gamma,\delta}(\R^n)$ for all $\gamma \in [n,n+1[$ because the sequence $S_k$ in the definition of $\bF_{\gamma,\delta}(\R^n)$ can be neglected due to $\cD_{n+1}(\R^n) = \{0\}$.

\begin{thm}
	\label{equivalence_thm}
	Let $n \geq 1$ and $T \in \cD_n(\R^n)$ be a metric current (or flat chain in the sense of Whitney \cite{Whi}) and ${d} \in \ ]n-1,n[$. Consider the following statements:
	\begin{enumerate}
		\item $T \in \bF_{n,d}(\R^n)$.
		\item $T$ has a continuous extension to a H\"older current
		\[
		\Hol^{\alpha_1}(\R^n) \times \dots \times \Hol^{\alpha_{n+1}}(\R^n) \to \R \ ,
		\]
		whenever $\alpha_1 + \dots + \alpha_{n+1} > n$ and $\alpha_2 + \dots + \alpha_{n+1} \geq {d}$.
		\item $\partial T$ has continuous extensions to H\"older currents
		\[
		\Hol^\alpha(\R^n)^n \to \R \quad \mbox{for some} \quad \alpha < 1 \ , \mbox{ and }
		\]
		\[
		\Hol^{{d} - (n-1)}(\R^n) \times \Lip(\R^n)^{n-1} \to \R \ .
		\]
		\item $T = \curr u$ for $u \in \BV^{{d} - (n - 1)}_c(\R^n)$.
		\item $T = \curr u$ for some $u \in L_c^1(\R^n)$ and there is a constant $C > 0$ and a sequence $(u_k)_{k \geq 0}$ in $\BV_c(\R^n)$ such that:
		\begin{enumerate}
			\item $\sum_{k\geq 0} u_k = u$ in $L^1$ and $\bigcup_k \spt(u_k)$ is bounded.
			\item $\|u_k\|_{L^1} \leq C 2^{k(d-n)}$ and $\bV(u_k) \leq C 2^{k({d} - (n-1))}$.
		\end{enumerate}
	\end{enumerate}
	Then $(1) \Rightarrow (2) \Rightarrow (3) \Rightarrow (4) \Rightarrow (5)$ and $(5)$ implies that $T \in \bigcap_{{d} < \delta < n}\bF_{n,\delta}(\R^n)$. In particular, $\bigcap_{{d} < \delta < n}\BV^{{\delta} - (n - 1)}_c(\R^n) = \bigcap_{{d} < \delta < n}\bF_{n,\delta}(\R^n)$.
\end{thm}

\begin{proof}
	The implication $(1) \Rightarrow (2)$ is a special case of Theorem~\ref{fractalflat_thm}. 
	
	\medskip 
	
	\noindent $(2) \Rightarrow (3)$ is obvious.
	
	\medskip 
	
	\noindent $(3) \Rightarrow (4)$: It follows from Proposition~\ref{finitemass_prop} that the first extension property of (3) implies $\bM(T) < \infty$. By a result of De Philippis and Rindler \cite[Theorem~1.14]{DR}, the measure $\|T\|$ is absolutely continuous with respect to the Lebesgue measure and hence $T = \curr u$ for some $u \in L^1_c(\R^n)$. With the second extension property of (3) a direct application of the Arzel\`a-Ascoli theorem shows that
	\begin{equation}
	\label{boundarybound2}
	\sup \bigl|\partial\curr u(F)\bigr| < \infty \ ,
	\end{equation}
	where the supremum is taken over all $F \in \Lip(\R^n)^{n}$ with $\Lip^{d-(n-1)}(F^1|_{\spt(u)}) \leq 1$ and $\Lip(F^i|_{\spt(u)}) \leq 1$ for $i=2,\dots,n$. Note that by translation invariance we can assume that $F(0) = 0$. Now \eqref{boundarybound2} is equivalent to $u \in \BV^{d-(n-1)}_c(\R^n)$ and this shows (4).
	
	\medskip 
	
	\noindent $(4) \Rightarrow (5)$ follows from Theorem~\ref{dyadic_thm} by setting $\alpha = {d} + 1 - n$.
	
	\medskip 
	
	\noindent The last statement was shown before \eqref{bvinclusion_eq}. Indeed if $u \in L_c^1(\R^n)$ has a decomposition $\sum_{k\geq 0} u_k$ as in (5), then 
	\[
	\sum_{k \geq 0} \|v_k\|_{L^1}2^{k(n-\delta)} < \infty \quad \mbox{and} \quad \sum_{k \geq 0} \bV(v_k) 2^{k((n-1)-\delta)} < \infty
	\]
	in case $\delta \in \ ]d,n[$.
\end{proof}

\noindent We have already seen in Example~\ref{counterbv_ex} that the implication $(5)\Rightarrow(4)$, thus also $(5)\Rightarrow(1)$, does not hold in general. Note that (3) is a statement purely about the multilinear functional $T$ and does not assume that $T$ can be expressed as an integral over some $u \in L^1_c(\R^n)$, or that $T$ has finite mass for that matter. This is precisely why we needed Proposition~\ref{finitemass_prop}.

\medskip

\noindent Building on the results obtained for fractal currents earlier, we can collect the following properties for functions of bounded fractional variation.

\begin{cor}
	\label{holderfractal_cor}
	The following statements hold:
	\begin{enumerate}
		\item Let $u \in \Hol^\alpha(\R^n)$ for $\alpha \in \ ]0,1]$. Then $u\chi_{\B(0,r)} \in \BV^\beta_c(\R^n)$ for all $r > 0$ whenever $\beta + \alpha > 1$, i.e.\ $\beta \in \ ]1-\alpha,1]$. Moreover, if $x \in \B(0,r)$, then
		\[
		\bV^\beta((u-u(x))\chi_{\B(0,r)}) \leq C(n,\alpha,\beta)r^{\alpha+\beta+n-1} \Hol^\alpha(u|_{\B(0,r)}) \ .
		\]
		\item Let $\alpha,\beta\in \ ]0,1[$. If $u \in \BV^\beta_c(\R^n)$ and $\alpha + \beta < 1$, then there exists $C \geq 0$, an exhaustion by measurable sets $D_1 \subset D_2 \subset \cdots \subset \R^n$ such that $\cL^n(\R^n\setminus D_k) \leq C k^{-1}$ and
		\[
		|u(x) - u(y)| \leq Ck|x-y|^{\alpha}
		\]
		for $x,y \in D_k$.
		\item If $U \subset \R^n$ is bounded and open with $\dim_{\boxd}(\partial U) \in [n-1,n[$, then $\chi_U \in \BV^{\delta - (n-1)}_c(\R^n)$ for all $\delta \in \ ]\dim_{\boxd}(\partial U),n]$.
		\item If $u \in \BV^\alpha_c(\R^n)$ for $\alpha \in [0,1[$ and $x \in \R^n$, then $u\chi_{\B(x,r)} \in \bigcap_{\alpha < \beta < 1} \BV^{\beta}_c(\R^n)$ for almost all $r > 0$.
	\end{enumerate}
\end{cor}

\begin{proof}
	These statements follow directly from Theorem~\ref{equivalence_thm} and the corresponding results for fractal currents: (1) from Lemma~\ref{holderinclusion_lem}, (2) from Proposition~\ref{almostholder_prop}, (3) from Lemma~\ref{dimension_lem} and (4) from Proposition~\ref{slicing_prop}(4). For the variational bound in (1) note that if $u \in \Hol^\alpha(\R^n)$ with $u(x_0)=0$, then it follows from the decomposition in Lemma~\ref{holderinclusion_lem} and the bound in Proposition~\ref{reverse_prop} that
	\begin{equation}
	\label{bvrescale_eq}
	\bV^\beta(u\chi_{\B(0,1)}) \leq C(n,\alpha,\beta) \Hol^\alpha(u|_{\B(0,1)}) \ .
	\end{equation}
	Set $v \defl (u\chi_{\B(0,r)}) \circ \eta_{r} = (u \circ \eta_r)\chi_{\B(0,1)}$, where $\eta_r(x) = rx$. By \eqref{rescaling_eq} it holds that $\bV^\beta(v \circ \eta_{r^{-1}}) = r^{\beta + n - 1}\bV^\beta(v)$. Because also $\Hol^\alpha((u|_{\B(0,r)}) \circ \eta_{r}) = \Hol^\alpha(u|_{\B(0,r)}) r^\alpha$, the statement in (1) follows from \eqref{bvrescale_eq}.
\end{proof}

\noindent It is not clear if (2) and (4) are sharp or if they also hold at the critical exponent.

\medskip

\noindent The statement in (1) can be seen as a higher dimensional generalization of a result of Young \cite{Y}. Indeed if $f\in \Hol^\alpha([-r,r])$, $g\in \Hol^\beta([-r,r])$ with $\alpha + \beta > 1$, then the Riemann-Stieltjes integral $\int_{-r}^r f\, dg$ exists and there is a constant $C=C(\alpha,\beta) > 0$ such that
\[
\biggl|\int_{-r}^r f\, dg - f(x)(g(r)-g(-r))\biggr| \leq C r^{\alpha + \beta}\Hol^\alpha(f)\Hol^\beta(g) \ ,
\]
for all $x \in [-r,r]$.

\section{Change of variables and Brouwer degree}
\label{push_section}

\subsection{Mass in \texorpdfstring{$\ell_\infty$}{L infinity}}

\noindent The Banach space $\ell_\infty$ is the collection of all bounded functions $f : \N \to \R$ equipped with the norm $\|f\|_\infty \defl \sup_{i \in \N} |f_i|$. The coordinate projections $\pi^i : \ell_\infty \to \R$ for $i \in \N$ are defined by $\pi^i(f) \defl f_i$. It is easy to check that $\Lip(\pi^i) = 1$ for all $i \in \N$.

\begin{lem}
\label{masslem}
Let $T \in \bM_n(\ell_\infty)$. Then
\begin{equation}
\label{masseq}
\bM(T) = \sup \sum_{\lambda \in \Lambda} T\bigl(f_\lambda,\pi^{\lambda_1},\dots,\pi^{\lambda_n}\bigr) \ ,
\end{equation}
where the supremum is taken over all finite sets $\Lambda \subset \N^n$ and all Lipschitz functions $f_\lambda : \ell_\infty \to \R$ with $\sum_{\lambda \in \Lambda} |f_\lambda| \leq 1$.
\end{lem}

\begin{proof}
Let us denote the right-hand side of \eqref{masseq} by $\tilde \bM(T)$. Similar to $\bM$ also $\tilde \bM$ is lower semicontinuous under weak convergence. This follows directly from the definition of $\tilde \bM(T)$ as a supremum. Because all the coordinate projections are $1$-Lipschitz it is clear that $\bM(T) \geq \tilde \bM(T)$. So it suffices to show the other inequality. First note that $\ell_\infty$ has the metric approximation property. Since $\spt(T)$ is compact this implies that for each $k \in \N$ there is a linear projection $p_k : \ell_\infty \to V_k$ of unit norm onto a finite dimensional subspace $V_k \subset \ell_\infty$ such that $\|x-p_k(x)\|_\infty \leq \frac{1}{k}$ for all $x \in \spt(T)$. Indeed, following the proof of \cite[Proposition~4.10]{DP}, we can take $p_k$ so that for all $i,k \in \N$ there is a $j \in \N$ such that $\pi^j = \pi^i \circ p_k$. This implies that $\tilde \bM(p_{k\#} T) \leq \tilde \bM(T)$ in analogy to $\bM(p_{k\#}T) \leq \bM(T)$ which follows from the fact that each $p_k$ is $1$-Lipschitz. Since $p_k$ converges to the identity on $\spt(T)$, the currents $p_{k\#} T$ converge weakly to $T$ and the lower semicontinuity of $\bM$ and $\tilde \bM$ then imply that $\lim_{k\to\infty}\tilde \bM(p_{k\#} T) = \tilde \bM(T)$ and $\lim_{k \to \infty} \bM(p_{k\#}T) = \bM(T)$. Thus if we can show that $\bM(T) \leq \tilde\bM(T)$ for any current $T$ supported in a finite dimensional subspace of $\ell_\infty$, then the same identity holds for all currents in $\ell_\infty$. So we will assume from now on that $T$ is supported in a finite dimensional subspace $V$ of $\ell_\infty$. For such a $V$ and $c \in \ ]0,1[$ there exists $k \in \N$ such that the truncating map $t_k : V \to \ell_\infty^k$ given by $t_k(x) \defl (x_1,\dots,x_k)$ satisfies $c\|v\|_\infty \leq \|t_k(v)\|_\infty \leq \|v\|_\infty$ for all $v \in V$. If this would not be the case, the local compactness of $V$ would guarantee the existence of some $v \in V$ with $\|v\|_\infty = 1$ but $\|v\|_\infty = \lim_{k\to\infty}\|t_k(v)\|_\infty < 1$, which is not possible. Thus we have $c^n\bM(T)\leq \bM(t_{k\#}T) \leq \bM(T)$ and also $\tilde\bM(t_{k\#}T) \leq \tilde\bM(T)$ since either $\pi^i\circ t_k = \pi^i$ or $\pi^i\circ t_k = 0$ for any $k,i$. By letting $c$ tend to $1$, $\lim_{k\to\infty} \bM(t_{k\#} T) = \bM(T)$ and it thus suffices to show $\bM(T) \leq \tilde\bM(T)$ in case the support of $T$ is contained in some truncated subspace $\ell_\infty^k$.

\medskip 

\noindent By a standard smoothing argument in $\ell_\infty^k$ we obtain that
\begin{equation}
\label{masseq2}
\bM(T) = \sup \sum_{\lambda \in \Lambda} T\bigl(f_\lambda,g^{\lambda_1},\dots,g^{\lambda_n}\bigr) \ ,
\end{equation}
where the supremum is taken over all finite sets $\Lambda$, all smooth functions $f_\lambda$, $g^{\lambda_1}$, $\dots$, $g^{\lambda_n} : \ell_\infty^k \to \R$ for $\lambda \in \Lambda$ such that $\sum_{\lambda \in \Lambda} |f_\lambda| \leq 1$ and $\Lip(g^{\lambda_i})\leq 1$ for all $\lambda,i$. Locally we can write $g^{\lambda_i} = c + l + r$, where $c \in \R$, $l : \ell_\infty^k \to \R$ is linear and $1$-Lipschitz and the Lipschitz constant of the second order term $r : \ell_\infty^k \to \R$ is arbitrary small. Using these decompositions and taking a Lipschitz partition of unity we can, up to an arbitrary small error, replace the $g^{\lambda_i}$ in \eqref{masseq2} by linear functions $l : \ell_\infty^k \to \R$ of operator norm at most $1$. Note that the second order terms can be neglected because of the estimate
\begin{align*}
\bigl|T\bigl(\mu,h^{1},\dots,h^{n}\bigr)\bigr| & \leq \bM(T) \|\mu\|_\infty \Lip(h^{1}|_{\spt(\mu)})\cdots \Lip(h^{n}|_{\spt(\mu)})
\end{align*}
for $(\mu,h^1,\dots,h^n) \in \Lip(\ell_\infty)^{n+1}$ due to Lemma~\ref{mass_lem}.

\medskip

\noindent The dual space of $\ell_\infty^k$ is $\ell_1^k$ and every element $l$ in its closed unit ball can be written as $l=\sum_{i =1}^k \mu_i(l)\pi^i$, where $\sum_{i=1}^k |\mu_i(l)| = \|l\| \leq 1$. We infer that
\begin{align*}
\bM(T) & = \sup \sum_{\lambda \in \Lambda} T\bigl(f_\lambda,l^{\lambda_1},\dots,l^{\lambda_n}\bigr) \\
 & = \sup \sum_{\lambda \in \Lambda} \sum_{i_1,\dots,i_n=1}^k T\bigl(f_\lambda\mu_{i_1}(l^{\lambda_1})\cdots \mu_{i_n}(l^{\lambda_n}),\pi^{i_1},\dots,\pi^{i_n}\bigr) \ .
\end{align*}
Note that for $l^{\lambda_1},\dots,l^{\lambda_n} \in \ell_1^k$,
\begin{align*}
\sum_{i_1,\dots,i_n=1}^k|\mu_{i_1}(l^{\lambda_1})\cdots \mu_{i_n}(l^{\lambda_n})| & = \prod_{i=1}^n\bigl(|\mu_1(l^{\lambda_i})| + \cdots + |\mu_k(l^{\lambda_i})|\bigr) \\
 & \leq \|l^{\lambda_1}\|\cdots\|l^{\lambda_n}\| \leq 1 \ ,
\end{align*}
and hence $\sum_{\lambda \in \Lambda}\sum_{i_1,\dots,i_n=1}^k|f_\lambda\mu_{i_1}(l^{\lambda_1})\cdots \mu_{i_n}(l^{\lambda_n})| \leq 1$. Thus $\bM(T)\leq \tilde \bM(T)$ if the support of $T$ is contained in $\ell_\infty^k$. This proves the lemma.
\end{proof}

\subsection{Push forwards into \texorpdfstring{$\ell_\infty$}{L infinity}}

The discussion just below is with only slight modifications also contained in \cite{Z2}. Assume that $T \in \bN_n(X)$ is a normal current in some metric space $X$ and let $0 \leq a < b \leq 1$. As in \cite[Theorem~5.2]{W}, which is a small modification of the cone construction in \cite[Proposition~10.2]{AK}, the multilinear functional $\curr{a,b} \times T : \Lip([0,1] \times X)^{n+2} \to \R$ given by
\begin{align*}
(\curr{a,b} \times T)&(f,g^1,\dots,g^{n+1}) \defl \\
 & \sum_{i=1}^{n+1} (-1)^{i+1} \int_a^b T\left(f_t \partial_t g^i_t, g^1_t, \dots, g^{i-1}_t,g^{i+1}_t, \dots, g^{n+1}_t\right) \, dt \ ,
\end{align*}
defines an element in $\bN_{n+1}([0,1] \times X)$. Out of convenience we put the $\ell_1$-metric on the product $[0,1] \times X$. It is also important to note that $\curr{a,b} \times T \in \bI_{n+1}([0,1] \times X)$ in case $T \in \bI_n(X)$. This construction of a product current has similar properties as the classical one \cite[Section~4.1.8]{F} and for example satisfies the homotopy formula
\begin{equation}
\label{bound}
\partial(\curr{a,b} \times T) = (\curr{b} \times T) - (\curr{a} \times T) - (\curr{a,b} \times \partial T) \ ,
\end{equation}
where the current $\curr t \times T$ in $\bN_n([0,1] \times X)$ is given by 
\[
(\curr{t} \times T)(f,g^1,\dots,g^n) \defl T(f_t,g^1_t,\dots,g^n_t) \ .
\]
From the definition of mass and of $\curr{a,b} \times T$ it is straight forward to show that
\begin{equation}
\label{prodmass_eq}
\bM(\curr{a,b} \times T) \leq (n+1)(b-a)\bM(T) \ .
\end{equation}
Consider a map $\varphi : X \to \ell_\infty$ and assume that the sequence $\balpha = (\alpha_i)_{i \in \N}$ in $]0,1]$ and $H \geq 0$ are such that
\[
\sup_{i \in \N} \Hol^{\alpha_i}(\varphi^i) \leq H < \infty \ .
\]
Here $\varphi^i \defl \pi^i \circ \varphi$ are the coordinate functions of $\varphi$. In order to formulate the results below we set $\tau_k(\balpha) \defl \inf_{i_1,\dots,i_{k} \in \N} \alpha_{i_1} + \dots + \alpha_{i_{k}}$ for any integer $k \geq 1$ and $\tau_0(\balpha) \defl 0$. In this definition we omit the reference to $\balpha$ when it is understood from the context. The map $\tilde \varphi : [0,1] \times X \rightarrow \ell_\infty$ is defined coordinate-wise by
\begin{equation}
\label{approxdef}
\tilde \varphi^i_{t}(x) \defl \inf_{y \in X} \varphi^i(y) + H t^{\alpha_i - 1}d(x,y) \ ,
\end{equation}
if $t > 0$ and $\tilde \varphi^i_{0}(x) \defl \varphi(x)$. From Lemma~\ref{approximation_lem} it follows that for all $i \in \N$ and $t \in \ ]0,1]$:
\begin{align}
\label{approxlocal}
	& \tilde \varphi^i_{t}(x) = \inf_{y \in \B(x,t)} \varphi^i(y) + H t^{\alpha_i - 1}d(x,y) \ , \\
\label{approxlipschitz}
	& \Lip(\tilde\varphi^i_{t}) \leq H t^{\alpha_i - 1} \ , \\
\label{approxclose}
	& \|\tilde\varphi^i_{t} - \varphi^i\|_\infty \leq H t^{\alpha_i} \ .
\end{align}
Using \eqref{approxlocal}, we see that for any $i$ and $x$ the function $t \mapsto \tilde \varphi^i_t(x)$ is $H ba^{\alpha_i-2}$-Lipschitz on $[a,b]$ if $0 < a \leq b \leq 1$. Thus if $\mu > 1$ and $0 < s < 1$ are such that $\mu s \leq 1$, then it is a consequence of \eqref{approxlipschitz} and the choice of the $\ell_1$-metric on $[0,1]\times X$, that each function $\tilde \varphi^i$ is $\mu Hs^{\alpha_i-1}$-Lipschitz on $[s,\mu s] \times X$. In particular, $\tilde \varphi$ is continuous on $]0,1]\times X$. Together with \eqref{approxclose} it follows that $\tilde \varphi$ is continuous everywhere. From Lemma~\ref{masslem}, the product mass estimate \eqref{prodmass_eq} and the bounds on the Lipschitz constants above it follows that
\begin{align}
\nonumber
\bM\bigl(\tilde \varphi_{\#} & (\curr{s,\mu s} \times T)\bigr) \\
\nonumber
 & = \sup_{\Lambda,f} \sum_{\lambda \in \Lambda} \tilde \varphi_{\#} (\curr{s,\mu s} \times T)\bigl(f_\lambda,\pi^{\lambda_1},\dots,\pi^{\lambda_{n+1}}\bigr) \\
\nonumber
 & = \sup_{\Lambda,f} \sum_{\lambda \in \Lambda} (\curr{s,\mu s} \times T)\bigl(f_\lambda \circ \tilde \varphi,\tilde \varphi^{\lambda_1},\dots,\tilde \varphi^{\lambda_{n+1}}\bigr) \\
\nonumber
 & \leq \sup_{\lambda_1,\dots,\lambda_{n+1} \in \N} \bM\bigl(\curr{s,\mu s} \times T\bigr) \prod_{i=1}^{n+1}\Lip\bigl(\tilde\varphi^{\lambda_i}|_{[s,\mu s] \times X}\bigr) \\
\nonumber
 & \leq \sup_{\lambda_1,\dots,\lambda_{n+1} \in \N} (n+1)(\mu-1)s\bM(T)(\mu H)^{n+1} s^{\alpha_{\lambda_1} + \dots + \alpha_{\lambda_{n+1}} - (n+1)} \\
\label{diffestimate}
 & \leq (n+1)(\mu-1)(\mu H)^{n+1}\bM(T) s^{\tau_{n+1} - n} \ .
\end{align}
Similarly, with \eqref{approxlipschitz} we can estimate,
\begin{align}
\nonumber
\bM\bigl(\tilde \varphi_{\#}(\curr{s} \times T)\bigr) & \leq \sup_{\lambda_1,\dots,\lambda_{n} \in \N} \bM\bigl(T\bigr) \prod_{i=1}^{n}\Lip\bigl(\tilde\varphi^{\lambda_i}_s) \\
\nonumber
 & \leq \sup_{\lambda_1,\dots,\lambda_{n} \in \N} \bM(T)H^{n} s^{\alpha_{\lambda_1} + \dots + \alpha_{\lambda_{n}} - n} \\
\label{diffestimate3}
 & \leq H^{n}\bM(T) s^{\tau_{n} - n} \ .
\end{align}
Assuming that $\tau_{n+1} > n$ and summing up in \eqref{diffestimate} it follows that
\begin{align}
\nonumber
\bM\bigl(\tilde \varphi_{\#}(\curr{0,\mu^{-k}} \times T)\bigr) & \leq \sum_{i=k}^\infty \bM\bigl(\tilde \varphi_{\#} (\curr{\mu^{-(i+1)},\mu^{-i}} \times T)\bigr) \\
\nonumber
 & \leq (n+1)(\mu-1)(\mu H)^{n+1}\bM(T) \sum_{i=k}^\infty \mu^{(i+1)(n-\tau_{n+1})} \\
\label{diffestimate2}
 & \leq C(n,\tau_{n+1},\mu) H^{n+1}\bM(T) \mu^{k(n-\tau_{n+1})} \ .
\end{align}
With these estimates we obtain a result about push forwards of normal currents with respect to H\"older maps.

\begin{prop}
\label{push_prop1}
Let $n \geq 1$, $T \in \bN_n(X)$ $(\mbox{or } T \in \bI_n(X)\, )$ and $\varphi : X \to \ell_\infty$ be such that $\sup_i \Hol^{\alpha_i}(\varphi^i) \leq H < \infty$ for some sequence $\balpha = (\alpha_i)_{i \in \N}$ in $]0,1]$ that satisfies $\tau_{n+1}(\balpha) > n$. The current $\varphi_\# T \defl \lim_{t \downarrow 0}\tilde \varphi_\# (\curr{t} \times T) \in \cD_{n}(\ell_\infty)$ is well defined as a weak limit.

\medskip

\noindent Given $\sigma > 0$, there is a constant $C = C(n,\sigma)\geq 0$ such that $\varphi_\# T = R_0 + \sum_{k \geq 1} R_k + \partial S_k$ for sequences $(R_k)_{k \geq 0}$ in $\bN_{n}(\ell_\infty)$ and $(S_k)_{k \geq 1}$ in $\bN_{n+1}(\ell_\infty)$ $($or in $\bI_{n}(\ell_\infty)$ and $\bI_{n+1}(\ell_\infty)\, )$ with
\[
\begin{array}{lll}
&\bM\bigl(S_k\bigr) \leq C H^{n+1} \bM(T)\eta^{k(\gamma - (n+1))} \ , &\bM\bigl(\partial S_k \bigr) \leq C H^{n} \bN(T)\eta^{k(\gamma-n)} \ , \\
&\bM\bigl(R_k\bigr) \leq C H^{n}\bM(\partial T)\rho^{k(\delta-n)} \ , &\bM\bigl(\partial R_k\bigr) \leq C H^{n-1}\bM(\partial T)\rho^{k(\delta-(n-1))} \ , \\
&\bM(R_0) \leq H^n \bM(T) \ , &\bM(\partial R_0) \leq H^{n-1} \bM(\partial T) \ ,
\end{array}
\]
for $k \geq 1$, where $\gamma \defl n + \frac{n-\tau_{n}}{\tau_{n+1}-\tau_{n}}$, $\delta \defl n-1 + \frac{n-1-\tau_{n-1}}{\tau_{n}-\tau_{n-1}}$, $\eta \defl \sigma^{\tau_{n+1}-\tau_n}$ and $\rho \defl \sigma^{\tau_{n}-\tau_{n-1}}$. In particular, $\varphi_\# T \in \bF_{\gamma', \delta'}(\ell_\infty)$ $($or $\varphi_\# T \in \cF_{\gamma', \delta'}(\ell_\infty)\, )$ if $\gamma' \in \ ]\gamma,n+1[$ and  $\delta' \in \ ]\delta,n[$.

\medskip

\noindent Moreover, if there is some $\epsilon \in \ ]0,1]$ such that the maps $\psi_0,\psi_1 : X \to \ell_\infty$ satisfy 
\begin{enumerate}[$\quad$(A)]
	\item $\max_{j=0,1} \Hol^{\alpha_i}(\psi_j^i) \leq H$ for all $i\in\N$, and
	\item $\|\psi_0^i - \psi_1^i\|_\infty \leq H \epsilon^{\alpha_i}$ for all $i \in \N$,
\end{enumerate}
then $\psi_{1\#}T - \psi_{0\#}T = R + \partial S$ for some $S \in \bF_{n+1}(\ell_\infty)$ and $R \in \bF_n(\ell_\infty)$ that satisfy
\[
\bM(S) \leq C'H^{n+1}\bM(T) \epsilon^{\tau_{n+1}-n} \quad \mbox{and} \quad \bM(R) \leq C'H^{n}\bM(\partial T) \epsilon^{\tau_n-(n-1)}
\]
for some constant $C'=C'(n) \geq 0$.
\end{prop}

\begin{proof}
Setting $\mu = \sigma$ and $s = \sigma^{-k}$ in \eqref{diffestimate} and \eqref{diffestimate3} we obtain that there is a constant $C_1 = C_1(n,\sigma)$ such that for all $k \geq 1$,
\begin{align}
\label{massbound_T}
\bM\bigl(\tilde \varphi_\#(\curr{\sigma^{-k},\sigma^{-k+1}}\times T)\bigr) &\leq C_1 H^{n+1} \bM(T)\sigma^{k(n-\tau_{n+1})} \ , \\
\label{massbound_bT}
\bM\bigl(\tilde \varphi_\#(\curr{\sigma^{-k},\sigma^{-k+1}}\times \partial T)\bigr) & \leq C_1 H^{n} \bM(\partial T)\sigma^{k(n-1-\tau_{n})} \ , \\
\label{massbound_sT}
\bM\bigl(\tilde \varphi_{\#}(\curr{\sigma^{-k}} \times T)\bigr) &\leq H^{n}\bM(T) \sigma^{k(n-\tau_{n})} \ .
\end{align}
Define $S_k \defl -\tilde\varphi_\#(\curr{\sigma^{-k},\sigma^{-k+1}}\times T) \in \bN_{n+1}(\ell_\infty)$ and $R_k \defl -\tilde\varphi_\#(\curr{\sigma^{-k},\sigma^{-k+1}} \times \partial T) \in \bN_{n}(\ell_\infty)$ for $k \geq 1$. The homotopy formula \eqref{bound} implies
\begin{equation}
\label{homotopy_normal}
R_k + \partial S_k = \tilde\varphi_\#(\curr{\sigma^{-k}} \times T) - \tilde\varphi_\#(\curr{\sigma^{-k+1}} \times T) \ .
\end{equation}
With \eqref{massbound_bT} and \eqref{massbound_sT} this allows to estimate
\[
\bM(\partial S_k) \leq C_2H^{n} \bN(T)\sigma^{k(n-\tau_{n})}
\]
for some $C_2=C_2(n,\sigma) \geq 0$ (note that $\sigma^{\tau_n-n} \leq 1$). Setting $R_0 \defl \tilde \varphi_\#(\curr 1 \times T)$ it follows from \eqref{diffestimate3} that $\bM(R_0) \leq H^n \bM(T)$ and $\bM(\partial R_0) \leq H^{n-1} \bM(\partial T)$. Thus
\begin{align*}
\varphi_\# T \defl \lim_{k \to \infty} \tilde\varphi_\#(\curr{\sigma^{-k}} \times T) = R_0 + \sum_{k \geq 1} R_k + \partial S_k
\end{align*}
is well defined since both $\sum_{k \geq 1} R_k$ and $\sum_{k\geq 1} S_k$ converge in mass by \eqref{massbound_T} and \eqref{massbound_bT} because $\tau_{n+1} > n$ and also $\tau_{n} > n-1$. Note that $Y \defl \tilde \varphi([0,1]\times \spt(T))$ is a compact set and $(\bM_m(Y),\bM)$ is a Banach space for all $m \geq 0$, see e.g.\ \cite[Proposition~4.2]{L}. If $\gamma = n + \frac{n-\tau_{n}}{\tau_{n+1}-\tau_{n}}$ and $\eta = \sigma^{\tau_{n+1}-\tau_n}$ as in the statement, then $\gamma - (n+1) = \frac{n-\tau_{n+1}}{\tau_{n+1}-\tau_n}$ and thus for $k \geq 1$
\begin{align*}
\bM(S_k) & \leq C_1 H^{n+1} \bM(T)\sigma^{k(n-\tau_{n+1})} = C_1 H^{n+1} \bM(T)\eta^{k(\gamma-(n+1))} \ , \\
\bM(\partial S_k) & \leq C_2H^{n} \bN(T)\sigma^{k(n-\tau_{n})} = C_2H^{n} \bN(T)\eta^{k(\gamma-n)} \ .
\end{align*}
Similar estimates hold for $R_k$, $k \geq 1$, with $\rho = \sigma^{\tau_{n}-\tau_{n-1}}$ and $\delta = n-1 + \frac{n-1-\tau_{n-1}}{\tau_{n}-\tau_{n-1}}$. That $\varphi_\# T = \lim_{k \to \infty} \tilde\varphi_\#(\curr{a_k} \times T)$ for any sequence $(a_k)_{k\geq 0}$ of positive numbers converging to zero is a direct consequence of \eqref{diffestimate}. Indeed, if $a \in [\sigma^{-k},\sigma^{-k+1}]$ and $F \in \Lip(\ell_\infty)^{n+1}$, then as in \eqref{homotopy_normal}
\begin{align*}
\bigl|\tilde\varphi_\#(\curr{\sigma^{-k}} \times T)(F) & - \tilde\varphi_\#(\curr{a} \times T)(F)\bigr| \\
 & = \bigl|\partial(\tilde\varphi_\#(\curr{\sigma^{-k},a} \times T))(F) - \tilde\varphi_\#(\curr{\sigma^{-k},a} \times \partial T)(F)\bigr| \\
 & \leq \bigl|\tilde\varphi_\#(\curr{\sigma^{-k},a} \times T)(1,F)| + |\tilde\varphi_\#(\curr{\sigma^{-k},a} \times \partial T)(F)\bigr| \\
 & \leq C_3(n,\sigma,H,F)\bigl(\bM(T)\sigma^{k(n-\tau_{n+1})} + \bM(\partial T)\sigma^{k(n-1-\tau_{n})}\bigr) \ .
\end{align*}
The latter term is arbitrarily small for $k$ big.

\medskip

\noindent For the second part consider first two Lipschitz maps $\gamma_0,\gamma_1 : X \to \ell_\infty$ and $\zeta \in \ ]0,1]$ such that $\|\gamma_1^i - \gamma_0^i\|_\infty \leq H\zeta^{\alpha_i}$ and $\Lip(\gamma^i_j) \leq H\zeta^{\alpha_i-1}$ for all $i \in \N$ and $j=0,1$. Let $\Gamma : [0,1] \times X \to \ell_\infty$ be the linear homotopy given by $\Gamma_t(x) \defl t\gamma_1(x) + (1-t)\gamma_0(x)$. For $i \in \N$ it is clear that $\|\partial_t \Gamma^i_t\|_\infty \leq H\zeta^{\alpha_i}$ and $\Lip(\Gamma^i_t) \leq H\zeta^{\alpha_i-1}$. For each $\lambda = (\lambda_1,\dots,\lambda_{n+1}) \in \N^{n+1}$ and $i \in \{1,\dots,n+1\}$ set
\[
\hat \Gamma^{\lambda,i}_t \defl \bigl(\Gamma^{\lambda_1}_t, \dots, \Gamma^{\lambda_{i-1}}_t,\Gamma^{\lambda_{i+1}}_t, \dots, \Gamma^{\lambda_{n+1}}_t\bigr) \ .
\]
Similarly to the estimate in \eqref{diffestimate} it follows from Lemma~\ref{masslem} that
\begin{align*}
\bM&\left(\Gamma_\# (\curr{0,1} \times T)\right) \\
 & = \sup_{\Lambda,f} \sum_{\lambda \in \Lambda} \Gamma_\# (\curr{0,1} \times T)\bigl(f_\lambda,\pi^{\lambda_1},\dots,\pi^{\lambda_{n+1}}\bigr) \\
 & = \sup_{\Lambda,f} \sum_{\lambda \in \Lambda} (\curr{0,1} \times T)\bigl(f_\lambda\circ \Gamma,\Gamma^{\lambda_1},\dots,\Gamma^{\lambda_{n+1}}\bigr) \\
 & = \sup_{\Lambda,f} \sum_{\lambda \in \Lambda} \sum_{i=1}^{n+1}(-1)^{i+1} \int_0^1 T\bigl(f_\lambda\circ\Gamma_t \, \partial_t\Gamma_t^{\lambda_i},\hat\Gamma_t^{\lambda,i}\bigr)\, dt \\
 & \leq \sup_{\Lambda,f} \sum_{\lambda \in \Lambda} \sum_{i=1}^{n+1} \int_0^1 \|\partial_t\Gamma_t^{\lambda_i}\|_\infty\biggl(\prod_{j\neq i} \Lip(\Gamma_t^{\lambda_j})\biggr) \int_{X} |f_\lambda \circ \Gamma_t| \, d\|T\| \, dt \\
 & \leq (n+1)H^{n+1} \sup_{\Lambda,f} \sum_{\lambda \in \Lambda} \zeta^{\alpha_{\lambda_1} + \dots + \alpha_{\lambda_{n+1}} - n} \int_0^1 \int_{X} |f_\lambda \circ \Gamma_t| \, d\|T\| \, dt \\
 & \leq (n+1)H^{n+1} \zeta^{\tau_{n+1}-n} \bM(T) \ .
\end{align*}
In the last line we used that $\sum_{\lambda \in \Lambda} |f_\lambda| \leq 1$ and $\tau_{n+1} \leq \alpha_{\lambda_1} + \dots + \alpha_{\lambda_{n+1}}$ for any $\lambda \in \Lambda$. Setting $R \defl \Gamma_\#(\curr{0,1}\times \partial T) \in \bN_n(\ell_\infty)$ and $S \defl \Gamma_\#(\curr{0,1}\times T) \in \bN_{n+1}(\ell_\infty)$ we obtain from the homotopy formula \eqref{bound} that
\begin{align}
\label{boundary_difference}
R + \partial S & = \gamma_{1\#}T-\gamma_{0\#}T \ , \\
\label{mass_difference_a}
\bM(S) & \leq (n+1)H^{n+1} \zeta^{\tau_{n+1}-n} \bM(T) \ , \\
\label{mass_difference_b}
\bM(R) & \leq nH^{n} \zeta^{\tau_{n}-(n-1)} \bM(\partial T) \ .
\end{align}
Assume that $\psi_0,\psi_1 : X \to \ell_\infty$ are as in the statement, i.e.\ there is some $\epsilon \in \ ]0,1]$ and $H \geq 0$ such that $\Hol^{\alpha_i}(\psi_j^i) \leq H$ for all $i\in\N$ and $j=0,1$, and $\|\psi_0^i - \psi_1^i\|_\infty \leq H \epsilon^{\alpha_i}$ for all $i \in \N$. We set $\sigma = 2$ and let $k \geq 0$ be the unique integer such that $2^{-k-1} < \epsilon \leq 2^{-k}$ and define $\tilde \psi_0$, $\tilde \psi_1 : [0,1]\times X \to \ell_\infty$ as in \eqref{approxdef}. Due to Lemma~\ref{approximation_lem}(6) it holds that for all $i \in \N$,
\begin{equation}
\label{diff_bound}
\|\tilde \psi^i_{1,2^{-k}} - \tilde \psi^i_{0,2^{-k}}\bigr\|_\infty \leq \bigl\|\psi^i_{1} - \psi^i_{0}\bigr\|_\infty \leq H \epsilon^{\alpha_i} \leq H 2^{-k\alpha_i} \ .
\end{equation}
Also, \eqref{approxlipschitz} implies for all $i \in \N$,
\begin{equation}
\label{lip_bound}
\max_{j=0,1} \Lip\bigl(\tilde \psi^i_{j,2^{-k}}\bigr) \leq H 2^{k(1-\alpha_i)} \ .
\end{equation}
We set $\tilde S_{j,l} \defl -\tilde\psi_{j\#}(\curr{2^{-l},2^{-l+1}}\times T) \in \bN_{n+1}(\ell_\infty)$, $\tilde R_{j,l} \defl -\tilde\psi_{j\#}(\curr{2^{-l},2^{-l+1}} \times \partial T) \in \bN_n(\ell_\infty)$ for $l \geq 1$, $j=0,1$, and $R' + \partial S' = \tilde \psi_{1,2^{-k}\#}T - \tilde \psi_{0,2^{-k}\#}T$ for $R'$ and $S'$ as in \eqref{boundary_difference}. Equation \eqref{homotopy_normal} implies that
\begin{align*}
\psi_{1\#}T - \psi_{0\#}T & = \tilde \psi_{1,2^{-k}\#}T - \tilde \psi_{0,2^{-k}\#}T + \sum_{l > k}(\tilde R_{1,l}-\tilde R_{0,l})+\partial(\tilde S_{1,l}-\tilde S_{0,l}) \\
 & = R' + \sum_{l > k}(\tilde R_{1,l}-\tilde R_{0,l}) + \partial\biggl(S'+\sum_{l > k}(\tilde S_{1,l}-\tilde S_{0,l})\biggr) \ .
\end{align*}
Hence $\psi_{1\#}T - \psi_{0\#}T = R + \partial S$, where $S \in \bF_{n+1}(\ell_\infty)$ and $R \in \bF_n(\ell_\infty)$ satisfy
\[
\bM(S) \leq C_4H^{n+1}\bM(T) \epsilon^{\tau_{n+1}-n} \quad \mbox{and} \quad \bM(R) \leq C_4H^{n}\bM(T) \epsilon^{\tau_n-(n-1)}
\]
for some $C_4 = C_4(n) \geq 0$. Here the mass bounds for $R'$ and $S'$ are obtained by applying the bounds \eqref{mass_difference_a} and \eqref{mass_difference_b} using \eqref{diff_bound} and \eqref{lip_bound}. The mass bounds for the sums are obtained by \eqref{massbound_T} and \eqref{massbound_bT} for $\sigma = 2$. Note that this also shows that $\varphi_\# T$ does not depend on the approximating map $\tilde \varphi$ that we used in the definition of $\varphi_\# T$.
\end{proof}

\noindent The next proposition treats push forwards of boundaries of fractal currents with respect to H\"older maps.

\begin{prop}
\label{push_prop2}
Let $n \geq 1$, $T \in \cD_{n}(X)$ and $\varphi : X \to \ell_\infty$. Assume that there are $d \in \ ]n-1,n[$, $V,H \geq 0$, $\sigma > 1$ and sequences $(S_{k})_{k \geq 0}$ in $\bN_{n}(X)$ $($or $\bI_{n}(X)\, )$ and $\balpha = (\alpha_i)_{i \in \N}$ in $]0,1]$ such that:
\begin{enumerate}
	\item $T = \sum_{k \geq 0} S_k$ weakly.
	\item For all $k \geq 0$,
	\[
	\begin{array}{lll}
	&\bM(S_k) \leq V \sigma^{k(d-n)} \ , &\bM(\partial S_k) \leq V \sigma^{k(d-(n-1))} \ .
	\end{array}
	\]
	\item $\sup_{i\in\N} \Hol^{\alpha_i}(\varphi^i) \leq H$ and $\tau_{n}(\balpha) > d$.
\end{enumerate}
Then $\varphi_\# \partial T \defl \lim_{k \to \infty} \varphi_\# \partial \sum_{l = 0}^k S_l \in \cD_{n-1}(\ell_\infty)$ is well defined as a weak limit. Indeed there is a sequence $(\tilde S_k)_{k\geq 0}$ in $\bN_{n}(\ell_{\infty})$ $($or $\bI_{n}(\ell_\infty)\, )$ and a constant $C=C(n,d,\sigma)$ such that $\varphi_\# \partial T = \partial\sum_{k \geq 0} \tilde S_k$, where
\[
\begin{array}{lll}
&\bM\bigl(\tilde S_k\bigr) \leq CV H^{n}\eta^{k(d'- n)} \ , &\bM\bigl(\partial \tilde S_k\bigr) \leq CV H^{n-1}\eta^{k(d'-(n-1))} \ ,
\end{array}
\]
for the parameters $\eta \defl \sigma^{\tau_{n}-\tau_{n-1}} > 1$ and $d' \defl n-1 + \frac{d-\tau_{n-1}}{\tau_{n}-\tau_{n-1}}$. Note that $d' = \frac{d}{\alpha}$ in case $\alpha = \alpha_i$ for all $i \in \N$.

\medskip

\noindent Moreover, if there is some $\epsilon \in \ ]0,1]$ such that the maps $\psi_0,\psi_1 : X \to \ell_\infty$ satisfy 
\begin{enumerate}[$\quad$(A)]
	\item $\max_{j=0,1}\Hol^{\alpha_i}(\psi_j^i) \leq H$ for all $i\in\N$, and
	\item $\|\psi_0^i - \psi_1^i\|_\infty \leq H \epsilon^{\alpha_i}$ for all $i \in \N$,
\end{enumerate}
then $\psi_{1\#}\partial T - \psi_{0\#} \partial T = \partial S$ where $S \in \bF_{n}(\ell_\infty)$ with
\[
\bM(S) \leq C' VH^n \epsilon^{d-\tau_n}
\]
for some constant $C'=C'(n,d,\tau_n,\sigma) \geq 0$.
\end{prop}

\begin{proof}
First note that for $S_k' \defl \sum_{l = 0}^{k} S_l$, $k \geq 0$, it holds that
\begin{equation}
\label{boundarybound}
\bM(\partial S_k') \leq \sum_{l = 0}^{k} V \sigma^{l(d-(n-1))} \leq C_1 V \sigma^{k(d-(n-1))}
\end{equation}
for some $C_1 = C_1(n,d,\sigma) \geq 0$. From \eqref{diffestimate3} it follows that for all $k \geq 0$,
\begin{align}
\nonumber
\bM\bigl(\tilde \varphi_{\#}(\curr{\sigma^{-k}} \times S_{k+1})\bigr) & \leq H^{n}\bM(S_{k+1}) \sigma^{k(n-\tau_{n})} \\
\nonumber
 & \leq V H^{n} \sigma^{k((n-\tau_{n}) + d-n)} \\
\label{est1}
 & = VH^{n} \sigma^{k(d-\tau_{n})} \ ,
\end{align}
and similarly \eqref{boundarybound} implies that
\begin{align}
\nonumber
\bM\bigl(\tilde \varphi_{\#}(\curr{\sigma^{-k}} \times \partial S_k')\bigr) & \leq H^{n-1}\bM(\partial S_k') \sigma^{k(n-1-\tau_{n-1})} \\
\nonumber
 & \leq C_1VH^{n-1} \sigma^{k((n-1-\tau_{n-1}) + d - (n-1))} \\
\label{est3}
 & = C_1VH^{n-1} \sigma^{k(d - \tau_{n-1})} \ .
\end{align}
From \eqref{diffestimate}, with $s=\sigma^{-k}$ and $\mu = \sigma$, it follows that
\begin{align}
\nonumber
\bM\bigl(\tilde \varphi_{\#}(\curr{\sigma^{-k},\sigma^{-k+1}} \times \partial S_k')\bigr) & \leq C_2 H^{n}\bM(\partial S_k') \sigma^{k(n-1-\tau_{n})} \\
\nonumber
 & \leq C_2 C_1 V H^{n} \sigma^{k((n-1-\tau_{n}) + d - (n-1))} \\
\label{est2}
 & = C_2 C_1 V H^{n} \sigma^{k(d - \tau_{n})} \ ,
\end{align}
for some constant $C_2 = C_2(n,\sigma)$. For all $k \geq 0$ we define the currents 
\begin{align}
\label{firstdef}
S^1_k & \defl \tilde \varphi_{\#}\bigl(\curr{\sigma^{-k}} \times S_{k+1}\bigr) \in \bN_{n}(\ell_\infty) \ , \\
\label{seconddef}
S^2_{k+1} & \defl \tilde \varphi_{\#}\bigl(\curr{\sigma^{-k-1},\sigma^{-k}} \times \partial S_{k+1}'\bigr) \in \bN_{n}(\ell_\infty) \ .
\end{align}
With the homotopy formula \eqref{bound} the boundary of the difference is 
\begin{align*}
\partial \bigl(S^1_k - S^2_{k+1}\bigr) & = \tilde \varphi_{\#}\bigl(\curr{\sigma^{-k}} \times \partial S_{k+1} + \curr{\sigma^{-k-1}} \times \partial S_{k+1}' - \curr{\sigma^{-k}} \times \partial S_{k+1}'\bigr) \\
 & = \tilde \varphi_{\#}\bigl(\curr{\sigma^{-k}} \times \partial (S_{k+1}'-S_k') + \curr{\sigma^{-k-1}} \times \partial S_{k+1}' - \curr{\sigma^{-k}} \times \partial S_{k+1}'\bigr) \\
 & = \tilde \varphi_{\#}\bigl(\curr{\sigma^{-k-1}} \times \partial S_{k+1}'\bigr) - \tilde \varphi_{\#}\bigl(\curr{\sigma^{-k}} \times \partial S_{k}'\bigr)\ .
\end{align*}
The mass estimates \eqref{est1}, \eqref{est3} and \eqref{est2} then imply that there is some constant $C_3=C_3(n,d,\sigma)\geq 0$ such that for all $k \geq 0$,
\begin{align}
\label{massbound_s1s2}
\bM\bigl(S^1_k - S^2_{k+1}\bigr) & \leq C_3 V H^{n}\sigma^{k(d-\tau_{n})} \ , \\
\label{massbound_bs1s2}
\bM\bigl(\partial(S^1_k - S^2_{k+1})\bigr) & \leq C_3 V H^{n-1}\sigma^{k(d-\tau_{n-1})} \ .
\end{align}
Since $\tau_{n} > d$, the sum $\sum_{k \geq 0} (S^1_k - S^2_{k+1})$ converges in mass and thus
\begin{equation}
\label{pushforward_def}
\varphi_\# \partial T \defl \lim_{k \to \infty} \tilde \varphi_{\#}\bigl(\curr{\sigma^{-k}} \times \partial S_{k}'\bigr) = \tilde \varphi_\# (\curr 1 \times \partial S_0) + \partial \sum_{k \geq 0} S^1_k - S^2_{k+1}
\end{equation}
is well defined as a weak limit. Note that with \eqref{diffestimate2} and \eqref{boundarybound}
\begin{align}
\nonumber
\bM\bigl(\tilde \varphi_{\#}(\curr{0,\sigma^{-k}} \times \partial S_k')\bigr) & \leq C_{\ref{diffestimate2}}(n-1,\tau_n,\sigma) H^{n}\bM(\partial S_k') \sigma^{k((n-1)-\tau_{n})} \\
\label{massbound_skprime}
 & \leq C_{\ref{diffestimate2}} H^{n} C_1 V \sigma^{k(d-\tau_{n})} \ ,
\end{align}
and thus $\lim_{k \to \infty} \varphi_{\#}\partial S_{k}' = \lim_{k \to \infty} \tilde \varphi_{\#}\bigl(\curr{\sigma^{-k}} \times \partial S_{k}'\bigr) = \varphi_\# \partial T$ by Proposition~\ref{push_prop1}. Because of \eqref{diffestimate3} the mass bounds $\bM(\tilde \varphi_\# (\curr 1 \times S_0)) \leq VH^{n}$ and $\bM(\partial \tilde \varphi_\# (\curr 1 \times S_0)) \leq VH^{n-1}$ hold. Finally, if $\eta = \sigma^{\tau_{n}-\tau_{n-1}}$ and $d' \defl n-1 + \frac{d-\tau_{n-1}}{\tau_{n}-\tau_{n-1}} = n + \frac{d-\tau_{n}}{\tau_{n}-\tau_{n-1}}$ as in the statement, then $\eta^{k(d'-n)} = \sigma^{k(d-\tau_{n})}$ and $\eta^{k(d'-(n-1))} = \sigma^{k(d-\tau_{n-1})}$. Together with \eqref{massbound_s1s2}, \eqref{massbound_bs1s2} and \eqref{pushforward_def} this concludes the decomposition result for $\varphi_\# \partial T$.

\medskip

\noindent In order to see that this push forward does not depend on the approximating sequence let $\psi_0,\psi_1$ and $\epsilon \in \ ]0,1]$ be as in the statement. Consider $k \geq 0$ such that $\sigma^{-k-1} < \epsilon \leq \sigma^{-k}$. Let $\tilde \psi_0, \tilde \psi_1 : [0,1] \times X \to \ell_\infty$ be as defined in \eqref{approxdef}. For $0 \leq l \leq k$ we use the second part of Proposition~\ref{push_prop1} to find that $\psi_{1\#}\partial S_k'-\psi_{0\#}\partial S_k' = \partial S'$ for some $S' \in \bF_n(\ell_\infty)$ with mass bound $\bM(S') \leq C_4(n)H^{n}\bM(\partial S_k') \epsilon^{\tau_{n}-(n-1)}$. Due to \eqref{boundarybound}
\begin{align}
\nonumber
\bM(S') & \leq C_4H^{n}\bM(\partial S_k') \sigma^{-k(\tau_{n}-(n-1))} \\
\nonumber
 & \leq C_5 V H^{n} \sigma^{k(d-(n-1))}\sigma^{k((n-1)-\tau_{n})} \\
\label{massbound_s}
 & = C_5 V H^{n} \sigma^{k(d-\tau_{n})} \ ,
\end{align}
for some $C_5 = C_5(n,d,\sigma) \geq 0$. We define $S^1_{j,l}$ and $S^2_{j,l+1}$ for $l \geq 0$ and $j=0,1$ as in \eqref{firstdef} and \eqref{seconddef}. For $j=0,1$ it holds that
\begin{align*}
\psi_{j\#}\partial T - \tilde \psi_{j\#}\bigl(\curr{\sigma^{-k}} \times \partial S_k'\bigr) & = \sum_{l \geq k} \tilde \psi_{j\#}\bigl(\curr{\sigma^{-l-1}} \times \partial S_{l+1}'\bigr) - \tilde \psi_{j\#}\bigl(\curr{\sigma^{-l}} \times \partial S_l'\bigr) \\
 & = \sum_{l \geq k} \partial \bigl(S^1_{j,l} - S^2_{j,l+1}\bigr)
\end{align*}
With \eqref{massbound_s1s2}, \eqref{massbound_skprime} and \eqref{massbound_s} we obtain that $\psi_{1\#}\partial T - \psi_{0\#}\partial T = \partial S$ for some $S \in \bF_n(\ell_\infty)$ with $\bM(S) \leq C_6(n,d,\tau_n,\sigma)VH^n \epsilon^{d-\tau_n}$.
\end{proof}

\noindent This proposition remains true for $d = n-1$ in case we assume that $T = S_0 \in \bN_n(X)$ since the only place we actually assume that $d > n-1$ is \eqref{boundarybound}.

\medskip

\noindent Note that although the two propositions above are formulated for push forwards into $\ell_\infty$ it also covers finite dimensional Euclidean targets as these are bi-Lipschitz equivalent to some $\ell_\infty^m$. Moreover, since any separable metric space can be isometrically embedded into $\ell_\infty$ using distance functions, these results also treat push forwards for H\"older maps in $\Hol^\alpha(X,Y)$ with the appropriate restrictions on $\alpha$.

\medskip

\noindent Together with Theorem~\ref{equivalence_thm} we can show that the exponents obtained in Proposition~\ref{push_prop1} and Proposition~\ref{push_prop2} are best possible (up to the critical exponent).

\begin{ex}
Fix some integer $n \geq 2$ and let $(a_k)_{k \in \N}$ be some decreasing sequence of positive numbers such that $\sum_{k \geq 1} a_k^{n-1} < \infty$ but $\sum_{k \geq 1} a_k^{n-1-\epsilon} = \infty$ for all $\epsilon > 0$. For $k \in \N$ we define the cubes $Q_k \defl [0,a_k]^n$. Because of the summability assumption on $(a_k)_{k \in \N}$, the current $T \defl \sum_{k \geq 1} \curr{Q_k}$ is an element of $\bI_n(\R^n)$, i.e.\ $T$ has finite boundary mass. Let $1 \geq \alpha_1 \geq \dots \geq \alpha_n > 0$ be such that $\alpha_1 + \dots + \alpha_n > n-1$. In this case $\tau_n = \tau_{n-1} + \alpha_1$. Consider the map $\varphi : \R^n \to \R^n$ given by $\varphi(x_1,\dots,x_n) \defl (x_1^{\alpha_1},\dots,x_n^{\alpha_n})$. Proposition~\ref{push_prop1} (or also Proposition~\ref{push_prop2}) implies that
\[
\varphi_\# \partial T = \sum_{k \geq 1} \varphi_\# \partial \curr{Q_n} = \sum_{k \geq 1} \partial \curr{[0,a_k^{\alpha_1}]\times \dots \times [0,a_k^{\alpha_n}]} \ .
\]
By Proposition~\ref{push_prop2} there exists $S \in \cD_n(\R^n)$ with $\partial S = \varphi_\# \partial T$. The particular decomposition of $S$ in the statement of Proposition~\ref{push_prop2} shows that $S \in \bF_{n,\delta}(\R^n)$ for all $\delta \in \ ]d',n[$, where $d' =  n-1 + \frac{d-\tau_{n-1}}{\tau_{n}-\tau_{n-1}}$. Note that by the constancy theorem for currents, $S$ is the unique filling of $\varphi_\# \partial T$. According to Theorem~\ref{equivalence_thm}, since $S$ belongs to $\bigcap_{d' < \delta < n}\bF_{n,\delta}(\R^n)$, the current $\partial S = \varphi_\# \partial T$ has an extension to a H\"older current on $\Hol^\alpha(\R^n) \times \Lip(\R^n)^{n-1}$ in case $\alpha \in \ ]d'-(n-1),1]$. On the other side if $\varphi_\# \partial T$ has an extension to a H\"older current on $\Hol^\alpha(\R^n) \times \Lip(\R^n)^{n-1}$, then
\begin{align*}
\varphi_\# \partial T(x_1^\alpha,x_2,\dots,x_n) & = \sum_{k \geq 1} \partial \varphi_\# \curr{Q_k}(x_1^\alpha,x_2,\dots,x_n) \\
 & = \sum_{k \geq 1} \curr{Q_k}(1,\varphi_1^\alpha,\varphi_2,\dots,\varphi_n) \\
 & = \sum_{k \geq 1} a_k^{\alpha \alpha_1 + \alpha_2 + \dots + \alpha_n} \\
 & = \sum_{k \geq 1} a_k^{\alpha (\tau_n-\tau_{n-1}) + \tau_{n-1}} \ .
\end{align*}
This sum is finite only if $\alpha (\tau_n-\tau_{n-1}) + \tau_{n-1} \geq n-1$. Thus the extension property can only hold for $\alpha\in [\frac{n-1 - \tau_{n-1}}{\tau_n-\tau_{n-1}},1]$ and this agrees, except for the critical exponent, with the range for $\alpha$ obtained above. Thus $d'$ as obtained in Proposition~\ref{push_prop2} and $\delta$ in Proposition~\ref{push_prop1} are optimal.
\end{ex}

\subsection{Push forwards into Euclidean spaces}

\label{push_subsec}

In this subsections we consider push forwards of $n$-dimensional currents living in a general metric space into $\R^n$. In the classical setting this is described by the generalized change of variables formula: If $u \in L^1_c(\R^n)$ and $\varphi \in \Lip(\R^n,\R^n)$, then $\varphi_\#\curr u = \curr v$ for the function $v \in L_c^1(\R^n)$ given by
\begin{equation}
\label{changeofvar}
v(y) = \sum_{x \in \varphi^{-1}(y)} u(x) \sign(\det D\varphi_x) \ ,
\end{equation}
for almost all $y \in \R^n$, see e.g.\ \cite[Lemma~3.7]{L}. Proposition~\ref{push_prop2} together with the constancy theorem for currents shows that $\varphi_\#\curr u = \curr v$ can be extended to a certain class of H\"older maps $\varphi$ in case $u$ is nice enough. We formulate this here first for arbitrary domains $X$.

\begin{thm}
\label{changeofvar_thm}
Let $n \geq 1$, $d \in [n-1,n[$ and $T \in \cD_n(X)$ for which there is a sequence $(R_k)_{k\geq 0}$ in $\bN_n(X)$ $(\mbox{or in } \bI_n(X)\,)$ such that
\begin{enumerate}
	\item $T = \sum_{k \geq 0} R_k$ weakly.
	\item There are $V \geq 0$ and $\rho > 1$ such that for all $k \geq 0$,
	\[
	\begin{array}{lll}
	&\bM(R_k) \leq V \rho^{k(d-n)} \ , &\bM(\partial R_k) \leq V \rho^{k(d-(n-1))} \ .
	\end{array}
	\]
\end{enumerate}
If $d = n-1$ we assume $T = R_0 \in \bN_n(X)$. Given $\varphi : X \to \R^n$ and $\alpha_1,\dots,\alpha_n \in \ ]0,1]$ with $\max_{i} \Hol^{\alpha_i}(\varphi^i) \leq H < \infty$ and $\tau_n \defl \alpha_1 + \dots + \alpha_n > d$, the current $\varphi_\# T$ is well defined by approximation and is equal to $\curr{v_{T,\varphi}}$ for some $v_{T,\varphi} \in L_c^1(\R^n)$ $($or $v_{T,\varphi} \in L_c^1(\R^n,\Z)\, )$. Indeed, if there is some $\epsilon \in \ ]0,1]$ and maps $\varphi,\psi : X \to \R^n$ that satisfy 
\begin{enumerate}[$\quad$(A)]
	\item $\max_i\{\Hol^{\alpha_i}(\varphi^i),\Hol^{\alpha_i}(\psi^i)\} \leq H$, and
	\item $\|\varphi^i - \psi^i\|_\infty \leq H \epsilon^{\alpha_i}$ for all $i \in \{1,\dots,n\}$,
\end{enumerate}
then
\[
\|v_{T,\varphi}-v_{T,\psi}\|_{L^1} \leq C' VH^n \epsilon^{d-\tau_n} \ ,
\]
for some constant $C'=C'(n,d,\tau_n,\sigma) \geq 0$.

\medskip

\noindent Moreover, $\partial \varphi_\# T = \varphi_\# \partial T$, where the right-hand side is defined in Proposition~\ref{push_prop2}. Further, there are $v_k \in \BV_c(\R^n)$ with:
\begin{enumerate}[$\quad$(i)]
	\item $v_{T,\varphi} = \sum_{k \geq 0} v_k$ in $L^1$ and such that $\bigcup_k \spt(v_k)$ is bounded.
	\item There is some $C=C(n,d,\sigma) \geq 0$ such that for all $k \geq 0$,
	\begin{equation}
	\label{changeofvar_eq}
	\|v_k\|_{L^1} \leq CV H^{n}\eta^{k(d'- n)} \ , \quad \bV(v_k) \leq CV H^{n-1}\eta^{k(d'-(n-1))} \ ,
	\end{equation}
	where $\eta \defl \sigma^{\tau_{n}-\tau_{n-1}} > 1$, $d' \defl n-1 + \frac{d-\tau_{n-1}}{\tau_{n}-\tau_{n-1}} = n + \frac{d-\tau_{n}}{\tau_{n}-\tau_{n-1}}$ and $\tau_{n-1} \defl \tau_n - \max_i \alpha_i$.
\end{enumerate}
Note that if $\alpha = \alpha_i$ for all $i$, then $\eta = \sigma^\alpha$ and $d' = \frac{d}{\alpha}$.
\end{thm}

\begin{proof} 
It follows from Proposition~\ref{push_prop2}, that $\varphi_\# \partial T = \partial R$ for some $R \in \bF_n(\R^n)$ (or $R \in \cF_n(\R^n)$) is well defined by approximation. Because $R$ is in the $\bM$-closure of $\bN_n(\R^n)$, respectively, the $L^1$-closure of $\BV_c(\R^n)$ (or $\BV_c(\R^n) \cap L^1(\R^n,\Z)$) by Lemma~\ref{BVnormal_lem}, it follows that $R = \curr{v_{T,\varphi}}$ for some $v_{T,\varphi} \in L^1_c(\R^n)$ (or $L^1(\R^n,\Z)$). The constancy theorem for currents implies that there can only be one such function in $L^1_c(\R^n)$.

\medskip

\noindent Let $\psi : X \to \R^n$ be as in the statement. By the second part of Proposition~\ref{push_prop2} it follows that there is some $S \in \bF_n(\R^n)$ with $\partial S = \varphi_\# \partial T - \psi_\# \partial T$ and $\bM(S) \leq C'(n,d,\tau_n,\rho)VH^n \epsilon^{d-\tau_n}$. Since $\bF_n(\R^n) = L^1_c(\R^n)$ by \cite[Section~4.1.18]{F} and \cite[Theorem~5.5]{L}. It follows that $S = \curr{v}$ for some $v \in L^1_c(\R^n)$ and $v = v_{T,\varphi} - v_{T,\psi}$ almost everywhere by the constancy theorem for currents. Thus $\|v_{T,\varphi} - v_{T,\psi}\|_{L^1} \leq C'(n,d,\tau_n,\rho)VH^n \epsilon^{d-\tau_n}$. The rest of the statements follow directly from Proposition~\ref{push_prop2}.
\end{proof}

\noindent It is possible to improve the bounds in \eqref{changeofvar_eq}. Take individual approximations in \eqref{approxdef}, where $H$ is replaced by $\Lip^{\alpha_i}(\varphi^i)$ in each coordinate $i$. The proofs of Proposition~\ref{push_prop1} and Proposition~\ref{push_prop2} work unchanged replacing the occurrences of $H^n$ and $H^{n-1}$ by $\prod_{i=1}^n \Lip^{\alpha_i}(\varphi^i)$ and $\max_{j}\prod_{i \neq j} \Lip^{\alpha_i}(\varphi^i)$ respectively. This is so because already in \eqref{diffestimate}, \eqref{diffestimate3} and \eqref{diffestimate2} this change can be made.

\medskip

\noindent If we consider $X = \R^n$, then the theorem above shows that the function
\[
y \mapsto \sum_{x \in \varphi_\epsilon^{-1}(y)} u(x) \sign(\det D\varphi_\epsilon(x))
\]
converges in $L^1$ if $\varphi$ is an appropriate H\"older map, $u$ is nice enough and $\varphi_\epsilon$ are good Lipschitz approximations of $\varphi$. One can take for example coordinatewise smoothing $\varphi_\epsilon^i = \rho_\epsilon\ast\varphi^i$, where $\rho_\epsilon$ is a smooth approximation of the identity. It is a simple exercise to check that $\Hol^{\alpha_i}(\varphi_\epsilon) \leq \Hol^{\alpha_i}(\varphi)$ and that $\lim_{\epsilon \downarrow 0}\varphi_\epsilon = \varphi$ locally uniformly.

\medskip

\noindent Together with Proposition~\ref{higherint_prop} we obtain the following integrability result for push forwards into $\R^{n}$. If $T$ and $\varphi$ are as in the theorem above, then $\varphi_\# T = \curr{v_{T,\varphi}}$ for some $v_{T,\varphi} \in L^1_c(\R^n)$ and $\varphi \in L^p_c(\R^n)$ whenever
\begin{equation}
\label{seminorm_est}
1 \leq p < \frac{n(\tau_{n} - \tau_{n-1})}{(n-1)\tau_{n} - n\tau_{n-1} + d} \ .
\end{equation} 
In case $\alpha = \alpha_1 = \cdots = \alpha_n$, then $\tau_{n} = \alpha n$, $\tau_{n-1} = \alpha (n-1)$ and the integrability range is $1 \leq p < \frac{\alpha n}{d}$. This agrees with the values for $p$ that we obtain in Theorem~\ref{pushforward_thm} and Theorem~\ref{degreeint_thm} in the situation $X = \R^n$ but it does not in case the exponents are different. This suggests that \eqref{seminorm_est} is not optimal for general domains $X$.

\medskip

\noindent In the setting of the Brouwer degree it is conjectured in \cite{DI} that the integrability range is $1 \leq p < \frac{\tau_{n}}{d}$ and hence shouldn't depend on $\tau_{n-1}$. We will prove this in Theorem~\ref{pushforward_thm} below. The proof relies on a dyadic cube decomposition of the domain and affine approximations of the functions. It is therefore not obvious how to adapt this to general ambient spaces.

\subsection{Brouwer degree functions as currents}

\label{brouwer_subseq}

\noindent We start with a very short review of the Brouwer degree. All the results about the Brouwer degree we will use can be found for example in \cite{OR}. Assume that $V \subset \R^n$ is a bounded open set and $\varphi : \operatorname{cl}(V) \to \R^n$ is a continuous map. For any point $q \in \R^n \setminus \varphi(\partial V)$ the Brouwer degree $\degr \varphi V q \in \Z$ is defined. In case $C \subset \R^n$ is compact and $\operatorname{cl}(\operatorname{int}(C)) = C$ we also use $\degr \varphi C q$ instead of $\degr \varphi {\operatorname{int}(C)} q$. If $\varphi$ is a smooth map and $q \in \R^n \setminus \varphi(\partial V)$ is a regular value, then $\varphi^{-1}(q)$ is a finite set and
\[
\degr {\varphi}{V}q = \sum_{p \in \varphi^{-1}(q) \cap V} \sign(\det(D \varphi_p)) \ .
\]
Here, as in \eqref{changeofvar}, we agree that $\degr {\varphi}{V}q = 0$ in case $\varphi^{-1}(q) \cap V$ is empty. Additionally, the function $q \mapsto \degr {\varphi}{V}q$ is locally constant on the domain of definition and is homotopy invariant in the sense that if $H : [0,1] \times \operatorname{cl}(V) \to \R^n$ is a continuous map and $\eta : [0,1] \to \R^n$ is a continuous path such that $\eta(t) \notin H_t(\partial V)$ for $0 \leq t \leq 1$, then $\degr{H_t}{V}{\eta(t)}$ is independent of $t$, see e.g.\ \cite[Chapter~IV, Proposition~2.4]{OR}. Further, if $\varphi, \psi : \operatorname{cl}(V) \to \R^n$ are two continuous extensions of a boundary map $\gamma : \partial V \to \R^n$ and $q \notin \gamma(\partial V)$, then
\[
\degr \varphi V q = \degr \psi V q \ ,
\]
see e.g.\ \cite[Chapter~IV, Proposition~2.6]{OR}. So, the degree is independent of the particular extension of $\gamma$.

\medskip

\noindent The following integrability result is a slight generalization of \cite[Proposition~4.6]{Zt}. For the sake of completeness we add a proof here. It sets the link between Brouwer degree functions and push forwards of currents. For Lipschitz maps this result is stated in \cite[Corollary~4.1.26]{F}.

\begin{lem}
\label{deg_lem}
Let $U \subset \R^n$ be a bounded open set and $\varphi : \R^n \to \R^n$ be a map such that $\varphi^i \in \Hol^{\alpha_i}(\partial U)$ for exponents $\alpha_1,\dots,\alpha_n \in \ ]0,1]$. If $\chi_U \in \BV^{d-(n-1)}(\R^n)$ for some $d \in [n-1,n[$ and $\tau_n \defl \alpha_1 + \dots + \alpha_n > d$, then
\begin{equation*}
(\varphi_{\#}\curr U) \res (\R^n \setminus \varphi(\partial U)) = \curr{\degr{\varphi}{U}{\cdot}} \ .
\end{equation*}
If additionally $\cL^{n}(\varphi(\partial U)) = 0$, then $\degr{\varphi}{U}{\cdot} \in L^1_c(\R^n)$ and
\[
\varphi_{\#}\curr U = \curr{\degr{\varphi}{U}{\cdot}} \ .
\]
If $\dim_{\boxd}(\partial U) < \tau_n$, then $\cL^{n}(\varphi(\partial U)) = 0$ and $\chi_U \in \BV_c^{\delta-(n-1)}(\R^n)$ for all $\delta \in \ ]\dim_{\boxd}(\partial U),\tau_n[$.
\end{lem}

\begin{proof}
First assume that $\varphi : \R^n \to \R^n$ is a smooth map. The density function $v \in L^1(\R^n,\Z)$ of $\varphi_\#\curr U$ is given by
\[
v(y)  = \sum_{x \in \varphi^{-1}(y)\cap U} \operatorname{sign}(\det(D\varphi_x))
\]
as stated in \eqref{changeofvar}. This agrees with $\degr \varphi {U} y$ for almost all $y \in \R^n \setminus \varphi(\partial U)$ by Sard's theorem. Thus $(\varphi_{\#}\curr U) \res (\R^n \setminus \varphi(\partial U)) = \curr{\degr \varphi {U} \cdot}$.

\medskip

\noindent For a given H\"older map $\varphi : \R^n \to \R^n$ as in the statement the current $\varphi_\# \curr U$ is well defined as a consequence of Theorem~\ref{dyadic_thm} and Theorem~\ref{changeofvar_thm}. Consider coordinatewise smoothings $\varphi_k^i \defl \rho_{1/k} \ast \varphi^i$, $k \in \N$, for some smooth approximation of the identity $\rho_\epsilon : \R^n \to \R$, $\epsilon > 0$. This approximating sequence has a uniform bound $\sup_{i,k}\Hol^{\alpha_i}(\varphi_k^i) < \infty$ on the H\"older constants and $\varphi_k$ converges to $\varphi$ uniformly on $U$. Let $v_k,v \in L^1_c(\R^n)$ be given by $\curr{v_k} = \varphi_{k\#}\curr{U}$ and $\curr{v} = \varphi_{\#}\curr{U}$. It follows from Theorem~\ref{changeofvar_thm} that $v_k$ converges to $v$ in $L^1$. Let $y \notin \varphi(\partial U)$ and $r > 0$ such that $\B(y,r) \subset \R^n\setminus \varphi(\partial U)$. Because $\varphi_k$ converges uniformly to $\varphi$, the ball $\B(y,r)$ does not intersect $\varphi_k(\partial U)$ for large enough $k$ and
\begin{equation}
\label{limitdegree}
\degr{\varphi_k}{U}z =\degr{\varphi}{U}z = \degr{\varphi}{U}y
\end{equation}
for all $z \in \B(y,r)$. For such an integer $k$, $\degr{\varphi_k}{U}{z} = v_k(z)$ for almost all $z \in \B(y,r)$ by the preparation for smooth maps above. By the constancy theorem for currents $\curr v \res \B(y,r) = n \curr{\B(y,r)}$ for some $n \in \R$ (note that $\spt(\partial \curr v) \subset \varphi(\partial U)$). With \eqref{limitdegree} the $L^1$-convergence of $v_k$ to $v$ implies
\[
0 = \lim_{k\to\infty} \int_{\B(y,r)}|v_k(z)-v(z)|\, dz = \cL^n(\B(y,r))(\degr{\varphi}{U}y - n) \ .
\]
Thus $n = \degr{\varphi}{U}z$ for almost all $z \in \B(y,r)$. By exhaustion $\curr{v} \res (\R^n\setminus \varphi(\partial U)) = \curr{\degr{\varphi}{U}\cdot}$.

\medskip

\noindent For the final part, fix some $\delta \in \ ]\dim_{\boxd}(\partial U) ,\tau_{n}[$. Then $\curr U \in \cF_{n,\delta}(\R^n)$ by Lemma~\ref{dimension_lem} and hence $\chi_U \in \BV_c^{\delta-(n-1)}(\R^n)$ by Theorem~\ref{equivalence_thm}. Since $\dim_\Haus(\partial U) \leq \dim_\boxd(\partial U)$ it holds that $\cH^\delta(\partial U) < \infty$ (it is equal zero actually). Hence there is a $C > 0$ such that for all small $\epsilon \in \ ]0,1]$ there is a countable cover $\bigcup_{i \in \N} A_i \supset \partial U$ with $d_i \defl \diam(A_i) \leq \epsilon$ and $\sum_{i \in \N} d_i^\delta \leq C$. Set $H \defl \max_i \Hol^{\alpha_i}(\varphi^i)$. Each image $\varphi(A_i)$ is contained in a box that is a translation of $[0,2Hd_i^{\alpha_1}] \times \dots \times [0,2Hd_i^{\alpha_n}]$. Hence,
\[
\cL^n(\varphi(\partial U)) \leq 2^nH^n \sum_{i \in \N} d_i^{\alpha_1}\cdots d_i^{\alpha_n} \leq 2^nH^n \epsilon^{\tau_n - \delta}\sum_{i \in \N} d_i^\delta \leq 2^nH^n C \epsilon^{\tau_n - \delta} \ .
\]
Since $\tau_n > \delta$ this converges to $0$ for $\epsilon \to 0$. Hence $\varphi(\partial U)$ is a set of Lebesgue measure zero.
\end{proof}

\subsection{Higher integrability of Brouwer degree functions}

\label{higherint_subsec}

Lemma~\ref{deg_lem} shows that the degree function for certain H\"older maps is integrable. Discussing higher integrability, we first treat the special case where the domain is a cube. Although the proof uses quite a bit of notation, the basic idea is rather simple. The original map is approximated by piecewise affine maps on successive simplicial decompositions of the cube. The $L^p$-norm of these approximations are easy to estimate because $\|\degr{A}{\Delta}{\cdot}\|_{L^p} = \cL^n(A(\Delta))^\frac{1}{p}$ in case $A:\R^n \to \R^n$ is affine and $\Delta \subset \R^n$ is a simplex.

\begin{lem}
\label{boxwindingintegral_lem}
Let $n \geq 1$, $Q \subset \R^n$ be a cube and $\varphi : Q \to \R^n$ be a map such that $\varphi^i \in \Hol^{\alpha_i}(Q)$ for $i=1,\dots,n$ and some exponents $\alpha_i \in \ ]0,1]$ that satisfy $\tau_n = \alpha_1 + \dots + \alpha_n > n-1$. Then the degree function $\degr{\varphi}{Q}{\cdot}$ is in $L_c^p(\R^n)$ in case $1 \leq p < \frac{\tau_n}{n-1}$ for $n > 1$ and $1 \leq p \leq \infty$ for $n=1$. Indeed (assuming $p < \infty$ in case $n=1$),
\[
\|\degr{\varphi}{Q}{\cdot}\|_{L^p} \leq C(n,\tau_n,p) \cL^n(Q)^\frac{\tau_n}{np} \Hol^{\alpha_1}(\varphi^1)^\frac{1}{p} \cdots \Hol^{\alpha_n}(\varphi^n)^\frac{1}{p} \ .
\]
\end{lem}

\begin{proof}
We abbreviate $H_i \defl \Hol^{\alpha_i}(\varphi^i)$ for $i=1,\dots,n$ and $H \defl H_1 \cdots H_n$. The statement for $n=1$ is trivial:
\[
\|\degr{\varphi}{Q}{\cdot}\|_{L^p} \leq \biggl(\int_{\varphi(Q)} 1^p \biggr)^\frac{1}{p} \leq \diam(\varphi(Q))^\frac{1}{p} \leq \diam(Q)^\frac{\alpha_1}{p}\Hol^{\alpha_1}(\varphi^1)^\frac{1}{p} \ .
\]
So we may assume that $n \geq 2$. We prove the lemma for the cube $Q = [0,1]^n$, and then a scaling argument will imply the statement for cubes of all volumes. For each integer $k \geq 0$ let $\cP_k \defl \{2^{-k}(p + Q) : p \in \Z^n\}$ be the dyadic decomposition of $\R^n$ to the scale $2^{-k}$. For any permutation $\sigma \in S_n$ there is an associated simplex $\{x \in Q : 0 \leq x_{\sigma(1)} \leq \dots \leq x_{\sigma(n)} \leq 1\}$ in $Q$. The collection $\cS_Q$ of these $n!$ simplices defines a simplicial complex with underlying set $Q$. If $k \geq 0$ and $R = 2^{-k} (p + Q) \in \cP_k$ set $\cS_R \defl \{2^{-k}(p + \Delta) : \Delta \in \cS_Q\}$ and $\cS_k \defl \bigcup_{R \in \cP_k, R \subset Q} \cS_R$. As before, $\cS_k$ defines a simplicial complex with underlying set $Q$. For any integer $k \geq 1$ and any of the $2n$ faces $F$ of $Q$ we let $\cS_{F,k}$ be the union over all $\cS_R$ where $R \in \cP_k$ is such that $F \cap R$ is a face of $R$ and $\operatorname{int}(R) \cap \operatorname{int}(Q) = \emptyset$. Note that the underlying set $F_k \defl \bigcup \cS_{F,k}$ is the set of points $\{x + t v \in \R^n : x \in F,\ t \in [0,2^{-k}]\}$, where $v$ is the outward unit normal to $F$, see Figure~\ref{figure1}.

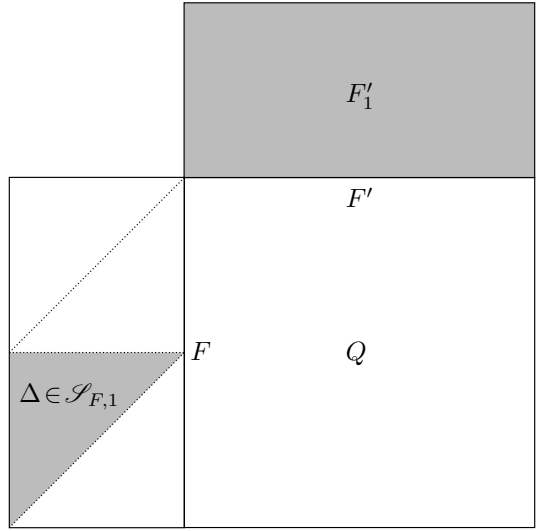
\begin{figure}[htb!]
	\centering
	\def\svgwidth{7cm}
	
\begingroup
  \makeatletter

  \providecommand\rotatebox[2]{#2}
  \ifx\svgwidth\undefined
    \setlength{\unitlength}{364.21959839bp}
    \ifx\svgscale\undefined
      \relax
    \else
      \setlength{\unitlength}{\unitlength * \real{\svgscale}}
    \fi
  \else
    \setlength{\unitlength}{\svgwidth}
  \fi
  \global\let\svgwidth\undefined
  \global\let\svgscale\undefined
  \makeatother
  \begin{picture}(1,1)
    \put(0,0){\includegraphics[width=\unitlength,page=1]{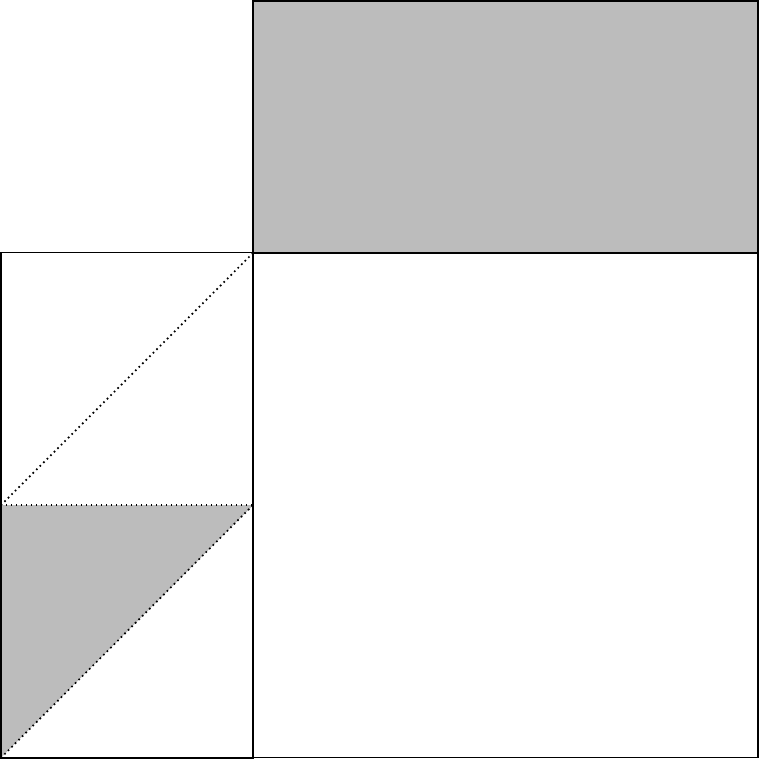}}
    \put(0.64,0.32){\color[rgb]{0,0,0}\makebox(0,0)[lb]{\smash{$Q$}}}
    \put(0.64,0.61){\color[rgb]{0,0,0}\makebox(0,0)[lb]{\smash{$F'$}}}
    \put(0.64,0.81){\color[rgb]{0,0,0}\makebox(0,0)[lb]{\smash{$F'_1$}}}
    \put(0.345,0.32){\color[rgb]{0,0,0}\makebox(0,0)[lb]{\smash{$F$}}}
    \put(0.02,0.24){\color[rgb]{0,0,0}\makebox(0,0)[lb]{\smash{$\Delta\! \in\! \mathscr S_{F,1}$}}}
  \end{picture}
\endgroup
	\caption{Depicted are some sets used in the construction of the piecewise affine approximation.}
	\label{figure1}
\end{figure}

\medskip

\noindent For $k \geq 0$ we define $\varphi_k : Q \to \R^n$ as the map with $\varphi_k(x) = \varphi(x)$ for any vertex $x$ of a simplex $\Delta \in \cS_k$ and for other points in this simplex $\varphi_k$ is the affine extension. If $x$ and $y$ are different vertices of $\Delta$, then $|x-y| \geq 2^{-k}$ and hence for all $i=1,\dots,n$,
\begin{align*}
\bigr|\varphi_k^i(x)-\varphi_k^i(y)\bigl| & = \bigr|\varphi^i(x)-\varphi^i(y)\bigl| \leq H_i|x-y|^{\alpha_i} \leq H_i |x-y|^{\alpha_i-1}|x-y| \\
 & \leq H_i 2^{k(1-\alpha_i)}|x-y| \ .
\end{align*}
This shows that $\Lip(\varphi_k^i|_\Delta) \leq C_1 H_i 2^{k(1-\alpha_i)}$ for any $k$,$i$ and $\Delta \in \cS_k$, where $C_1$ is some constant depending only on $n$. It is clear that $\varphi_k$ converges uniformly to $\varphi$. Similarly, for $k \geq 1$ we define $\gamma_{F,k} : F_k \to \R^n$ as follows: If $x$ is a vertex of some $\Delta \in \cS_{F,k}$ that is contained in $F$, then $\gamma_{F,k}(x) \defl \varphi_{k}(x)$. If $x$ is a vertex of some $\Delta \in \cS_{F,k}$ that is not in $F$, i.e.\ it has distance $2^{-k}$ from $F$, then $\gamma_{F,k}(x) \defl \varphi_{k-1}(\rho(x))$, where $\rho : F_k \to F$ is the orthogonal projection. On the remaining points of such a $\Delta$, $\gamma_{F,k}$ is the affine extension. As before, $\Lip(\gamma_{F,k}^i|_\Delta) \leq C_2 H_i 2^{k(1-\alpha_i)}$ for any $F$,$k$,$i$ and $\Delta \in \cS_{F,k}$, where $C_2 \geq C_1$ is some constant depending only on $n$. For all $k \geq 1$ and almost all $q \in \R^n$ it follows from the additivity of the Brouwer degree that
\begin{equation}
\label{face_eq}
 \degr{\varphi_{k-1}}{Q}{q} + \sum_{F \subset Q} \degr{\gamma_{F,k}}{F_k}{q} = \degr{\varphi_{k}}{Q}{q} \ .
\end{equation}
Since $\Delta \in \cS_{F,k}$ has volume $\frac{1}{n!}2^{-kn}$, it follows that $\gamma_{F,k}(\Delta)$ is a simplex with volume estimate
\begin{align*}
\cL^n(\gamma_{F,k}(\Delta)) & \leq \tfrac{1}{n!}2^{-kn}\Lip(\gamma_{F,k}^1|_\Delta)\cdots\Lip(\gamma_{F,k}^n|_\Delta) \\
 & \leq \tfrac{1}{n!}C_2^nH 2^{-kn}2^{k(1-\alpha_1)}\cdots 2^{k(1-\alpha_n)} \\
 & \leq C_3 H 2^{-k\tau_n} \ ,
\end{align*}
where $C_3 \defl \tfrac{1}{n!}C_2^n$. Since $(\gamma_{F,k}|_\Delta)^{-1}(q)$ consists of at most one point for almost every $q \in \R^n$, it holds that for all $p \in [1,\infty[$,
\[
\|\degr{\gamma_{F,k}}{\Delta}{\cdot}\|_{L^p} = \cL^n(\gamma_{F,k}(\Delta))^\frac{1}{p} \leq C_3^\frac{1}{p}H^\frac{1}{p} 2^{-k\frac{\tau_n}{p}} \ .
\]
The number of faces $F$ of $Q$ is $2n$, the corresponding set $F_k$ consists of $2^{(n-1)k}$ cubes in $\cP_k$ and each such cube is composed of $n!$ simplices. Hence
\begin{equation}
\label{faceestimate_eq}
\sum_{F \subset Q} \|\degr{\gamma_{F,k}}{F_k}{\cdot}\|_{L^p} \leq C_4 H^\frac{1}{p} 2^{k(n-1 - \frac{\tau_n}{p})} \ ,
\end{equation}
where $C_4 \defl n!2nC_3^\frac{1}{p}$. Similarly we obtain the estimate $\|\degr{\varphi_{0}}{Q}{\cdot}\|_{L^p} \leq C_4 H^\frac{1}{p}$. Assuming $1 \leq p < \frac{\tau_n}{n-1}$ it follows from \eqref{face_eq} and \eqref{faceestimate_eq} that
\begin{align*}
\|\degr{\varphi_{0}}{Q}{\cdot}\|_{L^p} & + \sum_{k \geq 1} \|\degr{\varphi_{k}}{Q}{\cdot} - \degr{\varphi_{k-1}}{Q}{\cdot}\|_{L^p} \\
 & \leq \|\degr{\varphi_{0}}{Q}{\cdot}\|_{L^p} + \sum_{k\geq 1} \sum_{F \subset Q} \|\degr{\gamma_{F,k}}{F_k}{\cdot}\|_{L^p} \\
 & \leq C_5 H^\frac{1}{p} \ ,
\end{align*}
for some constant $C_5 = C_5(n,p,\tau_n) \geq 0$. So, $(\degr{\varphi_{k}}{Q}{\cdot})_{k \in \N}$ is a Cauchy-sequence in $L^p(\R^n)$ and hence converges to some $u \in L^p(\R^n)$. Because $\varphi(\partial Q)$ is a set of measure zero by Lemma~\ref{deg_lem} and $(\varphi_k)_{k \in \N}$ converges uniformly to $\varphi$, the sequence $(\degr{\varphi_{k}}{Q}{\cdot})_{k \in \N}$ converges pointwise almost everywhere to $\degr{\varphi}{Q}{\cdot}$. Hence $u = \degr{\varphi}{Q}{\cdot} \in L^p(\R^n)$ with a norm estimate as in the statement.

\medskip

\noindent In the general situation for an arbitrary cube $Q \subset \R^n$ with side length $r = \cL^n(Q)^\frac{1}{n}$ let $\eta_r : [0,1]^n \to Q$ be the bi-Lipschitz map given by $\eta_r(x) \defl p + rx$ for some $p \in \R^n$. It holds $\degr{\varphi}{Q}{\cdot} = \degr{\varphi\circ \eta_r}{[0,1]^n}{\cdot}$ and it is simple to check that $\Hol^{\alpha_i}(\varphi^i\circ\eta_r) \leq \Hol^{\alpha_i}(\varphi^i)r^{\alpha_i}$ for all $i=1,\dots,n$. Thus
\[
\Hol^{\alpha_1}(\varphi^1\circ\eta_r)^\frac{1}{p}\cdots\Hol^{\alpha_n}(\varphi^n\circ\eta_r)^\frac{1}{p} \leq H^\frac{1}{p} r^\frac{\tau_n}{p} = H^\frac{1}{p} \cL^n(Q)^\frac{\tau_n}{np} \ .
\]
With the part above, the statement for arbitrary cubes $Q$ follows.
\end{proof}

\noindent Due to Theorem~\ref{dyadic_thm}, any function of bounded fractional variation can be approximated in a controlled way by sums over cubes. Thus we obtain an estimate of the $L^p$-norm for the push forward of currents induced by such functions.

\begin{thm}
\label{pushforward_thm}
Let $n \geq 1$, $d \in [n-1,n[$, $u \in \BV^{d-(n-1)}_c(\R^n)$ and $\varphi : \R^n\to\R^n$. Assume that $\alpha_1,\dots,\alpha_n \in \ ]0,1]$ and $r > 0$ are such that:
	\begin{enumerate}
		\item $\spt(u) \subset [-r,r]^n$.
		\item $\max_{i=1,\dots,n}\Hol^{\alpha_i}(\varphi^i) < \infty$.
		\item $\tau_n \defl \alpha_1 + \dots + \alpha_n > d$.
	\end{enumerate}
Then $\varphi_\# \curr{u} = \curr {v_{u,\varphi}}$ is defined for some $v_{u,\varphi} \in L^1_c(\R^n)$ with
\[
\|v_{u,\varphi}\|_{L^p} \leq C(n,\tau_n,d,p,r) \bV^{d-(n-1)}(u) \Hol^{\alpha_1}(\varphi^1)^\frac{1}{p} \cdots \Hol^{\alpha_n}(\varphi^n)^\frac{1}{p}
\]
for all $1 \leq p < \frac{\tau_n}{d}$ (or $1 \leq p < \infty$ if $d=n-1=0$). Further, if $(\varphi_k)_{k \in \N}$ is a sequence of maps that converges uniformly to $\varphi$ with $\sup_{i,k}\Hol^{\alpha_i}(\varphi_k^i) < \infty$, then $v_{u,\varphi_k}$ converges in $L^p$ to $v_{u,\varphi}$ for $p$ in the same range.

\medskip

\noindent Moreover, $v_{u,\varphi} \in \bigcap_{d'<\delta<n}\BV_c^{\delta-(n-1)}(\R^n)$ for $d' \defl n-1 + \frac{d-\tau_{n-1}}{\tau_{n}-\tau_{n-1}} = n + \frac{d-\tau_{n}}{\tau_{n}-\tau_{n-1}}$, where $\tau_{n-1} \defl \tau_n - \max_i \alpha_i$ (note that $d' = \frac{d}{\alpha}$ in case $\alpha = \alpha_1 = \cdots = \alpha_n$). If $F \in \Lip(\R^n)^n$ and $\beta_1,\dots,\beta_n \in \ ]0,1]$ satisfy $\beta \defl \beta_1 + \cdots + \beta_n > d'$, then there is a constant $C' = C'(n,d,r,\tau_n,\tau_{n-1},\beta) \geq 0$ such that
\begin{equation}
\label{bvpushbound_eq}
\left|\int_{\R^n} v_{u,\varphi}(y) \det DF_y \, dy \right| \leq C' \bV^{d-(n-1)}(u) h(\varphi)^{\beta + 1 - n}H_{n-1}(\varphi)H_n(F)  \ ,
\end{equation}
where $h(\varphi) \defl \min_i \Hol^{\alpha_i}(\varphi_i)$, $H_{n-1}(\varphi) \defl \max_{j}\prod_{i \neq j} \Hol^{\alpha_i}(\varphi^i)$,  and $H_{n}(F) \defl \prod_{i=1}^n\Hol^{\beta_i}(F^i)$.
\end{thm}

\begin{proof}
Abbreviate $H_{n}(\varphi) \defl \Hol^{\alpha_1}(\varphi^1)\cdots\Hol^{\alpha_n}(\varphi^n)$. From Theorem~\ref{dyadic_thm} it follows that there is an $L^1$ converging sum $u = \sum_{k \geq 0} u_k$ for $u_k \in \BV_c(\R^n)$ with $\spt(u_k) \subset [-r,r]^n$ and
\begin{equation}
\label{aprioribounds}
\|u_k\|_{L^1} \leq C_1\bV^{d-(n-1)}(u) 2^{k(d-n)} \ , \ \bV(u_k) \leq C_1\bV^{d-(n-1)}(u) 2^{k(d-(n-1))} \ ,
\end{equation}
for some constant $C_1 = C_1(n,d,r) \geq 0$. Indeed, $u_k = \sum_{R \in \cP_k} a_R \chi_R$, where $\cP_0 = \{[-r,r]^n\}$, $\cP_k = \{r2^{1-k}(p + [0,1]^n) : p \in \Z^n\}$ for $k \geq 1$ and $a_R \in \R$. From Lemma~\ref{deg_lem} and Lemma~\ref{boxwindingintegral_lem} it follows that $\varphi_\#(a_R \curr R) = a_R\curr{\degr{\varphi}{R}{\cdot}}$ for any $R \in \cP_k$ and
\begin{align}
\nonumber
\|a_R\degr{\varphi}{R}{\cdot}\|_{L^p} & \leq C_2 H_n(\varphi)^\frac{1}{p} |a_R|\cL^n(R)^\frac{\tau_n}{np} \\
\label{pushbound}
 & = C_2 H_n(\varphi)^\frac{1}{p} \|a_R \chi_R\|_{L^1} (r2^{1-k})^{\frac{\tau_n}{p}-n} \ ,
\end{align}
where $C_2 = C_2(n,\tau_n,p) \geq 0$. Let $w_k \in L^1_c(\R^n)$ be defined by $\curr{w_k} = \varphi_{\#}\curr{u_k}$. As a finite sum $w_k = \sum_{R \in \cP_k} a_R\degr{\varphi}{R}{\cdot}$ almost everywhere. Due to \eqref{pushbound} and \eqref{aprioribounds}
\begin{align*}
\|w_k\|_{L^p} & \leq \sum_{R \in \cP_k} \|a_R\degr{\varphi}{R}{\cdot}\|_{L^p} \\
 & \leq C_2 H_n(\varphi)^\frac{1}{p} (r2^{1-k})^{\frac{\tau_n}{p}-n} \sum_{R \in \cP_k} \|a_R \chi_R\|_{L^1} \\
 & = C_2 H_n(\varphi)^\frac{1}{p} (r2^{1-k})^{\frac{\tau_n}{p}-n}\|u_k\|_{L^1} \\
 & \leq C_1C_2(2r)^{\frac{\tau_n}{p}-n} \bV^{d-(n-1)}(u) H_n(\varphi)^\frac{1}{p} 2^{k(d-\frac{\tau_n}{p})} \ .
\end{align*}
Thus the partial sums of $\sum w_k$ converge in $L^p$ and almost everywhere to some $w \in L^p_c(\R^n)$ for $1 \leq p < \frac{\tau_n}{d}$ with
\begin{equation}
\label{lpboundeq}
\|w\|_{L^p} \leq C(n,\tau_n,d,p,r) \bV^{d-(n-1)}(u) H_n(\varphi)^\frac{1}{p} \ .
\end{equation}
Since $\sum w_k$ converges in $L^1$ to $w$, the partial sums of $\sum \curr{w_k}$ converge in mass to $\curr w$ and in particular weakly as currents. It follows that the boundaries $\sum \varphi_\#\partial\curr{u_k} = \sum \partial\curr{w_k}$ converge weakly to $\partial \curr w$. Thus $\varphi_\# \partial \curr u = \partial \curr w$ by the definition of this push forward in Proposition~\ref{push_prop2}. From Theorem~\ref{changeofvar_thm} it follows that $\varphi_\# \curr u = \curr{v_{u,\varphi}}$ is well defined by approximation for some $v_{u,\varphi} \in L^1_c(\R^n)$ and together with Theorem~\ref{equivalence_thm} we obtain that $v_{u,\varphi} \in \BV^{\delta-n+1}_c(\R^n)$ for all those $\delta$ as in the statement. Since $\partial \curr{v_{u,\varphi}} = \varphi_\# \partial \curr u = \partial \curr w$ by Theorem~\ref{changeofvar_thm}, the constancy theorem for currents implies that $w = v_{u,\varphi}$ almost everywhere.

\medskip

\noindent If $(\varphi_k)_{k\geq 0}$ is a sequence of maps on $[-r,r]^n$ that converges uniformly to $\varphi$ with $\sup_{i,k}\Hol^{\alpha_i}(\varphi^i_k) < \infty$, then $v_{u,\varphi_k}$ converges in $L^1$ to $v_{u,\varphi}$ by Theorem~\ref{changeofvar_thm}. Let $1 \leq p < \frac{\tau_n}{d}$. Fix some $\gamma_i \in \ ]0,\alpha_i[$ such that $\gamma \defl \gamma_1 + \cdots + \gamma_n > dp$. Then $L \defl \sup_{i,k}\Hol^{\gamma_i}(\varphi^i_k) < \infty$ by \eqref{inclusionholder} and $\lim_{k\to\infty} \Hol^{\gamma_i}(\varphi^i-\varphi_k^i) = 0$ by \eqref{holdertrick_eq}. Define the maps $F^i_k \defl (\varphi^1, \dots, \varphi^{i-1}, \varphi^i-\varphi_k^i, \varphi_k^{i+1},\dots,\varphi_k^n)$ from $[-r,r]^n$ to $\R^n$. Then $\varphi_\# \curr u - \varphi_{k\#} \curr u = \sum_{i=1}^n F^i_{k\#}\curr u$ and by \eqref{lpboundeq}
\begin{align*}
\|v_{u,\varphi} - v_{u,\varphi_k}\|_{L^p} & \leq \sum_{i=1}^n \bigl\|v_{u,F^i_k}\bigr\|_{L^p} \leq C(n,\gamma,d,p,r) \bV^{d-(n-1)}(u) \sum_{i=1}^n H_n(F^i_k)^\frac{1}{p} \\
 & \leq C(n,\gamma,d,p,r) \bV^{d-(n-1)}(u) L^\frac{n-1}{p}\sup_{i=1,\dots,n}\Hol^{\gamma_i}(\varphi^i_k-\varphi^i_k)^\frac{1}{p} \\
 & \to 0
\end{align*}
for $k \to \infty$. 

\medskip

\noindent It remains to show the bound in \eqref{bvpushbound_eq}. From the decomposition of $u$ in \eqref{aprioribounds} it is a consequence of Theorem~\ref{changeofvar_thm} (and the comment after its proof) that there are $v_k \in \BV_c(\R^n)$ such that $\sum_{k \geq 0} \curr{v_k} = \curr{v_{u,\varphi}}$ and 
\begin{align}
\label{wkestimate_eq}
\|v_k\|_{L^1} & \leq C_3\bV^{d-(n-1)}(u) H_n(\varphi)\eta^{k(d'- n)} \ , \\
\label{wkestimate_eq2}
\bV(v_k) & \leq C_3\bV^{d-(n-1)}(u) H_{n-1}(\varphi)\eta^{k(d'-(n-1))} \ ,
\end{align}
where $C_3 = C_3(n,d,r) \geq 0$ is a constant, $\eta = 2^{\tau_{n}-\tau_{n-1}} > 1$, $d' = n-1 + \frac{d-\tau_{n-1}}{\tau_{n}-\tau_{n-1}} = n + \frac{d-\tau_{n}}{\tau_{n}-\tau_{n-1}}$ and $\tau_{n-1} = \tau_n - \max_i \alpha_i$ as in the statement of the theorem. Without loss of generality we assume that $0 < h(\varphi) = \Hol^{\alpha_1}(\varphi^1) \leq \cdots \leq \Hol^{\alpha_n}(\varphi^n)$. If $h(\varphi) = 0$, then some $\varphi^i$ is constant and thus \eqref{bvpushbound_eq} is obvious because the left hand side vanishes. Let $A : \R^n \to \R^n$ be the linear map $A(x) = h(\varphi)x$. With \eqref{changeofvar} it is clear that $T \defl (A^{-1})_\# \curr{v_{u,\varphi}} = \curr{v_{u,\varphi}\circ A}$ and $R_k \defl (A^{-1})_\# \curr{v_k} = \curr{v_k \circ A}$. With \eqref{wkestimate_eq},
\begin{align*}
\bM(R_k) & \leq h(\varphi)^{-n} \bM(\curr{v_k}) \leq C_3\bV^{d-(n-1)}(u) h(\varphi)^{-n}H_{n}(\varphi)\eta^{k(d'- n)} \\
 & = C_3\bV^{d-(n-1)}(u) h(\varphi)^{1-n}H_{n-1}(\varphi)\eta^{k(d'- n)} \ ,
\end{align*}
and with \eqref{wkestimate_eq2}
\begin{align*}
\bM(\partial R_k) & \leq h(\varphi)^{1-n} \bM(\partial\curr{v_k})\\
 & \leq C_3\bV^{d-(n-1)}(u) h(\varphi)^{1-n} H_{n-1}(\varphi)\eta^{k(d'-(n-1))} \ .
\end{align*}
From Proposition~\ref{reverse_prop} it follows that $T = \sum_{k \geq 0} R_k \in \bF_{n,\beta}(\R^n)$ for $\beta = \beta_1 + \cdots + \beta_n > d'$ and with Theorem~\ref{fractalflat_thm} (where $f = (1,F\circ A)$, $\alpha_1 = 1$, $\alpha_i = \beta_{i-1}$ for $i>1$, $\delta = \beta$, $\gamma = 1 + \beta$ and $\rho = \eta$),
\begin{align*}
|\partial &\curr{v_{u,\varphi}}(F)| = |\partial T(F \circ A)| = |T(1,F \circ A)| \\
 & \leq C_4\bV^{d-(n-1)}(u) h(\varphi)^{1-n}H_{n-1}(\varphi) \Hol^{\beta_1}(F^1 \circ A) \cdots \Hol^{\beta_n}(F^n \circ A) \\
 & = C_4\bV^{d-(n-1)}(u) h(\varphi)^{1-n}H_{n-1}(\varphi) \Hol^{\beta_1}(F^1)h^{\beta_1} \cdots \Hol^{\beta_n}(F^n)h^{\beta_n} \\
 & = C_4\bV^{d-(n-1)}(u) h(\varphi)^{\beta+1-n}H_{n-1}(\varphi) \Hol^{\beta_1}(F^1) \cdots \Hol^{\beta_n}(F^n)
\end{align*}
for a constant $C_4 = C_4(n,d,r,\tau_n,\tau_{n-1},\beta) \geq 0$ and $F \in \Lip(\R^n)^n$ as in the statement. This proves the theorem.
\end{proof}

\noindent As a direct consequence we obtain the following result about degree functions that generalizes \cite[Proposition~2.4]{Z2}, \cite[Theorem~1.1, Theorem~1.2(i)]{O} and \cite[Theorem~2.1]{DI}. It also proves a conjecture stated in \cite{DI} about the higher integrability of the Brouwer degree function for a map with coordinate functions of variable H\"older regularity. This is a restatement of Theorem~\ref{degreeint_corintro} in the introduction.

\begin{thm}
\label{degreeint_thm}
Let $U \subset \R^n$ be a bounded open set such that $\partial U$ has box counting dimension $d \in [n-1,n[$. Assume $\varphi : \R^n\to\R^n$ satisfies $\max_{i}\Hol^{\alpha_i}(\varphi^i) < \infty$ for some $\alpha_1,\dots,\alpha_n \in \ ]0,1]$ with $\tau_n \defl \alpha_1 + \dots + \alpha_n > d$. Then
\[
\|\degr \varphi U \cdot \|_{L^p} \leq C(U,n,\tau_n,p) \Hol^{\alpha_1}(\varphi^1)^\frac{1}{p}\cdots\Hol^{\alpha_n}(\varphi^n)^\frac{1}{p}
\]
for all $1 \leq p < \frac{\tau_n}{d}$ (or $1 \leq p < \infty$ if $d=n-1=0$). Further, if $(\varphi_k)_{k \in \N}$ is a sequence of maps that converges uniformly to $\varphi$ with $\sup_{i,k}\Hol^{\alpha_i}(\varphi_k^i) < \infty$, then $\degr {\varphi_k} U \cdot$ converges in $L^p$ to $\degr \varphi U \cdot$ for $p$ in the same range.

\medskip

\noindent Moreover, $\degr \varphi U \cdot \in \bigcap_{d'<\delta<n}\BV_c^{\delta-(n-1)}(\R^n)$ for $d' \defl n-1 + \frac{d-\tau_{n-1}}{\tau_{n}-\tau_{n-1}} = n + \frac{d-\tau_{n}}{\tau_{n}-\tau_{n-1}}$, where $\tau_{n-1} \defl \tau_n - \max_i \alpha_i$ (note that $d' = \frac{d}{\alpha}$ in case $\alpha = \alpha_1 = \cdots = \alpha_n$). If $F \in \Lip(\R^n)^n$ and $\beta_1,\dots,\beta_n \in \ ]0,1]$ satisfy $\beta \defl \beta_1 + \cdots + \beta_n > d'$, then
\begin{equation*}
\left|\int_{\R^n} \degr \varphi U y \det DF_y \, dy \right| \leq C'(U,n,\tau_n,\tau_{n-1},\beta) h(\varphi)^{\beta + 1 - n}H_{n-1}(\varphi)H_n(F)  \ ,
\end{equation*}
where $h(\varphi) \defl \min_i \Hol^{\alpha_i}(\varphi_i)$, $H_{n-1}(\varphi) \defl \max_{j}\prod_{i \neq j} \Hol^{\alpha_i}(\varphi^i)$,  and $H_{n}(F) \defl \prod_{i=1}^n\Hol^{\beta_i}(F^i)$.
\end{thm}

\begin{proof}
This is a direct consequence of Corollary~\ref{holderfractal_cor}, Lemma~\ref{deg_lem} and Theorem~\ref{pushforward_thm}.
\end{proof}

\noindent In order to see that the integral estimate in the theorem above generalizes the second part of \cite[Theorem~2.1]{DI} assume that $\varphi \in \Hol^\alpha(\R^n,\R^n)$ for some $\alpha \in \ ]\frac{d}{n},1]$. If $\gamma = \beta_1 > d'-(n-1)$, $\beta_2 = \cdots = \beta_n=1$, $f \in \Lip(\R^n)$ and $i \in \{1,\dots,n\}$, then 
\begin{align*}
\left|\int_{\R^n} \degr \varphi U y \partial_i f(y) \, dy \right| & = \left|\int_{\R^n} \degr \varphi U y \det D(f,\pi^1,\dots,\hat \pi^i,\dots,\pi^n)_y \, dy \right| \\
 & \leq C'(U,n,\alpha,\gamma)h(\varphi)^{\beta - (n-1)}H_{n-1}(\varphi)\Hol^\gamma(f) \\
 & \leq C'(U,n,\alpha,\gamma)\Hol^{\alpha}(\varphi)^{\beta}\Hol^\gamma(f) \ .
\end{align*}
Note that $\beta = \gamma + n - 1$ and the condition $(\gamma+n-1)\alpha > d$ in \cite[Theorem~2.1]{DI} is precisely $\gamma > d'-(n-1)$ because $d' = \frac{d}{\alpha}$.

\medskip

\noindent In this situation where all the exponents $\alpha_i$ are identical, it is shown in \cite[Theorem~1.2(ii)]{O} and \cite[Theorem~1.3]{DI} that the integrability range for $p$ in Theorem~\ref{degreeint_thm} is best possible (except possibly for the critical exponent).

\medskip

\noindent Although the condition on $U$ in Theorem~\ref{degreeint_thm} is given in terms of the box counting dimension $d$ of $\partial U$, we could have made the more general assumption $\chi_U \in \bigcap_{d < \delta < n}\BV_c^{\delta - (n-1)}(\R^n)$ and $\cL^n(\varphi(\partial U))=0$. The first assumption also holds for domains that satisfy the condition used in \cite[Theorem~A,B]{HN} or \cite[Theorem~2.2]{Gus} as discussed after Lemma~\ref{dimension_lem}.

%%%%%%%%%%%%%%%%%%%%%%%%%%%%%%%%%%%%%%%%%%%%%%%%%%%%%%%%%%%%%%%%%%%%%%%%%%%%%%%%%%%%%%%%%%%%%%%%%%%%%%%%%%%%%%%%


\begin{thebibliography}{99}

\bibitem{AFP}
L. Ambrosio, N. Fusco, D. Pallara.
Functions of bounded variation and free discontinuity problems.
Oxford Mathematical Monographs, The Clarendon Press, Oxford University Press, 2000.

\bibitem{AK}
L. Ambrosio, B. Kirchheim.
Currents in metric spaces.
Acta Math. 185 (2000), 1--80.

\bibitem{DI}
C. De Lellis, D. Inauen.
Fractional Sobolev regularity for the Brouwer degree.
Comm. Partial Differential Equations 24 (2017), 1510--1523.

\bibitem{DP}
T. De Pauw.
Approximation by polyhedral $G$ chains in Banach spaces.
Z. Anal. Anwend. 33 (2014), 311--334.

\bibitem{PH}
T. De Pauw, R. Hardt.
Rectifiable and flat $G$ chains in metric spaces.
Amer. J. Math. 134 (2012), 1--69.

\bibitem{DR}
G. De Philippis, F. Rindler.
On the structure of $\mathscr A$-free measures and applications.
Ann. of Math. 184 (2016), 1017--1039.

\bibitem{F}
H. Federer.
Geometric measure theory.
Die Grundlehren der mathematischen Wissenschaften, Band 153, Springer, 1969.

\bibitem{G2}
M. Gromov.
Metric structures for Riemannian and non-Riemannian spaces (with appendices by M. Katz, P. Pansu, and S. Semmes).
vol. 152 of Prog. Math., Birkh\"auser, 1999.

\bibitem{Gus}
Y. Guseynov.
Integrable boundaries and fractals for {H}\"older classes; the Gauss–Green theorem.
Calc. Var. and PDE 55 (2016), 1--27.

\bibitem{HN}
J. Harrison, A. Norton.
The Gauss-Green theorem for fractal boundaries.
Duke Math. J. 67 (1992), 575--588.

\bibitem{L}
U. Lang.
Local currents in metric spaces.
J. Geom. Anal. 21 (2011), 683--742.

\bibitem{M}
P. Mattila.
Geometry of sets and measures in Euclidean spaces.
Cambridge Stud. Adv. Math. 44, Cambridge Univ. Press, 1995.

\bibitem{O}
H. Olbermann.
Integrability of the Brouwer degree for irregular arguments.
Ann. Inst. H. Poincar\'e Anal. Non Lin\'eaire 34 (2016), 933--959.

\bibitem{OR}
E. Outerelo, J. M. Ruiz.
Mapping degree theory.
Graduate Studies in Mathematics, vol. 108, American Mathematical Society, 2009.

\bibitem{RY}
T. Rivi\`ere, D. Ye.
Resolutions of the prescribed volume form equation.
Nonlinear Differential Equations Appl. 3 (1996), 323--369.

\bibitem{S}
E. M. Stein.
Singular integrals and differentiability properties of functions.
Princeton Univ. Press, Princeton (1970)

\bibitem{W}
S. Wenger.
Isoperimetric inequalities of Euclidean type and applications.
phd thesis, ETH Z\"urich, 2004.

\bibitem{Whi}
H. Whitney.
Geometric Integration Theory.
Princeton University Press, 1957.

\bibitem{Y}
L. C. Young.
An inequality of the {H}\"older type, connected with Stieltjes integration.
Acta Math. 67 (1936), 251--282.

\bibitem{Zt}
R. Z\"ust.
Currents in snowflaked metric spaces.
phd thesis, ETH Z\"urich, 2011.

\bibitem{Z}
R. Z\"ust,
Integration of {H}\"older forms and currents in snowflake spaces,
Calc. Var. and PDE 40 (2011), 99--124.

\bibitem{Z2}
R. Z\"ust.
A solution of Gromov's H\"older equivalence problem for the Heisenberg group.
arXiv:1601.00956 (2016).

\end{thebibliography}
\end{document}